\documentclass[12pt,reqno]{amsart}
\usepackage[headings]{fullpage}
\usepackage{amssymb,amsmath,mathtools,bbm,tikz,tikz-cd,color,relsize}
\usepackage[all,cmtip]{xy}
\usepackage{url}
\usetikzlibrary{positioning}
\usetikzlibrary{calc}
\usetikzlibrary{decorations.pathreplacing}
\usetikzlibrary{arrows}
\usetikzlibrary{calc}
\usetikzlibrary{decorations.markings}
\tikzstyle{hvector}=[inner sep=2pt,draw=blue!50,fill=blue!10,thick]
\tikzstyle{unit}=[inner sep=2pt,shape=circle, draw]
\tikzstyle{counit}=[inner sep=2pt,shape=circle, draw,fill=gray]
\tikzstyle{antipode}=[inner sep=2pt,shape=rectangle, draw]
\tikzstyle{fantipode}=[inner sep=2pt,shape=rectangle, draw,fill,white]
\tikzstyle{cocycle}=[inner sep=2pt,shape=circle, draw]
\tikzstyle{twistedm}=[inner sep=2pt,shape=circle, fill=gray]
\tikzstyle{autom}=[inner sep=2pt,shape=circle, draw]
\tikzstyle{coact}=[inner sep=2pt,shape=circle, fill=black]
\usepackage[T1]{fontenc}
\usepackage{lmodern}
\usepackage{mathrsfs}
\newcommand{\Fsurg}{\mathscr{F}}
\newcommand{\figref}[1]{Fig.~\protect\ref{#1}}

\newcommand{\qp}[2]{(#1;#2)_\infty}
\numberwithin{equation}{section}
\newcommand{\CT}{\operatorname{CT}}
\newcommand{\ba}{\begin{aligned}}
\newcommand{\ea}{\end{aligned}}
\newcommand{\bea}{\begin{align}}
\newcommand{\eaa}{\end{align}}
\newcommand{\re}{\mathrm{e}}
\newcommand{\ri}{\mathrm{i}}
\newcommand{\rd}{\mathop{}\!\mathrm{d}}
\newcommand{\had}{\odot}
\newcommand{\me}[1]{\mu_{#1}}
\newcommand{\rme}[1]{\hat{\mu}_{#1}}
\def\IZ{{\mathbb Z}}
\def\IR{{\mathbb R}}
\def\IC{{\mathbb C}}
\def\tq{{\tilde{q}}}
\def\tx{{\tilde{x}}}

\newcommand{\CA}{{\mathcal A}}
\newcommand{\CB}{{\mathcal B}}
\newcommand{\CC}{{\mathcal C}}
\newcommand{\CJ}{{\mathcal J}}
\newcommand{\CF}{{\mathcal F}}
\newcommand{\CO}{{\mathcal O}}
\newcommand{\CP}{{\mathcal P}}

\newcommand{\CM}{{\mathcal M}}
\newcommand{\tX}{{\widetilde X}}

\newcommand{\mL}{{\mathsf{L}}}
\newcommand{\mH}{\mathrm{H}}
\newcommand{\ii}{\mathrm i}

\newcommand{\qpoch}[2]{(#1;#2)_\infty}
\newcommand{\qpochn}[3]{(#1;#2)_{#3}}
\newcommand{\fn}[1]{g_{#1}}

\newcommand{\gammaq}[1]{\gamma_q\!\left(#1\right)}

\tikzset{
  on each segment/.style={
    decorate,
    decoration={
      show path construction,
      moveto code={},
      lineto code={
        \path [#1]
        (\tikzinputsegmentfirst) -- (\tikzinputsegmentlast);
      },
      curveto code={
        \path [#1] (\tikzinputsegmentfirst)
        .. controls
        (\tikzinputsegmentsupporta) and (\tikzinputsegmentsupportb)
        ..
        (\tikzinputsegmentlast);
      },
      closepath code={
        \path [#1]
        (\tikzinputsegmentfirst) -- (\tikzinputsegmentlast);
      },
    },
  },
  mid arrow/.style={postaction={decorate,decoration={
        markings,
        mark=at position .5 with {\arrow[#1]{To}}
      }}},
  mid end arrow/.style={postaction={decorate,decoration={
        markings,
        mark=at position .7 with {\arrow[#1]{To}}
      }}},
  end arrow/.style={postaction={decorate,decoration={
        markings,
        mark=at position .95 with {\arrow[#1]{To}}
      }}},
}

\usepackage[bookmarks=true,%
    colorlinks=true,%
    linkcolor=blue,%
    citecolor=blue,%
    filecolor=blue,%
    menucolor=blue,%
    urlcolor=blue,%
    breaklinks=true]{hyperref}
\usepackage{slashed}
\usepackage{verbatim}
\usepackage[normalem]{ulem}  %
\usepackage{diagbox}

\newcommand{\sectiono}[1]{\section{#1}\setcounter{equation}{0}}

\usepackage{amssymb,amsmath,tikz}

\newcommand{\Z}{\mathbb{Z}}
\newcommand{\CZ}{\mathcal{Z}}
\providecommand{\C}{}
\renewcommand{\C}{\mathbb{C}}
\newcommand{\R}{\mathbb{R}}
\newcommand{\Q}{\mathbb{Q}}

\newcommand{\mg}[2]{\gamma_{#1}\!\left(#2\right)}

\newcommand{\phioneone}{{}_1\phi_1}

\DeclareMathOperator{\dd}{d\!}
\newcommand{\qbinom}[2]{\genfrac{[}{]}{0pt}{}{#1}{#2}_{q}}

\newtheorem{theorem}{Theorem}
\newtheorem{proposition}{Proposition}
\newtheorem{lemma}{Lemma}
\newtheorem{corollary}{Corollary}
\theoremstyle{remark}
\newtheorem{definition}{Definition}
\newtheorem{remark}{Remark}
\newtheorem{conjecture}{Conjecture}

\newtheorem{example}{Example}
\def\be{\begin{equation}}
\def\ee{\end{equation}}
\newcommand{\ps}[1]{\mathcal{P}_{#1}}
\newcommand{\res}[1]{\operatorname{Res}_{#1}}
\newcommand{\fps}[1]{\mathbb{C}[[{#1}]]}

\newcommand{\smf}{\mathcal{SM}}

\title{Analytic quantum invariants}
\author{Rinat Kashaev}
\address{Section de Mathématiques, Université de Genève,
rue du Conseil-G\'en\'eral 7-9, 1205 Gen\`eve, Suisse\\}
\email{Rinat.Kashaev@unige.ch}
\author{Vladimir Mangazeev}
\address{Research School of Physics, Australian National University, Canberra, ACT 2601, Australia\\}
\email{Vladimir.Mangazeev@anu.edu.au}
\author{Marcos Mariño}
\address{Section de Mathématiques et Département de Physique Théorique, Université de Genève,
1211 Genève 4, Switzerland}
\email{Marcos.Marino@unige.ch}
\date{Started: September 6, 2025; updated: \today}

\begin{document}

\begin{abstract}
Using a deformed Borel--Laplace resummation, we propose a framework for
repackaging quantum invariants as analytic invariants. We illustrate our
approach through a detailed analysis of the coloured Jones polynomials
of several simple knots, deriving explicit expressions for the
resulting analytic invariants in terms of $(q,X)$-series. These
invariants are closely related to, but generally distinct from, those
obtained via the conventional Borel resummation of the
Melvin--Morton--Rozansky expansion.

As a further example, we construct an analytic invariant from a sequence
of two-variable polynomials conjectured to be the Links--Gould
polynomials of the figure-eight knot, associated with the quantum
superalgebras $U_q(\mathfrak{gl}(n|1))$. Finally, we apply our technique
to $q$-series arising in the study of three-manifold invariants,
revealing connections with mock modular forms.
\end{abstract}

\maketitle

{\footnotesize
\tableofcontents
}

\section{Introduction}

The coloured Jones polynomials $J_{K,n}(q)$ of a knot $K\subset \R^3$
form a fundamental family of quantum knot invariants, arising from the
$n$-dimensional irreducible representations of the quantum group $U_q(\mathfrak{sl}_2)$~\cite{MR1026957,MR1091619}.
 Beyond their algebraic definition, these invariants exhibit striking analytic, geometric, and asymptotic features~\cite{MR1328734,MR1389962,MR1434238,rozansky, MR1828373, gukov,MR2349300,Garoufalidis2004,MR2860990}.

One problem that has received considerable attention in recent years is how to promote these polynomials to analytic functions of the variables \(q\) and \(x\), where \(x\) agrees with the colour at positive integer values. One possible approach is to start from the Melvin--Morton--Rozansky (MMR) series associated with the knot~\cite{MR1328734,rozansky}. This series in powers of \(h=-\log q\) is typically factorially divergent  with coefficients that are functions of \(X=q^x\). It was conjectured in~\cite{MR4233201} that the MMR series can be reorganized into a formal two-variable series \(F_K(q,X)\), which, for many knots, admits a compact representation in terms of the ``inverted Habiro series'' introduced in~\cite{park-inverted}. In many examples, the latter is given by a convergent expression for \(|q|<1\); however, it has poles at \(X=q^n\) and therefore does not yield an analytic interpolation of the coloured Jones polynomials.

A more direct analytic approach is to use Borel resummation techniques to convert the MMR series into an analytic function in angular sectors of the complex \(h\)-plane~\cite{MR4841756}. In this setting, the resulting function typically involves both \(q\)-series and ``dual'' \(\tilde q\)-series, where $\tilde q=\re^{-4\pi^2/h}$. Another approach was proposed in~\cite{MR4404799}, where a unified two-variable invariant was constructed that specializes to the coloured Jones polynomials for suitable choices of the variables. However, this specialization is algebraic, taking place in a specific completion of a polynomial ring, and does not necessarily correspond to a limit of analytic functions. Finally, the work~\cite{MR2809462} addresses the problem of analytic continuation from the perspective of the path integral formulation of complex Chern--Simons theory.

In this work, we use a specific deformed variant of Borel--Laplace resummation~\cite{Borel1899,MR1317343,MR1722871,MR2474083} to define an analytic knot invariant from a formal expression closely related to the coloured Jones polynomials. This expression, involving an interpolation to non-integral values of the colour parameter, is the result of an approach initiated by K.~Habiro in~\cite{MR2002603} through his cyclotomic expansion of the coloured Jones polynomials, subsequently studied in~\cite{MR4404799,park-inverted}. This expansion reads
\begin{equation}\label{eq:cJH}
J_{K,n}(q)=\sum_{m=0}^{n-1}
\frac{(q^{1+n},q^{1-n};q)_m}{\mg{q}{-m}} H_{K,m}(q),
\end{equation}
where we use the $q$-hypergeometric notation of~\cite{MR2128719},
\begin{equation}\label{eq:mg-notation}
\mg{q}{n}:=(-1)^n q^{n(n-1)/2},
\end{equation}
and where $H_{K,m}(q)$ are the Habiro polynomials (also known as cyclotomic coefficients; Habiro uses the notation $a_m(K)$ for $H_{K,m}(q)$ in~\cite{MR2002603}). The expression (\ref{eq:cJH}) suggests the \emph{Jones interpolating series}
\begin{equation}\label{eq:cJHi}
J_{K,x}(q,z)=\sum_{m=0}^\infty \frac{(q^{1+x},q^{1-x};q)_m}{\mg{q}{-m}}
 H_{K,m}(q)z^m,
\end{equation}
which, when $x=n\in\Z_{>0}$, truncates at $m=n-1$  and hence recovers~\eqref{eq:cJH} upon setting $z=1$. The role of the variable $z$ in~\eqref{eq:cJHi} is to turn the expression into a power series, although the specialisations $z=q^\ell$ were already considered in~\cite{MR4700600} in the definition of the descendants of the coloured Jones polynomials.

To produce our \emph{analytic Jones invariant}, denoted by $\hat{J}_{K,x}(q,z)$, we apply a deformed variant of
Borel--Laplace resummation to the series~\eqref{eq:cJHi}  (for generic $x$) whenever it is divergent. For example, for the figure-eight knot $K={\bf 4}_1$ one has $H_{{\bf 4}_1,m}(q)=1$ for all $m\ge0$, see~\cite[(9)]{MR2002603}, and hence series~\eqref{eq:cJHi} diverges. The value of the analytic Jones invariant at $z=1$,
\be
\label{atz1}
 \hat{J}_{K,x}(q):=\hat{J}_{K,x}(q,1), 
\ee
can be regarded as the \emph{interpolating Jones invariant}, reproducing the coloured Jones polynomials at positive integer values of $x$. For simplicity, we restrict our attention to only real values of \(q\) satisfying \(0<q<1\). In this case, the analytic Jones invariant $\hat{J}_{K,x}(q,z)$, viewed as a function of complex variable $z$, has a logarithmic branch cut; therefore (\ref{atz1}) can depend on the choice of branch, similarly to the dependence on the choice of lateral resummation in conventional Borel resummation. 

The deformed Borel--Laplace resummation may itself produce a divergent expression. In such cases, one may try to iterate the resummation procedure until a finite result is obtained. We illustrate this phenomenon for the knot \({\bf 5}_2^-\) (of negative signature), for which we apply the deformed Borel--Laplace resummation twice. 

In the literature on linear $q$-difference equations, various versions of the $q$-Borel and $q$-Laplace transforms are used; see, for example,~\cite{MR3384489} and the references therein, as well as~\cite{gar-wheel} for an application to knot theory. The transforms we use here are different. In particular, our $q$-Laplace transform, defined below in~\eqref{eq:q-laplace}, is not a sum but a genuine (Lebesgue) integral over the positive real half-line, just as in the case of the usual Laplace transform. As a result, the analytic picture in the Borel plane is very similar to that in the undeformed setting. Moreover, as we will see in Remark~\ref{rem:deformedBL} of Section~\ref{sec:deformedBL}, our \(q\)-deformed Borel--Laplace resummation approaches the standard Borel resummation as \(q\) approaches \(1\).

\subsection{$(q,X)$-series from the interpolating Jones invariant} 
In this paper, we show that the interpolating Jones invariant is
well-defined for $0<q<1$ and is analytic in $x$ in suitable
domains. In the examples considered here, it can be expressed
as a finite linear combination of explicit $(q,X)$-series
(that is, series in $q$ whose coefficients are Laurent
polynomials in $X$), multiplied by powers of $(\log q)^{-1}$.
By construction, it is an even function of $x$ and
specializes to the $|n|$-coloured Jones polynomial at $x=n\in\Z_{\ne0}$.
Moreover, it is also regular at $x=0$.

The series~ \eqref{eq:cJHi} determines a formal
$z$-deformation of the MMR series; setting $z=1$ gives the usual MMR series. In the examples studied in the paper, we
show that, with $q=\re^{-h}$ and $x=-w/h$, the analytic Jones invariant $\hat{J}_{K,x}(q,z)$ admits the corresponding deformed MMR series as
its fixed-$w$ formal asymptotic expansion as $h \to0^+$. We do not provide the remainder estimates needed to establish a Poincaré asymptotic expansion uniform in $w$ on a neighborhood of the origin. However, our numerical evaluations for $w$ near $0$ are consistent with such uniform asymptotic behavior.

It is natural to ask whether the analytic Jones invariant agrees with the conventional Borel resummation of the MMR series, or with the invariants obtained in \cite{MR4233201,park-inverted}. Although the interpolating Jones invariant seems to be closely related to these objects, it is distinct from both of them. In the examples considered here, the interpolating Jones invariant and conventional Borel resummation have the same formal MMR expansion but differ by exponentially small terms in $h=-\log q$. In physics parlance, they define different non-perturbative completions of the perturbative MMR series. They also lead to different mathematical objects, since as we noted above, the conventional Borel resummation involves in general functions depending as well on ``dual'' variables $\tilde q=\re^{-4\pi^2/h}$, $\tilde X=\re^{2\pi\ri x}$ \cite{MR4841756}. 

Our analytic invariant differs in nature from the ones in \cite{MR4233201,park-inverted} as well. As mentioned above, outside the cases where a convergent construction is available, the proposal of \cite{MR4233201} specifies only formal two-variable series \(F_K(q,X)\). When convergent, the inverted Habiro series of~\cite{park-inverted} provides a meromorphic realization, or regularization~\cite{svoboda}, of \(F_K(q,X)\), but it does not satisfy two of the basic properties of our invariant: it is singular when $x=n\in \IZ_{\ge 0}$, and it does not asymptote to the MMR series when $w$ is near $0$. The inverted Habiro series nevertheless appears naturally as one contribution to our answer, while additional terms enforce these two properties.

Our procedure leads to explicit formulae for the analytic Jones invariant in many cases. A particularly important example is the figure-eight knot ${\bf 4}_1$, which we use to illustrate the procedure in Section~\ref{sec:41}. A more general formalism is developed in Section~\ref{sec:gen-for}. In particular, we propose a resummed version of the divergent basic hypergeometric function ${_r\phi_s}$ with $r>s+1$ (Theorem~\ref{thm:r-fi-s}). The general framework allows us to work with a simpler generating function of the Habiro polynomials, as explained in Section~\ref{sec:gen-hab-ser}. As shown in Proposition~\ref{Jxb}, the analytic Jones invariant can then be obtained from the resummed generating function of the Habiro polynomials in a universal form, in the same way that the coloured Jones polynomials can be obtained from the Habiro polynomials by using a universal kernel.

In Section~\ref{sec:2variable}, we apply the deformed Borel--Laplace resummation to a sequence
of two-variable polynomials conjectured to be the Links--Gould
polynomials of the figure-eight knot, associated with the quantum
superalgebras $U_q(\mathfrak{gl}(n|1))$.

Besides quantum knot polynomials, our procedure can also be applied to other $q$-series and sequences of polynomials arising in quantum topology. In Section~\ref{sec:wrt}, we consider $q$-series that arise naturally in the study of the WRT invariants of three-manifolds. By using deformed Borel--Laplace resummation, we obtain analytic functions of $q$ whose formal asymptotic expansions are given by the Ohtsuki series of three-manifolds \cite{ohtsuki}. These functions are also closely related to the $q$-series appearing in \cite{lz,hikami1, hikami2}, as well as to the so-called $\hat Z$-invariants of three-manifolds \cite{gpv,gppv}. As in the case of the analytic Jones invariant, the formulae we obtain include in general corrections to the $\hat Z$-invariants obtained in the literature. These corrections are needed in order to have the correct Ohtsuki asymptotics. At present, these analytic functions are attached to surgery presentations; invariance under changes of presentation has not been addressed yet.

\subsection{Open problems} In this paper we propose a preliminary exploration of the techniques of deformed Borel--Laplace resummation mainly in the context of knot theory, but many open problems remain. The most important problem is to determine whether the method works universally for all knots. Different knots require different types of $p$-Borel--Laplace resummation, where $p=q^{\beta}$, and it would be important to determine what the required value of $\beta$ is for a given knot. In the case of the ${\bf 5}_2^-$ knot, an additional resummation is needed. A precise characterization of the knots for which an iterative method is needed would be welcome. On a different note, one should develop a better understanding of the singularity structure of the deformed Borel--Laplace transform and determine whether interesting information can be extracted from its singularities, as in the usual Borel resummation. The extension of the procedure to more general values of $q$ should also be considered.

It would also be interesting to understand better the situation in the case of three-manifolds, and see whether the approach proposed here could lead to genuine $q$-series invariants. The analytic functions that we have defined for surgery presentations could be easily studied for general Brieskorn homology spheres. Our preliminary results in Section \ref{sec:wrt} indicate that the procedure of deformed Borel--Laplace resummation might shed some light on the problem of finding pairs of ``dual'' mock modular forms for three-manifolds related by orientation reversal \cite{3dmod}.

Finally, there are many $q$-series arising in mathematical physics that could be studied with the techniques developed in this paper. Particularly interesting examples are the topological string free energies of Calabi--Yau threefolds and their quantum mirror maps in the large radius limit. These are series in the exponentiated K\"ahler parameters whose coefficients are functions of $q$. When $0<q<1$, these series are typically divergent \cite{hmmo,Marino2015,akp}, and their deformed Borel--Laplace resummation may lead to new results on non-perturbative topological string theory.

\sectiono{Deformed Borel--Laplace resummation}\label{sec:deformedBL}
Let
\begin{equation}
f(z)=\sum_{n=0}^\infty a_n z^n
\end{equation}
 be a divergent series such that
 \begin{equation}
a_n=\mathcal{O}(r^{-n}q^{-n(n+1)/2}),\quad r\in\R_{>0}.
\end{equation}
 Then, the \emph{$q$-Borel transform}
\begin{equation}
\tilde f(z)=(\mathcal{B}_q f)(z):=\sum_{n=0}^\infty \frac{(-1)^n}{(q^{-1};q^{-1})_n} a_nz^n=\sum_{n=0}^\infty \frac{q^{n(n+1)/2}}{(q;q)_n} a_nz^n
\end{equation}
converges in the disk $|z|<r$.  Assuming that it admits an analytic continuation along the ray through \(z\), and that the integral below converges, we define the \emph{\(q\)-Borel--Laplace resummation} of \(f(z)\) by
\begin{equation}\label{eq:hat-transform}
\hat f(z)=\frac{-1}{\log q}(\mathcal{L}_q\tilde f)(z),
\end{equation}
where $\mathcal{L}_q$ is the $q$-Laplace transform defined by
\begin{equation}\label{eq:q-laplace}
(\mathcal{L}_qh)(z):=\int_0^\infty \frac{h(zt)}{(-t;q)_\infty}\rd  t. 
\end{equation}
Formula~\eqref{eq:hat-transform} is based on the following integral formula:
\begin{equation}
\frac{-1}{\log q}\int_0^\infty \frac{t^n}{(-t;q)_\infty}\rd t=(-1)^{n} (q^{-1};q^{-1})_n,\quad  n\in\Z_{\ge0},
\end{equation}
which can be deduced from Ramanujan’s $q$-beta integral
\begin{equation}\label{eq:ramanujan-beta}
\int_0^\infty t^{s-1}\frac{(-at;q)_\infty}{(-t;q)_\infty}\rd t=\frac{(a,q^{1-s};q)_\infty}{(q,aq^{-s};q)_\infty}\frac{\pi}{\sin(\pi s)},\quad \Re s>0,\ |aq^{-s}|<1,
\end{equation}
 see~\cite[(2.8), (2.9)]{MR567718}.

\begin{remark}\label{rem:deformedBL}
If we consider the slightly modified versions of the $q$-Borel and $q$-Laplace transformations given by
\be
(\mathcal{B}'_q f)(z)=(\mathcal{B}_q f)(z(1-q)),\quad
(\mathcal{L}'_q f)(z)=\frac{-1}{\log q}(\mathcal{L}_q f)(z/(1-q)),
\ee
then one readily sees that, in the limit $q\to 1$,
\be
\mathcal{B}'_1:=\lim_{q\to 1}\mathcal{B}'_q=\mathcal{B},\quad
\mathcal{L}'_1:=\lim_{q\to 1}\mathcal{L}'_q=\mathcal{L},
\ee
where $\mathcal{B}$ and $\mathcal{L}$ denote the usual Borel and Laplace transformations, respectively. In particular, our $q$-Borel--Laplace resummation approaches the usual Borel--Laplace resummation as $q$ approaches $1$, as follows from the identity
\be
\frac{-1}{\log q}\mathcal{L}_q\mathcal{B}_q
=
\mathcal{L}'_q\mathcal{B}'_q.
\ee
This justifies the use of the term ``deformed Borel--Laplace resummation''. \hfill$\diamond$
\end{remark}

If the $q$-Borel transform $\tilde f(z)$ is singular at a point $z_\omega \in \IC$, the ray $\re^{\ri \theta}\IR_{\ge 0}$, with $\theta={\rm arg}(z_\omega)$, will be called a {\it Stokes ray}, in analogy with the usual Borel resummation. The $q$-Laplace transform (\ref{eq:q-laplace}) is singular whenever $z$ lies on a Stokes ray. Thus, the $q$-Borel--Laplace resummation defines an analytic function on wedges, and its discontinuities can be determined from the behavior of the $q$-Borel transform near its singularities. In this paper, we will not calculate these discontinuities explicitly. Rather, whenever possible, we will calculate the deformed Borel--Laplace resummation for generic values of $z$, with its discontinuities encoded in the branch cuts of the resulting expression. In the case of the Jones interpolation series~\eqref{eq:cJHi}, we will often set $z=1$ after resummation, and in some cases, such as the trefoil knot, this specialization requires a choice of branch in the resulting formula.

The deformed Borel--Laplace resummation procedure has some advantages over the standard Borel resummation. The coefficients of the MMR series are usually not known in closed form, which typically precludes obtaining closed-form expressions for the conventional Borel resummation (although, in some cases, conjectural closed-form expressions can be obtained with the help of state integrals; see \cite{ggm1,peacock,MR4841756}). In contrast, the Habiro polynomials appearing in the series~(\ref{eq:cJHi}) are known to all orders for many knots~\cite{MR1997328}.

We now define the analytic Jones invariant of a knot \(K\) as follows.
\begin{definition}\label{def:analytic-inv}
Suppose that the coefficients of the interpolating series~\eqref{eq:cJHi} of the colored Jones polynomials of a knot \(K\) grow as \(q^{-\beta m^2/2}\). We then define the {\it analytic Jones invariant of \(K\)}, denoted by \(\hat J_{K,x}(q,z)\), to be the \(p\)-Borel--Laplace resummation of \(J_{K,x}(q,z)\), where \(p=q^\beta\), provided this resummation exists.\hfill$\diamond$
\end{definition}
We make a few remarks concerning this definition.
\begin{enumerate}
\item The value of \(\beta\) depends on the knot. If (\ref{eq:cJHi}) has a positive radius of convergence, we set \(\beta=0\). In this case, we simply define
\be
\hat J_{K,x}(q,z)=J_{K,x}(q,z).
\ee
\item The \(p\)-Borel--Laplace resummation may not exist. In such cases, further resummations may be performed to obtain a finite result. This will be illustrated below for the knot \({\bf 5}_2^-\).
\item The analytic Jones invariant of a knot can be constructed from a simpler analytic quantum invariant based on the Habiro polynomials. This will be explained and developed in Section~\ref{sec:gen-hab-ser}.
\end{enumerate}

\sectiono{The figure-eight knot}
\label{sec:41}
Our first example is the figure-eight knot $K={\bf 4}_1$. We will carry out a detailed analysis of this example and, on the basis of it, develop in the next sections a general procedure for arbitrary knots.

 \subsection{Deformed Borel--Laplace resummation and contour integral representation}\label{sec:4_1-int-sub-sec1}

In the case of the figure-eight knot, the Habiro polynomials are simply
\be
\label{hab-41}
H_{{\bf 4}_1,m}(q)=1, \quad m\ge0,
\ee
so that
\begin{equation}\label{eq:4_1n}
J_{{\bf 4}_1,n}(q)=\sum_{m=0}^{n-1}
\frac{(q^{1+n},q^{1-n};q)_m}{\mg{q}{-m}}=J_{{\bf 4}_1,n}(q,1),
\end{equation}
where
\begin{equation}\label{eq:4_1z}
J_{{\bf 4}_1,x}(q,z)=\sum_{m=0}^{\infty} \frac{(q^{1+x},q^{1-x};q)_m}{\mg{q}{-m}}
z^m
\end{equation}
is the Jones interpolating series, which is divergent unless $x\in \Z_{\ne0}$. Its coefficients grow like $q^{-m^2/2}$; therefore, in Definition~\ref{def:analytic-inv}, $\beta=1$, and the analytic Jones invariant is defined by $q$-Borel--Laplace resummation.

Applying first the $q$-Borel transform, we obtain
\begin{equation}\label{tildefz}
\tilde J_{{\bf 4}_1,x}(q,z):=(\mathcal{B}_qJ_{{\bf 4}_1,x}(q,\cdot))(z)=\sum_{m=0}^{\infty}
\frac{(q^{1+x},q^{1-x};q)_m}{(q;q)_m}(-z)^m=
{_2\phi_1}\Bigl[
\begin{matrix}
 q^{1+x},q^{1-x}\\
0
\end{matrix}
;q,-z\Bigr].
\end{equation}
Thus, the $q$-Borel--Laplace resummation of~\eqref{eq:4_1z} gives
\begin{equation}\label{eq:interpolated-jones}
\hat J_{{\bf 4}_1,x}(q,z):=\frac{-1}{\log q}(\mathcal{L}_q \tilde J_{{\bf 4}_1,x}(q,\cdot))(z)=\frac{-1}{\log q}\int_0^\infty \frac{\tilde J_{{\bf 4}_1,x}(q,tz)}{(-t;q)_\infty}\rd  t.
\end{equation}
In order to make sense of this formula, we need an analytic continuation of the function $\tilde J_{{\bf 4}_1,x}(q,w)$ along the rays
\(
w\in \re^{\ri\theta}\IR_{>0},
\)
\(
\theta={\rm arg}(z).
\)

A first analytic continuation can be obtained by  specialising $c= 0$ in Jackson's transformation formula (see~\cite[(1.5.4)]{MR2128719})
\begin{equation}
{_2\phi_1}\Bigl[
\begin{matrix}
 a,b\\
c
\end{matrix}
;q,z\Bigr]=\frac{(az;q)_\infty}{(z;q)_\infty}
{_2\phi_2}\Bigl[
\begin{matrix}
 a,c/b\\
c,az
\end{matrix}
;q,bz\Bigr],
\end{equation}
with the result
\begin{equation}
\label{bt-41}
\tilde J_{{\bf 4}_1,x}(q,z)=\frac{(-q^{1+x}z;q)_\infty}{(-z;q)_\infty}
{_1\phi_1}\Bigl[
\begin{matrix}
 q^{1+x}\\
-q^{1+x}z
\end{matrix}
;q,-q^{1-x}z\Bigr].
\end{equation}

A more useful analytic continuation for our purposes follows from
Watson's contour integral representation for the basic hypergeometric series, which states that \cite[(4.2.2)]{MR2128719}
\begin{equation}\label{eq:watson-ac}
{_2\phi_1}\Bigl[
\begin{matrix}
 a,b\\
c
\end{matrix}
;q,z\Bigr]=\frac{(a,b;q)_\infty}{(c,q;q)_\infty}\left(\frac{-1}{2\pi \ri}\right)\int_C\frac{(q^{1+s},cq^s;q)_\infty}{(aq^s,bq^s;q)_\infty}\frac{\pi(-z)^s}{\sin(\pi s)}\rd s,
\end{equation}
where the contour $C$ runs upward along the imaginary axis, deviating in the finite part to separate the poles at $s\in\Z_{\ge 0}$ from all other poles of the integrand. The condition for the convergence of the integral is $|\arg(-z)|<\pi$. Specialising to our case, we obtain
\begin{align}
\tilde J_{{\bf 4}_1,x}(q,z)&=\frac{(q^{1+x},q^{1-x};q)_\infty}{(q;q)_\infty}\left(\frac{-1}{2\pi \ri}\right)\int_C\frac{(q^{1+s};q)_\infty}{(q^{1+s+x},q^{1+s-x};q)_\infty}\frac{\pi z^s}{\sin(\pi s)}\rd s\nonumber\\
&=\frac{(q^{1+x},q^{1-x};q)_\infty}{(q;q)_\infty}\left(\frac{1}{2\pi \ri}\right)\int_{\R-\ri\epsilon}\frac{(q^{\ri s};q)_\infty}{(q^{\ri s+x},q^{\ri s-x};q)_\infty}\frac{\pi z^{\ri s-1}}{\sinh(\pi s)}\rd s,\label{eq:watson-integral}
\end{align}
where the last equality is valid if $0<\epsilon<1$ and $-\epsilon<x<\epsilon$. Substituting~\eqref{eq:watson-integral} into~\eqref{eq:interpolated-jones} and exchanging the order of integrations, we write
\begin{equation}
\hat J_{{\bf 4}_1,x}(q,z)=\frac1{2\ri\log q}\frac{(q^{1+x},q^{1-x};q)_\infty}{(q;q)_\infty}\int_C\frac{(q^{1+s};q)_\infty}{(q^{1+s+x},q^{1+s-x};q)_\infty}\frac{R_s(q)z^s}{\sin(\pi s)}\rd s,
\end{equation}
where
\begin{equation}\label{eq:t-s-int}
R_s(q):=\int_0^\infty\frac{ t^{s}}{(-t;q)_\infty}\rd t=-\frac{(q^{-s};q)_\infty}{(q;q)_\infty}\frac{\pi}{\sin(\pi s)}=\frac{1}{(q;q)_{-s-1}}\frac{\pi}{\sin(-\pi s)}.
\end{equation}
The last equality is a special case of~\eqref{eq:ramanujan-beta}. Thus,
\begin{align}
\hat J_{{\bf 4}_1,x}(q,z)&=\frac{-1}{2\ri\log q}\frac{(q^{1+x},q^{1-x};q)_\infty}{(q;q)_\infty^3}\int_C\frac{\theta_q(q^{-s})}{(q^{1+s+x},q^{1+s-x};q)_\infty}\frac{\pi z^{s}}{(\sin(\pi s))^2}\rd s\nonumber\\
&=\frac{1}{2\log q}\frac{(q^{1+x},q^{1-x};q)_\infty}{(q;q)_\infty^3}\int_{\R-\ri\epsilon}\frac{\theta_q(q^{\ri s})}{(q^{\ri s+x},q^{\ri s-x};q)_\infty}\frac{\pi z^{\ri s-1}}{(\sinh(\pi s))^2}\rd s,
\label{eq:41cont-int}
\end{align}
where the second equality is valid if $0<\epsilon<1$ and $-\epsilon<x<\epsilon$, and we use the theta-function
\begin{equation}\label{eq:theta-q1}
\theta_q(u):=\sum_{n\in\Z}\gamma_q(n)u^n=(q,u,q/u;q)_\infty,
\end{equation}
where the second equality is Jacobi's triple product identity~\cite[(1.6.1)]{MR2128719}.
\begin{remark}
One can show that the following linear functional equations are satisfied
\begin{align}
\frac{1 - qX}{1- qX^2}\hat J_{{\bf 4}_1,x+1}(q,z)&=\Bigl[\frac{1 - qX}{1- qX^2}  -(1 + X^{-1})z\Bigr] \hat J_{{\bf 4}_1,x}(q,z)\nonumber\\
 &+ \frac{z}{qX}\hat J_{{\bf 4}_1,x}(q,z/q) +  zq\hat J_{{\bf 4}_1,x}(q,zq)
\label{eq:gxz-rec-xz}
\end{align}
and 
\begin{equation}\label{eq:gxz-rec-z}
(z^{-1}-X-X^{-1})\hat J_{{\bf 4}_1,x}(q,z)+q\hat J_{{\bf 4}_1,x}(q,zq)+q^{-1}\hat J_{{\bf 4}_1,x}(q,z/q)=z^{-1},
\end{equation}
which imply the difference equation 
 \begin{equation}\label{eq:inhomogeneous-eq-x}
\frac{X
   V_{x-1}}{1-X^2/q}+\frac{XV_{x+1}}{ 1-X^2q}+\bigl[\frac{1-X^{4}}{\left(1-X^2/q\right) \left(1-X^2q\right)}+\left(X-X^{-1}\right)z\bigr] V_x=1+X,
   \end{equation}
    where $X=q^x$ and
   \begin{equation}\label{eq:def-G}
V_x:=(1-X)\hat J_{{\bf 4}_1,x}(q,z).
\end{equation}

Equation~\eqref{eq:inhomogeneous-eq-x} at $z=1$ is related to the $AJ$-conjecture of~\cite{Garoufalidis2004}. 
  Indeed, it states that $1+X$ is the image of $V_x$ acted upon  by the linear operator
 \begin{equation}\label{eq:lin-op}
D(L,M,z,q)=  \frac{M}{1-q^{-1}M^2}L^{-1}+\frac{M}{ 1-qM^2}L+\frac{1-M^4}{\left(1-q^{-1}M^2\right) \left(1-qM^2\right)}+\big(M-\frac1M\big)z,
\end{equation}
where  $M$ and $L$ are invertible operators acting in the space of functions $f_x$ as follows:
\begin{equation}
(Mf)_x=Xf_x,\quad (Lf)_x=f_{x+1},
\end{equation}
so that they satisfy the Weyl commutation relation
\begin{equation}
LM=qML.
\end{equation}
Thus, in the limit $q\to1$, operators $M$ and $L$ commute, and the operator~\eqref{eq:lin-op} becomes
\begin{equation}
D(L,M,z,1)
=\frac{ L+\frac{1}{L}+M+\frac{1}{M}-\left(M-\frac{1}{M}\right)^2
   z}{\frac{1}{M}-M},
\end{equation}
where the numerator at $z=1$ is the $A$-polynomial of the figure-eight knot.
\hfill$\diamond$
\end{remark}

\subsection{Formal asymptotic expansion} \label{prop-as} Let us introduce the variables
 \be
 \label{uu}
 w=-xh,\quad h=-\log q,
 \ee
 and
 \be
 \label{Xqx}
X=q^x=\re^{w}.
 \ee
The Jones interpolating series~\eqref{eq:4_1z} leads to a formal series in $h$ as follows. Write it as 
 \be
 \label{def-41-habiro}
 J_{{\bf 4}_1,x}(q,z)=\sum_{k \ge 0} z^{k} \prod_{j=1}^{k} \bigl(\re^w+ \re^{-w}-\re^{jh} -\re^{-jh}\bigr),
 \ee
expand the summand in $h$, and then, by exchanging the order of summations, sum over $k \ge 0$. In this way one obtains a formal power series $\Phi^{{\bf 4}_1}_0 \left( h; w ,z \right)$. One finds, at the very first orders in $h$,
 \be
\Phi^{{\bf 4}_1}_0 \left( h; w ,z \right)=\frac{1}{1-\left(X+X^{-1}-2\right) z} -\frac{X^3 z \left(X^2 z-(2 z-1)X+z\right)}{\left(X^2 z-X (2 z+1)+z\right)^4} h^2+ {\mathcal O}(h^4). 
\ee
 For $z=1$, the series $\Phi^{{\bf 4}_1}_0(h;w):= \Phi^{{\bf 4}_1}_0(h;w, 1)$ is the conventional MMR series. When $z \not=1$, $\Phi^{{\bf 4}_1}_0(h;w, z)$ is a deformation of the MMR series by the parameter $z$.
 
 A natural question is whether the interpolating Jones invariant $\hat J_{{\bf 4}_1,x}(q,1)$ gives the MMR series in the limit $h\to 0^+$, where $q=\re^{-h}$. We will now show that
\be
\label{J41-as}
\hat J_{{\bf 4}_1,x}(q,z)\sim \Phi^{{\bf 4}_1}_0 \left( h; w, z \right)
\ee
as a formal asymptotic expansion. The derivation is based on the standard asymptotic analysis of functions defined by
Mellin--Barnes integrals, see \cite{PK}. 

We first note the asymptotic expansion for the $q$-Pochhammer symbol when $h\to 0$ \cite{mcintosh}
\be
\label{qpoch-z}
\log\left(\re^{-h z}; \re^{-h} \right)_\infty \sim  -{\pi^2 \over 6 h}-\left(z-{1\over 2}\right) \log h + \log {\sqrt{2 \pi} \over \Gamma(z)} +{h \over 4} B_2(z)-\sum_{n \ge 1}{ B_{2n} B_{2 n+1}(z) \over (2n) (2n+1)! } h^{2 n},
\ee
where $B_n(x)$ are Bernoulli polynomials. We use (\ref{qpoch-z}) in (\ref{eq:41cont-int}) to obtain
\be
\ba
 \hat J_{ {\bf{4}}_1, x} (q, z) \sim -{\ri \over 2 h} \int_{\IR- \ri \epsilon} &  {\Gamma(\ri s +x) \Gamma(\ri s -x) \over  \Gamma(1+x) \Gamma(1-x)}
{ \re^{ (2 \ri s -1) \log h + (\ri s-1) \log z } \over \sinh(\pi s)} \\
& \times  \exp\left[ \sum_{n \ge 1} {B_{2n} \over 2n (2n+1)!} P_n(x, \ri s) h^{2n}\right] \rd s,
\ea
\ee
where
\be
P_n(x,y)=  B_{2n+1}(x+y) + B_{2n+1}(y -x) - B_{2n+1}(1+x) - B_{2n+1}(1 -x)= \sum_{\ell=0}^n p_{n, \ell }(y) x^{2\ell}.
\ee
The polynomials $p_{n,\ell}(y)$ can be evaluated as
\be
\ba
p_{n,n-\ell}(y)= {2 (2n+1)! \over (2n-2\ell)!(2\ell+1)!} \left( B_{2\ell+1}(y)-B_{2\ell+1}(1)\right).
\ea
\ee
Noting that
\be
B_{2\ell+1}(1)={1\over 2}\delta_{\ell,0},
\ee
we find
\be
\ba
\sum_{n \ge 1} {B_{2n} \over 2n (2n+1)!} P_n(x,y) h^{2n}
&=2 (y-1)  \log\left( { \sinh(w/2) \over w/2} \right)\\
& + 2 \sum_{\ell \ge 1} {B_{2 \ell +1}(y) \over (2\ell+1)!} h^{2\ell} \left( {\rd \over \rd w} \right)^{2 \ell} \log\left( { \sinh(w/2) \over w/2} \right),
\ea
\ee
with $w$ defined in \eqref{uu},
 and we used the equality
 \be
  \sum_{n \ge 1} {B_{2n} \over 2n (2n)!} w^{2n}=  \log\left( { \sinh(w/2) \over w/2} \right).
  \ee
The asymptotic expansion can be evaluated by picking the residues in the lower half-plane,
at $s=-\ri k$, $k \ge1$, with the result
\be
\ba
 \hat J_{ {\bf{4}}_1, x} (q, z)\sim &\sum_{k \ge 1} z^{k-1} ( 2 \sinh(w/2))^{2(k-1)}\prod_{j=1}^{k-1} \left(1-{ h^2 \over w^2} j^2 \right)  \\ & \times
 \exp \left\{ 2 \sum_{\ell \ge 1} {B_{2 \ell +1}(k) \over (2\ell+1)!} h^{2\ell} \left( {\rd \over \rd w} \right)^{2 \ell} \log\left( { \sinh(w/2) \over w/2} \right) \right\}.
\ea
\ee
By using the relations
\be
 \left( {\rd \over \rd w} \right)^{2 \ell}\log w= -{(2\ell-1)! \over w^{2\ell}}, \quad \exp \left\{ - \sum_{\ell \ge 1} {B_{2 \ell +1}(k) \over (2\ell+1)\ell} h^{2\ell}  \right\}= \prod_{j=1}^{k-1}\left(1- h^2 j^2 \right),
 \ee
we further simplify
 \be
 \label{phi0-41}
\hat J_{ {\bf{4}}_1, x} (q, z)\sim \sum_{k \ge 1} z^{k-1} \left( 2 \sinh(w/2) \right)^{2(k-1)}
 \exp \left[ 2 \sum_{\ell \ge 1} {B_{2 \ell +1}(k) \over (2\ell+1)!} h^{2\ell} \left( {\rd \over \rd w} \right)^{2 \ell} \log\left( 2 \sinh\left({w\over 2}\right) \right) \right], 
 \ee
which is $ \Phi^{{\bf 4}_1}_0 \left( h; w,z \right)$ as defined by asymptotic expansion of~\eqref{def-41-habiro}, following the procedure explained right after it.

The preceding Mellin--Barnes analysis gives the formal asymptotic series but does
not provide a proof of (\ref{J41-as}) in the usual sense of Poincar\'e asymptotics. However, our numerical calculations are consistent with (\ref{J41-as}) being the asymptotic behavior of the analytic Jones invariant, uniformly in $w$, for $w$ near $0$. 

\subsection{Evaluation of the contour integral I}\label{subsection:evaluation-cont-int-4-1}
\begin{theorem}\label{thm:4_1-full} Let $n\in\Z_{\ge0}$ and $x\in\R$ be such that
\begin{equation}\label{eq:n-dependent-cond}
n+1>|x|.
\end{equation}
Assume also that $z\in\mathbb C\setminus (q^{\Z}\cup \R_{\le0})$. Then
\begin{align}\label{eq:hat-J-x-z-q-1-fi-1}
&\hat{J}_{{\bf 4}_1,x}(q,z)=\sum_{m=0}^{n-1}\frac{(q^{1+x},q^{1-x};q)_{m}}{\mg{q}{-m}}z^m\nonumber\\
+&\frac{z^n(q^{1+x},q^{1-x};q)_\infty }{(q;q)_\infty^3}\sum_{m\in\Z}\frac{\left(\frac{\log z}{\log q} +m\right)q^{mn}\mg{q}{-m}({q^{1-m-x}};q)_\infty}{1-zq^{m}}{_1\phi_1}\Bigl[
\begin{matrix}
 0\\
 q^{1-m-x}
\end{matrix}
;q,q^{1-m+x}\Bigr].
\end{align}
\end{theorem}
A proof of this theorem will be given later on, at the end of Section~\ref{sec:resummation-r-fi-s}, after developing a general formalism for working with deformed Borel--Laplace resummations.

Taking the limit $z\to1$ in~\eqref{eq:hat-J-x-z-q-1-fi-1}, we obtain the interpolating Jones  invariant
\begin{equation}\label{eq:hat-f-x-1}
\hat J_{{\bf 4}_1,x}(q)=\lim_{z\to 1}\hat J_{{\bf 4}_1,x}(q,z)=A_{{\bf 4}_1,x}(q)-\frac{1}{\log q}B_{{\bf 4}_1,x}(q),
\end{equation}
where
\begin{align}
&A_{{\bf 4}_1,x}(q):=\sum_{k=0}^{n-1}\frac{(q^{1+x},q^{1-x};q)_{k}}{\mg{q}{-k}}
\nonumber\\
+&\frac{(q^{1+x},q^{1-x};q)_\infty }{(q;q)_\infty^3}\sum_{m\in\Z_{\ne0}}\frac{mq^{mn}\mg{q}{-m}({q^{1-m-x}};q)_\infty}{1-q^m}
{_1\phi_1}\Bigl[
\begin{matrix}
 0\\
 q^{1-m-x}
\end{matrix}
;q,q^{1-m+x}\Bigr]
\label{eq:G-0-q}
\end{align}
and
\begin{equation}\label{eq:H-q}
B_{{\bf 4}_1,x}(q):=\frac{(q^{1+x};q)_\infty(q^{1-x};q)_\infty^2 }{(q;q)_\infty^3}
{_1\phi_1}\Bigl[
\begin{matrix}
 0\\
 q^{1-x}
\end{matrix}
;q,q^{1+x}\Bigr],
\end{equation}
where condition ~\eqref{eq:n-dependent-cond} is assumed. By using the $q$-Bessel function
\be
       {\mathcal {J}}(X,Y;q) =(qY;q)_\infty{_1\phi_1}\Bigl[
\begin{matrix}
 0\\
 qY
\end{matrix}
;q,qX\Bigr],
  \end{equation}
which is symmetric in its arguments: $\CJ(X,Y;q)= \CJ(Y, X;q)$, we can write (\ref{eq:H-q}) as
\be
\label{b41-ex}
B_{{\bf 4}_1,x}(q)= \frac{ (q^{1+x},q^{1-x};q)_\infty }{(q;q)_\infty^3}
\CJ (q^{x},q^{-x};q).
\ee
For $n=0$, the series (\ref{eq:G-0-q}) reads
\be
\begin{aligned}
A_{{\bf 4}_1,x}(q)={}&\left(X+X^{-1}-2\right)q+\left(5 (X+X^{-1})-9\right)q^2\\& 
+\left(-4 (X^2+X^{-2})+21(X+X^{-1})-31\right)q^3\\
&+\left(X^3+X^{-3}-18(X^2+X^{-2})+70(X+X^{-1})-101\right)q^4
+\mathcal{O}(q^5).
\end{aligned}
   \ee
Setting $x=0$ and $n=0$ in~\eqref{eq:hat-f-x-1}--\eqref{eq:H-q}, we obtain
\begin{equation}
\label{z1x0}
\hat J_{{\bf 4}_1,0}(q)=A_{{\bf 4}_1,0}(q)-\frac{1}{\log q}B_{{\bf 4}_1,0}(q),
\end{equation}
where
\begin{equation}
\label{A0}
A_{{\bf 4}_1,0}(q):=\left.\lim_{x\to0}A_{{\bf 4}_1,x}(q)\right\vert_{n=0}=\sum_{m\in\Z_{\ne0}}\frac{m\mg{q}{\min(m,0)}}{(q^{-m}-1)(q;q)_{|m|}}{_1\phi_1}\Bigl[
\begin{matrix}
 0\\
 q^{1+|m|}
\end{matrix}
;q,q^{1-\min(m,0)}\Bigr]
\end{equation}
and
\begin{equation}
\label{B0}
B_{{\bf 4}_1,0}(q)={_1\phi_1}\Bigl[
\begin{matrix}
 0\\
 q
\end{matrix}
;q,q\Bigr]=g(q). 
\ee
Here, 
\be
\label{def-gq}
g(q) =\sum_{n\geq 0}\frac{\gamma_q(-n)}{(q;q)_n^2} \end{equation}
is the $q$-series that first appeared in \cite{MR3696603} and that pervades the quantum topology of the figure-eight knot.
%

\subsection{Evaluation of the contour integral II}\label{subsection:evaluation-cont-int-4-1-2}

We will now evaluate \eqref{eq:41cont-int} by a residue calculation. 

Let us define the function
\be\label{def-fqn}
\fn{n,q}(z,x)=\sum_{m\in\mathbb{Z}}\frac{z^{-x-\frac{2\pi \ri m}{\log q}}}{\sin^n(\pi(x+\frac{2\pi \ri m}{\log q}))},
\ee
where $n\in\Z_{>0}$, and we will need more explicit expressions for the cases $n=1$ and $n=2$. 

For $n=1$, the function $\fn{1,q}(z,x)$ can be transformed by the bilateral Kronecker identity \cite{Weil1976} written in terms
of the dual nome
\be
\label{tq-def}
\tilde q=\re^{-4 \pi^2/h},\quad q=\re^{-h}.
\ee
For our purposes, it is better
to use formula~\cite[(4.3.9)]{MR2128719}
\be
\fn{1,q}(z,x)=\frac{h (q;q)^3_\infty}{\pi}\frac{\theta_q(-zq^x)}{\theta_q(q^x)\theta_q(-z)}\label{ident1a},
\ee
where we use the standard definition of the theta-function~\eqref{eq:theta-q1}.

%
Let us rewrite \eqref{def-fqn} with $n=1$ as follows
\be
\sum_{m\in\mathbb{Z}}\frac{\left(\tilde q^{1/2} z^{\frac{2\pi \ri}{h}}\right)^m}
{1-\re^{-2\pi \ri x}\tilde q^{m}}=\frac{1}{2\ri}z^x \re^{\ri\pi x}\fn{1,q}(z,x).
\ee
Differentiating with respect to $x$ and changing the sign of $z$, we come to
\be
\label{ident2}
\fn{2,q}(z,x)=\left[-\frac{1}{\pi}\log z-\frac{1}{\pi}\frac{\partial}{\partial x}\right]\fn{1,q}(-z,x).
\ee

\begin{theorem}\label{thm2}
We have
\be
\label{J41-res}
\ba
\hat J_{{\bf 4}_1,x} (q,z)&=-\frac{1}{z} {1 \over \left( q^{x\over 2}- q^{-{x\over 2}} \right)^2}\,
{_2\phi_2}\Bigl[
\begin{matrix}
 0,q\\
q^{1+x},q^{1-x}
\end{matrix}
;q,\frac{q}{z}\Bigr]\\
+ &{\pi^2 \over z } {  \left( q^{x\over 2}- q^{-{x\over 2}} \right) (q^{1+x},q^{1-x};q)_\infty^2 \over (\log q)^2  (q;q)^3_\infty}
\sum_{a\in\{\pm1\}}a \frac{q^{-{ax\over 2}} \fn{2,q}(z,ax)}{(q^{-2ax};q)_\infty}\,
{_1\phi_1}\Bigl[
\begin{matrix}
 0\\
q^{1+2ax}
\end{matrix}
;q,\frac{q^{1+ax}}{z}\Bigr].
\ea
\ee
\end{theorem}
The proof is given in Appendix~\ref{app:B}.

Formula~\eqref{J41-res} significantly simplifies
when $z=1$. For this, we define a quasi-modular Jacobi form of weight two,
\be
\label{quasi-mod}
\Phi(q,X)=  {X \over (X-1)^2} + \sum_{k \ge 1} {k q^k \over 1-q^k} (X^k + X^{-k}),
\ee
a Jacobi form of weight $-3$,
\be
\label{G3}
\mathcal{G}(q,X)= {2  \over  (q;q)_\infty^6}  {(X- 1)^3 \over X(X+ 1)} {(qX,q X^{-1};q)^4_\infty\over(q X^2,q X^{-2};q)_\infty},
\ee
and two more series
\be
\label{ABC}
\ba
\CC(q,X)&=
  \sum_{k=0}^{\infty}\frac{\gamma_q(-k)}{(X,X^{-1};q)_{k+1}}, \\
  \CZ(q,X)&= {(q;q)_\infty^5 \over 2 (X^{1/2} -X^{-1/2})} \sum_{a\in\{\pm1\}}{X^{-a/2} \over (\theta_q(X^{-a}))^2} \CJ(X^a, X^{2a}; q).
  \ea
\ee
Note that the Laurent series in~\eqref{quasi-mod} converges in the open annulus $|q|<|X|<|q|^{-1}$. We analytically continue $\Phi (q,X)$ outside of this annulus by the series 
\be
\Phi(q,X)=\sum_{\ell\in\mathbb Z}\frac{q^\ell X}{(1-q^\ell X)^2}.
\ee
In particular, $\Phi(q,X)$ has double poles at $X\in q^{\mathbb Z}$.
\begin{corollary}\label{cor:hatJ4_1-evaluates} When $z=1$, (\ref{J41-res}) evaluates to
\be
\label{jcz}
\hat J_{{\bf 4}_1,x} (q)=\CC( q,X)+ \mathcal{G}( q,X) \CZ(q,X) (h^{-1} +\Phi(q,X)),\quad X=q^x,\quad h=-\log q.
\ee
\end{corollary}
See~Appendix~\ref{app:B} for the proof.

\begin{remark} By comparing (\ref{jcz}) to (\ref{eq:hat-f-x-1}), (\ref{eq:H-q}),
we deduce the following identity for the $q$-Bessel function:
 \be
 \label{conj-id}
 {(qX, q/X ;q)_\infty  \over (qX^2, q/X^2;q)_\infty } \sum_{a\in\{\pm1\}}X^{a/2} \CJ(X^a, X^{2a};q)=
(X^{1/2}+X^{-1/2})\CJ(X, X^{-1};q).
\ee
By using this identity we also obtain the following expression for the series $A_{{\bf 4}_1,x}(q)$, see~\eqref{eq:G-0-q},
\be
\label{a-easy}
A_{{\bf 4}_1,x}(q)=\CC(q, X)+ \Phi( q,X) B_{{\bf 4}_1,x}(q),
\ee
where $B_{{\bf 4}_1,x}(q)$ is given in (\ref{b41-ex}). We can also write
\be
\label{j41-alt}
\hat J_{{\bf 4}_1,x}(q)=\CC(q, X)+ \ B_{{\bf 4}_1,x}(q) \left( \Phi( q,X)+ {1\over h} \right).
\ee
We note that the result (\ref{a-easy}) agrees with (\ref{eq:G-0-q}), where $n=0$.

We can use (\ref{a-easy}) to express the particular case $x=0$ in terms of known $q$-series. By expanding $X=\re^w$ as a series in $w$, and taking into account identity~\eqref{conj-id}, one finds
\be
A_{{\bf 4}_1,0}(q)= G_2(q)- \frac{1}{8} E_2(q) g(q).
\end{equation}
In this equation, $g(q)$ is defined in~\eqref{def-gq}, and
\be
G_2(q)=
\sum_{n\geq 0} \frac{\gamma_q(-n)}{(q;q)_n^2}
\left[
\frac{1}{2}\left(n+\frac{1}{2} -2E_1^{(n)}(q)\right)^2
- E_2^{(n)}(q)\right],
\end{equation}
where we have introduced, following~\cite{MR3696603}, the $q$-series
\begin{equation}
E_\ell^{(n)}(q) = \sum_{s=1}^\infty s^{\ell-1}\frac{q^{s(n+1)}}{1-q^s}, \qquad E_2(q)=1-24E_2^{(0)}(q).
\end{equation}
The series $G_2(q)$ also plays an important role in quantum topology of the figure-eight knot, as shown in \cite{MR4841756}.\hfill$\diamond$
\end{remark}

Let us make some observations on the above expressions for the interpolating Jones invariant $\hat J_{{\bf 4}_1,x}(q)$. The series $\CC(q,X)$ was denoted by
$\CC_0(q, X)$ in \cite{MR4841756}, and it appears as one of the ingredients 
in the conventional Borel resummation of the
MMR series. It agrees with the inverted Habiro series of the figure-eight knot obtained in \cite{park-inverted}. The series in 
the second line of (\ref{ABC}) is an appropriate combination of the series $\CA_0$, $\CB_0$ 
appearing in \cite{peacock,MR4841756}. Our formula~\eqref{j41-alt} for the interpolating Jones invariant of the figure-eight knot is then
given by its inverted Habiro series $\CC(q,X)$ (which comes from the residues at $s=\ri n$ in
the Watson integral representation (\ref{eq:41cont-int})) plus an additional or ``correction'' term
which comes from the other residues. The series $\CC(q, X)$ and $\Phi(q, X)$ are singular when $X=q^n$, $n\in \IZ$,
but these singularities cancel in the sum of the two terms in~\eqref{j41-alt}. It can be verified
that for $X=q^n$, $n \ge 1$,
the expressions (\ref{jcz}), (\ref{a-easy}) specialize to the $n$-colored Jones polynomial of the
figure-eight knot (the series $B_{{\bf 4}_1,x}(q)$ vanishes for this specialization). This correction term is also crucial for reproducing the MMR series in the asymptotic regime $h\to 0$ near $X=1$.

The formulae obtained so far allow us to evaluate the analytic Jones invariant $\hat J_{{\bf 4}_1,x}(q,z)$ for $0<q<1$, as a function of $x$ and $z$. In \figref{41-fig} we show two evaluations, illustrating different aspects of the invariant. In both cases we set $z=1$ and we use numerical integration to calculate the $q$-Laplace transform of (\ref{bt-41}). This plot illustrates how the interpolating Jones invariant specializes to $J_{{\bf 4}_1,n}(q)$ when $x=n$, and how it asymptotes to the MMR series near $q=1$. 

\begin{figure}[!ht]
\leavevmode
\begin{center}
\includegraphics[height=4cm]{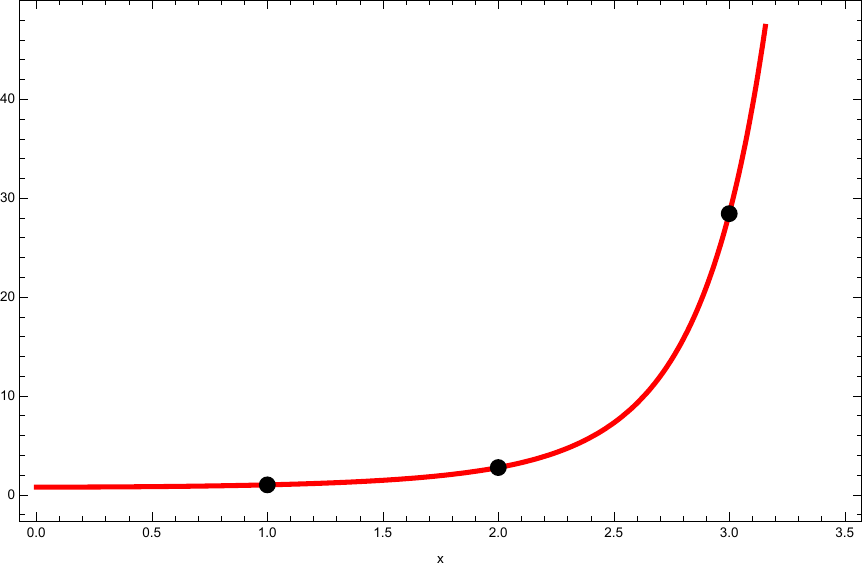}\qquad \qquad \includegraphics[height=4cm]{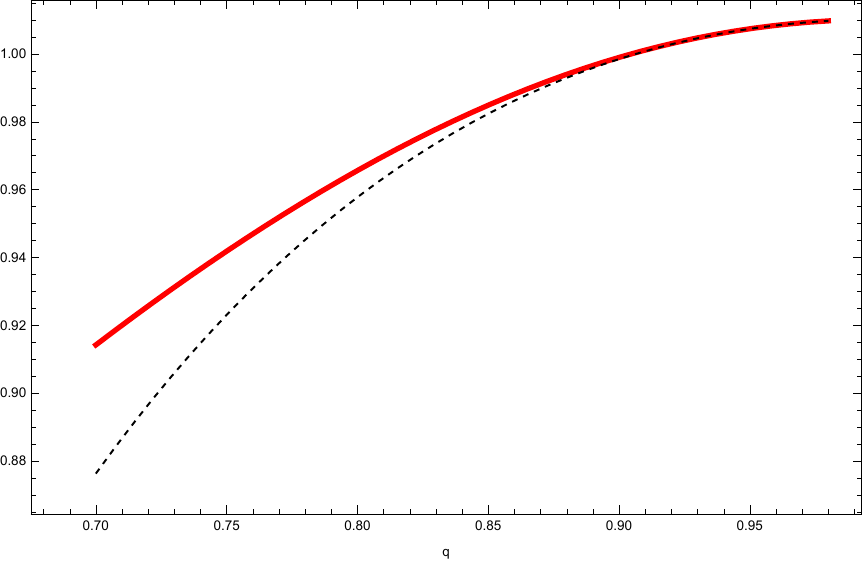}
\end{center}
\caption{Left-hand side: the red line is the interpolating Jones invariant $\hat J_{{\bf 4}_1,x}(q)$ as a function of $x$ for $q=1/2$. The black dots are the corresponding values $1, 11/4, 1819/64$ of the colored Jones polynomial for $n=1,2,3$. Right-hand side: the red line is the interpolating Jones invariant $\hat J_{{\bf 4}_1,x}(q)$ for fixed $w=-1/10$ as a function of $q$. The dashed black line is the MMR expansion, truncated at order $h^2$.}
\label{41-fig}
\end{figure}

\subsection{Comparison with the conventional Borel resummation}

Let us recall the following facts about the conventional Borel resummation. Given a Gevrey-1 series
\begin{equation}
  \varphi(h)=\sum_{n \ge 0} a_n h^{n},\quad a_n = \mathcal{O}(C^n n!),
\end{equation}
its Borel transform is defined by
\begin{equation}
  \widehat \varphi(\zeta)= \sum_{k \ge 0} {a_k \over k!} \zeta^k
\end{equation}
and is a holomorphic function in a neighborhood of the origin. In
favorable cases, this function can be extended to the complex
$\zeta$-plane (also called Borel plane), but it will have
singularities. The Borel resummation of
$\varphi(h)$, when it exists, is defined as the Laplace transform
\begin{equation}
  s(\varphi)(h) = \int_0^\infty \re^{-\zeta}
  \widehat \varphi(h \zeta) \rd \zeta.
\end{equation}
This has discontinuities at Stokes rays in the $h$ plane, whenever
$ {\rm arg}(h)= {\rm arg}(\zeta_s)$, where $\zeta_s$ is a
singularity of $\widehat \varphi(\zeta)$. We define the lateral Borel
resummations for $h$ near a Stokes ray by
\begin{equation}
  s_\pm (\varphi)(h)= \int_0^{\re^{\pm \ri \epsilon} \infty}
  \re^{-\zeta} \widehat \varphi(h \zeta) \rd \zeta.
\end{equation}

As we have seen, the $q$-Borel--Laplace resummation asymptotes to the MMR series, so it has to agree with its Borel resummation up to possible non-perturbative corrections in $h$. We now determine this difference explicitly.

So far, we have considered the MMR series $\Phi_0(h;w)$. However, to any knot $K$, it is possible to associate an additional set of formal perturbative series $\Phi^K_\sigma (h;w)$, labelled by non-trivial flat connections $\sigma$, which can be obtained e.g. by expanding the state integral of \cite{MR3227503, MR2330673} around its different saddle-points, see \cite{ggm1,peacock,MR4841756} for a review of their construction and meaning. In the case of the figure-eight knot, there are two additional series $\Phi^{{\bf 4}_1}_{\sigma}(h;w)$, labelled by the geometric $\sigma=g$ and the conjugate $\sigma=c$ connections. When $w=0$ the series $\Phi^{{\bf 4}_1}_\sigma (h):= \Phi^{{\bf 4}_1}_\sigma(h; 0)$ are given as follows. For $\sigma=g$ we have:
\be
\Phi^{{\bf 4}_1}_g(h)= \exp\left( -{\ri V \over h} \right) \varphi_g(h),
\ee
where
\be
 \varphi_{g} (h)=
  \frac{1}{\sqrt[4]3}\,
 \left(1 - \frac{11h}{72\sqrt{-3}}
  + \frac{697h^2}{2\,(72\sqrt{-3})^2}
  - \frac{724351 h^3}{30\,(72\sqrt{-3})^3} +\cdots\right)
\ee
and $V$ is the hyperbolic volume of the knot. In addition,
\be
\Phi^{{\bf 4}_1}_c(h)= \exp\left( {\ri V \over h} \right) \varphi_c(h), \qquad \varphi_c(h)= \ri \varphi_g(-h).
\ee

In \cite{ggm1,peacock,MR4841756}, explicit formulae were conjectured for the lateral Borel resummations of these formal series slightly above and below the negative imaginary axis in the $h$ plane. They express the resummations in terms of series in $q, X$ and in the dual variables $\tq, \tX$ defined in (\ref{tq-def}), (\ref{tX-def}) (see \cite{fantini-wheeler} for proofs of some related results). One can then use the Stokes automorphisms, also obtained explicitly in \cite{ggm1,peacock,MR4841756}, to obtain the (lateral) Borel resummation in arbitrary directions. In particular, one finds from the results in \cite{MR4841756} that the Borel resummation of the MMR series along the positive real axis of $h$ is given by
\be
\label{phi0im}
s (\Phi^{{\bf 4}_1}_0)(h;w)= \CC(q,X)-{\ri h \over 2 \pi}  {\CC (\tq,\tX)  \over  \CZ (\tq , \tX)} \CZ (q , X).
\ee
The series $\Phi^{{\bf 4}_1}_{\sigma}(h;w)$, $\sigma\in\{g,c\}$, are not  Borel summable along the positive real axis of $h$, and one has to consider lateral Borel resummations. By using the results of \cite{MR4841756}, one finds
\be
s_- (\Phi^{{\bf 4}_1}_g)(h;w)-s_+(\Phi^{{\bf 4}_1}_c)(h;w)= \left( {\ri h \over 2 \pi} \right)^{3/2}{ X^{1/2}- X^{-{1/2}} \over  \tX^{1/2}- \tX^{-{1/2}}}
{\CZ (q , X)  \over  \CZ (\tq , \tX)}.
\ee
We can then write (\ref{phi0im}) as
\be
\label{C0-res}
s (\Phi^{{\bf 4}_1}_0)(h;w)= \CC(q,X) -\left( {\ri h \over 2 \pi} \right)^{1/2} {\tX^{1/2}- \tX^{-{1/2}} \over X^{1/2}- X^{-{1/2}}} \CC(\tq,\tX) \left(s_- (\Phi^{{\bf 4}_1}_g)(h;w)-s_+(\Phi^{{\bf 4}_1}_c)(h;w)\right).
\ee
Let us note that the expression (\ref{phi0im}) shares many properties with (\ref{jcz}). In particular, the second order pole of $\CC(q,X)$ at $X=1$ cancels against a similar pole in the second term in the r.h.s. of (\ref{phi0im}). By using the modular properties of (\ref{quasi-mod}) and (\ref{G3}), we find
\be
\hat J_{{\bf{4}}_1,x}(q)=s(\Phi^{{\bf 4}_1}_0)(h;w) +\left( {\ri h \over 2 \pi} \right)^{-1/2}  { \tX^{1/2}- \tX^{-{1/2}} \over  X^{1/2}- X^{-{1/2}}} A_{{\bf{4}}_1,\tx}(\tq) \left( s_- (\Phi^{{\bf 4}_1}_g)(h;w)-s_+(\Phi^{{\bf 4}_1}_c)(h;w) \right),
\ee
where $\tq^{\tx}= \tX$. Since $ A_{{\bf{4}}_1,\tx}(\tq) \sim \tq$, the interpolating Jones invariant $\hat J_{{\bf{4}}_1,x}(q)$ and the conventional Borel resummation of the MMR series indeed differ by exponentially small terms in $h$.

The properties above can also be seen very clearly in the limit $X=1$. In this case, one finds
\be
s (\Phi^{{\bf 4}_1}_0)(h;0)=G_2(q) -{1 \over 24} E_2(q) g(q) +{4 \pi^2\over h^2} \left({G_2(\tq) \over g(\tq)}-{1\over 24} E_2(\tq) \right)g(q),
\ee
and
\be
\hat J_{{\bf 4}_1,0} (q)=s (\Phi^{{\bf 4}_1}_0)(h;0) -{4 \pi^2\over h^2}   A_{{\bf 4}_1,0}(\tq) {g(q) \over g(\tq)}.
\ee

\sectiono{General formalism}
\label{sec:gen-for}
In order to treat more general divergent power series, in particular those arising from arbitrary knots, we will now introduce a general formalism.

 \subsection{Simple meromorphic functions}

For a meromorphic function \(f(z)\) on \(\mathbb{C}\), regular at the origin, we denote by \(\ps{z}(f(z))\) (or just $\ps{}(f)$ if the argument $z$ is clear from the context)
the set of all poles of $f(z)$, endowed with an enumeration in nondecreasing order
of the modulus. We also denote by $\res{z}(f(z),a)$ (or just $\res{}(f,a)$, if the argument $z$ is clear from the context) the residue of the 1-form $f(z)\rd z$ at $z=a$:
\be
   \res{z}(f(z),a):=\res{z=a}f(z)\,\rd z.
\ee

\begin{definition}
Let  \(f(z)\) be a meromorphic function on \(\mathbb{C}\) regular at the origin. We say that $f(z)$ is
a \emph{{simple meromorphic}} function, or \emph{{sm-function}},  if
\begin{itemize}
 \item all poles of $f(z)$ are simple;
 \item $f(z)$ admits a Mittag-Leffler expansion with no entire part and no
convergence-correction terms, that is
\begin{equation}\label{eq:simple-meromorphic}
   f(z)=\sum_{a\in \ps{}(f)}\frac{\res{}(f,a)}{z-a},
\end{equation}
where the series converges normally on
\(\mathbb{C}\setminus \ps{}(f)\).
\end{itemize}
An sm-function is called of \emph{ finite type} if it has finitely many poles, otherwise it is called of \emph{ infinite type}.\hfill$\diamond$
\end{definition}
\begin{remark}\label{rem:necessary-cond-sm}
 Let $f(z)$ be an sm-function of infinite type. Then, from~\eqref{eq:simple-meromorphic}  it follows that the sum
\begin{equation}\label{eq:abs-conv-ms-fun-crit}
\sum_{a\in\ps{}(f)}\frac{-1}{a}\res{}(f,a)
\end{equation}
absolutely converges to $f(0)$. In practice, this convergence condition can serve as a necessary condition for a meromorphic function (regular at the origin and having all poles simple) to be an sm-function.
\end{remark}
\begin{remark}
The set $\smf$ of all sm-functions is obviously a $\C$-vector space. The subset $\smf_{\text{fin}}\subset \smf$ of sm-functions of finite type is a vector subspace generated by \emph{elementary} sm-functions $\frac{1}{z-a}$ with $a\in\C^\times$.\hfill$\diamond$
\end{remark}
\begin{remark}\label{rem:az+b-trans-sm-func}
 For any sm-function $f(z)$ and any scalars $a\in\C^\times$ and $b\in\C\setminus\ps{}(f),$ the function $f(az+b)$ is also an sm-function. Furthermore, $c$ is a pole of $f(az+b)$ if and only if $ac+b$ is a pole of $f(z)$, and the relation between the residues is given by the equality
 \begin{equation}
 \res{z}( f(az+b),c)=\frac{ 1}a\res{z}(f(z),ac+b).
\end{equation}
\hfill$\diamond$
\end{remark}

\begin{example}\label{ex:fabq(z)}
For any $a$ such that $|a|<1$ and $b\in\C\setminus q^{-\Z_{\ge0}}$, the  series
\begin{equation}\label{eq:fabq(z)}
f_{a,b,q}(z):=\sum_{n\ge0}\frac{(a;q)_nz^n}{(b;q)_n}=
{_2\phi_1}\Bigl[
\begin{matrix}
 a,q\\
 b
\end{matrix}
;q,z\Bigr],
\end{equation}
with radius of convergence 1, is Taylor's expansion of an sm-function.
Indeed, for $|z|<1$, we write
\begin{align}
f_{a,b,q}(z)&=\sum_{m\ge0}\frac{(a,bq^m;q)_\infty z^m}{(b,aq^m;q)_\infty}=\sum_{m,n\ge0}\frac{(a;q)_\infty (\frac ba;q)_n a^{n}q^{mn}z^m}{(b;q)_\infty(q;q)_n}\nonumber\\
&=\sum_{n\ge0}\frac{(a;q)_\infty (\frac ba;q)_n a^{n}}{(b;q)_\infty(q;q)_n(1-q^{n}z)},
\end{align}
where the second equality follows from the \(q\)-binomial theorem. The last expression corresponds to a Mittag-Leffler expansion with no entire part and no
convergence-correction terms. This means that it is an sm-function with the pole set
$$
\ps{}(f_{a,b,q})\subseteq q^{-\Z_{\ge0}},
$$
where each pole $q^{-n}$ is simple or removable with the residue
\begin{equation}\label{eq:res-fabq(z)}
\res{}(f_{a,b,q},q^{-n})=-\frac{(a;q)_\infty (\frac ba;q)_n (\frac aq)^{n}}{(b;q)_\infty(q;q)_n}.
\end{equation}
\hfill$\diamond$
\end{example}

\begin{example}\label{ex:example-3}
Later in this paper we will come across the function
 \begin{equation}\label{eq:main-non-sm-func}
F_q(z):=\sum_{n\ge0}\frac{q^{n(n+1)}}{(q;q)_n(zq^{n};q)_\infty}z^n=\frac{1}{(z;q)_\infty}\sum_{n\ge0}\frac{(z;q)_n}{(q;q)_n}q^{n(n+1)}z^n=\frac{{_1\phi_2}\Bigl[
\begin{matrix}
 z\\
 0,0
\end{matrix}
 ;q,zq^{2}\Bigr]}{(z;q)_\infty},
\end{equation}
which is not an sm-function, despite the fact that it has the pole set
\be
\ps{z}(F_q(z))\subseteq q^{-\Z_{\ge0}},
\ee
where each pole $q^{-n}$ is simple or removable with the residue
\begin{equation}\label{eq:res-of-main-non-sm-fun}
\res{z}(F_q,q^{-n})=-\sum_{k=0}^n\frac{\mg{q}{k}}{(q;q)_{\infty,k,n-k}}q^{k(k-n)}.
\end{equation}
To derive this expression, we first set $b=q$ and send $a\to0$ in~\eqref{eq:res-fabq(z)}:
\be
\res{}(f_{0,q,q},q^{-n})=-\frac{\mg{q}{n}}{(q;q)_{\infty,n}}.
\ee
Next, we remark that
\be
f_{0,q,q}(z)=(z;q)_\infty^{-1},
\ee
so that, from~\eqref{eq:main-non-sm-func}, we obtain
\be
\res{}(F_q,q^{-n})=\res{}(f_{0,q,q},q^{-n})
{_1\phi_2}\Bigl[
\begin{matrix}
 q^{-n}\\
 0,0
\end{matrix}
 ;q,q^{2-n}\Bigr]=-\mg{q}{n}\sum_{k=0}^n\frac{(q^{-n};q)_k}{(q;q)_{\infty,n,k}}q^{k(k-n+1)},
\ee
By using the equality
\be
(q^{-n};q)_k
=
\mg{q}{k}q^{-kn}
\frac{(q;q)_n}{(q;q)_{n-k}},
\ee
and changing the summation variable $k\mapsto n-k$, we arrive at~\eqref{eq:res-of-main-non-sm-fun}.

One can show that $\res{}(F_q,q^{-n})$ diverges as $q^{-n^2/6}$ when $n\to\infty$.
\hfill\(\diamond\)
\end{example}
\subsection{Hadamard product}
Given two formal power series
\begin{equation}\label{eq:series-f-g}
f(z)=\sum_{n\ge0}f_n z^n,
\quad
g(z)=\sum_{n\ge0}g_n z^n,
\end{equation}
their \emph{Hadamard product} is defined by
\be
(f\had g)(z)=\sum_{n\ge0}f_n g_n z^n.
\ee
\begin{remark}
 The set $\fps{z}$ of all formal power series is an algebra with componentwise addition and multiplication (Hadamard product), to be denoted $\fps{z}_{\had}$ and to be called \emph{Hadamard ring}, where the multiplicative identity element is the geometric series
 \be
 \sum_{n\ge 0}z^n=(1-z)^{-1}.
 \ee
 An element $f(z)\in\fps{z}$ is Hadamard invertible if and only if $f_n\ne0$ for any $n\in\Z_{\ge0}$ with the Hadamard-inverse
 \be
 f(z)^{\had-1}=\sum_{n\ge0}\frac{z^n}{f_n}\in\fps{z}.
 \ee
 For example, the inverse infinite $q$-Pochhammer symbol $(z;q)_\infty^{-1}$ is an invertible element of $\fps{z}_{\had}$ with the inverse
 \be
 \left((z;q)_\infty^{-1}\right)^{\had-1}=\sum_{n\ge0}(q;q)_nz^n,
 \ee
 which is simple meromorphic, as a special case of Example~\ref{ex:fabq(z)} given by $f_{q,0,q}(z).$
 \hfill$\diamond$
\end{remark}
If \(f\) and \(g\) have strictly positive radii of convergence \(r(f)\) and \(r(g)\), respectively, then the Hadamard product can be expressed as the contour integral
\begin{equation}\label{eq:cont-int-hadamart-pr}
(f\had g)(z)
=
\oint_{\mathcal C} f\Big(\frac zu\Big)g(u)\frac{\rd u}{2\pi \ri u},
\end{equation}
where the closed contour $\mathcal C$ in the $u$-plane encircles all singularities of $f(\frac zu)\frac{\rd u}{u}$ in the anti-clockwise direction, while all singularities  of $g(u)$ are outside of  $\mathcal C$. In particular, for $0\le |z|<r(f)r(g),$ one can choose
\be
\mathcal C=C_\rho:=\{u\in\C\mid |u|=\rho\},\ \text{ with }\ \frac{|z|}{r(f)}<\rho<r(g).
\ee
In what follows, we will also informally write \(f(z)\had g(z)\) or \(f(z)\had_z g(z)\), instead of \((f\had g)(z)\).

\begin{lemma}[Residue formula for the Hadamard product]\label{lem:res-form-hadamart-pr}
 Let the respective radii of convergence \(r(f)\) and \(r(g)\) of the series
\(f(u)\) and \(g(u)\) be strictly positive, and assume that $f(u)$ is Taylor's expansion of a simple meromorphic function.
Let $z\in\C$ be such that the analytically continued value $g(\frac za)$ is well defined for any $a\in \ps{}(f)$. Then
\begin{equation}\label{eq:res-form-hadamart-pr}
(f\had g)(z)=-\sum_{a\in \ps{}(f)} \frac{\res{}(f,a)}{a} g\Big(\frac za\Big).
\end{equation}
\end{lemma}
\begin{proof}
 Denote by $B{\mathcal C}$ the bounded component of $\C\setminus\mathcal C$, where $\mathcal C$ is the integration contour in~\eqref{eq:cont-int-hadamart-pr}.

Assume first that $f(u)=\frac{1}{u-a}$ is an elementary sm-function. The 1-form $f(\frac zu)g(u)\frac{\rd u}{u}$ for $u\in B{\mathcal C}$  is meromorphic with the only simple pole $u=\frac za$ and the residue
\be
\res{u=\frac za}f\left(\frac zu\right)g(u)\frac{\rd u}{u} =-\frac{ \res{}(f,a)}{a}
g\left(\frac za\right).
\ee
By linearity of the integral, we conclude that formula~\eqref{eq:res-form-hadamart-pr} holds for any sm-function of finite type $f\in\smf_{\text{fin}}.$

Let now $f(u)$ be an arbitrary sm-function with the set of poles \(\ps{}(f)=\{a_1,a_2,\ldots\},\) enumerated so that $|a_j|\le |a_{j+1}|$ for any $j\ge1.$ Its Mittag-Leffler expansion is
\be
   f(u)=\sum_{j\geq 1}\frac{\res{}(f,a_j)}{u-a_j}.
\ee
Let
\be
   f_N(u):=\sum_{j=1}^N \frac{\res{}(f,a_j)}{u-a_j}\in\smf_{\text{fin}}.
\ee
As \(f_N(u)\) converges to \( f(u)\) normally on \(\mathbb C\setminus \ps{}(f)\), we conclude, in particular,
that \(f_N(\frac zu)\) converges to \(f(\frac zu)\) uniformly on \(u\in\mathcal C\), and we get
\be
   \lim_{N\to\infty}(f_N\had g)(z)=(f\had g)(z).
\ee
On the other hand, since~\eqref{eq:res-form-hadamart-pr} holds for every $f_N$, it holds also for $f$.
\end{proof}

\begin{lemma}\label{lem:product-pole-sets-discrete}
Let \(A\) and \(B\) be discrete subsets of \(\mathbb C^\times\) with no
accumulation point in \(\mathbb C\). Then
\be
   AB:=\{ab \mid a\in A,\ b\in B\}
\ee
is again a discrete subset of \(\mathbb C^\times\) with no
accumulation point in \(\mathbb C\). Moreover, for every
\(c\in AB\), the set
\be
   \{(a,b)\in A\times B \mid ab=c\}
\ee
is finite.
\end{lemma}

\begin{proof}
Since \(0\notin A\) and \(0\notin B\), and neither \(A\) nor \(B\) has an
accumulation point at \(0\), there exist \(\varepsilon_A,\varepsilon_B>0\) such
that
\be
   |a|\geq \varepsilon_A
   \qquad\text{and}\qquad
   |b|\geq \varepsilon_B
\ee
for all \(a\in A\) and \(b\in B\).

Let \(K\subset\mathbb C\) be compact. Choose \(R>0\) such that
\(K\subseteq \{z\in\mathbb C\mid |z|\leq R\}\). If \(ab\in K\), then
\be
   |a||b|=|ab|\leq R.
\ee
Hence
\be
   |a|\leq \frac{R}{\varepsilon_B},
   \qquad
   |b|\leq \frac{R}{\varepsilon_A}.
\ee
Thus \(a\) and \(b\) must belong to the finite sets
\be
   A\cap \Bigl\{|z|\leq \frac{R}{\varepsilon_B}\Bigr\},
   \qquad
   B\cap \Bigl\{|z|\leq \frac{R}{\varepsilon_A}\Bigr\}.
\ee
Therefore \(AB\cap K\) is finite. This proves that \(AB\) has no accumulation
point in \(\mathbb C\), hence is discrete. Since \(A,B\subset\mathbb C^\times\),
we also have \(AB\subset\mathbb C^\times\).

Taking \(K=\{c\}\), the same argument shows that the fiber
\be
   \{(a,b)\in A\times B \mid ab=c\}
\ee
is finite.
\end{proof}

\begin{proposition}\label{prop:sm-stable-hadamard}
The space of all sm-functions is stable under the Hadamard product.
\end{proposition}

\begin{proof}
Let \(f\) and \(g\) be sm-functions. Write
\be
   f(u)=\sum_{a\in \ps{}(f)}
   \frac{\res{}(f,a)}{u-a},
   \qquad
   g(u)=\sum_{b\in \ps{}(g)}
   \frac{\res{}(g,b)}{u-b}.
\ee
As the pole sets
\(\CP(f)\) and \(\CP(g)\) are subsets of \(\mathbb C^\times\) with no accumulation point in $\C$, by
Lemma~\ref{lem:product-pole-sets-discrete}, the product set
\be
   \ps{}(f)\ps{}(g):=\{ab \mid a\in \ps{}(f),\ b\in \ps{}(g)\}
\ee
is again a discrete subset of \(\mathbb C^\times\) with no accumulation point in $\C$.

Let \(z\notin \ps{}(f)\ps{}(g)\). By the residue formula for the Hadamard product, we
have
\be
   (f\had g)(z)
   =
   -\sum_{a\in \ps{}(f)}
   \frac{\res{}(f,a)}{a}
   g\!\left(\frac za\right).
\ee
Using the Mittag-Leffler expansion of \(g\), we get
\be
   g\!\left(\frac za\right)
   =
   \sum_{b\in \ps{}(g)}
   \frac{\res{}(g,b)}{\frac za-b}
   =
   \sum_{b\in \ps{}(g)}
   \frac{a\,\res{}(g,b)}{z-ab}.
\ee
Therefore
\be
   (f\had g)(z)
   =
   -\sum_{a\in \ps{}(f)}
    \sum_{b\in
    \ps{}(g)}
   \frac{\res{}(f,a)\res{}(g,b)}{z-ab}.
\ee
Grouping together the finitely many terms with the same product \(c=ab\), we
obtain
\be
   (f\had g)(z)
   =
   \sum_{c\in \ps{}(f)\ps{}(g)}
   \frac{R_c}{z-c},
\ee
where
\be
   R_c
   :=
   -\sum_{\substack{a\in \ps{}(f),\, b\in \ps{}(g)\\ ab=c}}
   \res{}(f,a)\res{}(g,b)
\ee
 is a finite sum by Lemma~\ref{lem:product-pole-sets-discrete}.

If \(R_c=0\), then the apparent pole at \(c\) is removable. Hence the actual
pole set of \(f\had g\) is
\be
   \ps{}(f\had g)
   =
   \{c\in \ps{}(f)\ps{}(g)\mid R_c\neq 0\}\subseteq \ps{}(f)\ps{}(g).
\ee
Thus \(f\had g\) has only nonzero simple poles and admits the Mittag-Leffler
expansion
\be
   (f\had g)(z)
   =
   \sum_{c\in \ps{}(f\had g)}
   \frac{\res{}(f\had g,c)}{z-c},
\ee
with no entire part and no convergence-correction terms. Therefore
\(f\had g\) is an sm-function.
\end{proof}
\begin{example}\label{ex:example-4}
Proposition~\ref{prop:sm-stable-hadamard} allows us to construct many more examples of sm-functions of infinite type by using Example~\ref{ex:fabq(z)}.
Namely, for any $r\in\Z_{\ge0}$, any $a_i$ such that $|a_i|<1$ with $1\le i\le r+1$, and any $b_i\notin q^{-\Z_{\ge0}}$ with $1\le i\le r$, the basic hypergeometric series
\be
{_{r+1}\phi_r}\Bigl[
\begin{matrix}
 a_1,\dots,a_{r+1}\\
 b_1,\dots,b_r
\end{matrix}
;q,z\Bigr]=f_{a_1,b_1,q}(z)\had\cdots\had f_{a_r,b_r,q}(z)\had f_{a_{r+1},q,q}(z)\label{phirs}
\ee
is simple meromorphic as a finite Hadamard product of simple meromorphic functions.
\hfill\(\diamond\)
\end{example}

\subsection{Deformed Borel and Laplace transforms}\label{sec:defBL}
We first note that the $q$-Borel transform can be written as the following Hadamard product
\be
\mathcal{B}_q f(z)=(-qz;q)_\infty\had_z f(z).
\ee
Due to the obvious equalities
\be\label{eq:q-borel-of-had-prod}
\mathcal{B}_q(f\had g)(z)=(\mathcal{B}_qf)(z)\had g(z)=f(z)\had (\mathcal{B}_qg)(z),
\ee
writing the expression without parentheses
\be
\mathcal{B}_qf(z)\had_z g(z)
\ee
leaves no ambiguity.
\begin{corollary}\label{lq-cor}
 Under the assumptions of Lemma~\ref{lem:res-form-hadamart-pr}, let $z\in\C^\times$ be such that the $q$-Laplace transform $\mathcal L_qg$ is well defined on the ray $\frac zat$, $t\in\R_{\ge0}$, for each $a\in \ps{}(f)$, and the following condition is satisfied:
 \be\label{eq:tonelli}
 \sum_{a\in \ps{}(f)}\Bigl|\frac{\res{}(f,a)}{a}\Bigr|\int_0^\infty\frac{|g(\frac zat)|}{(-t;q)_\infty}\rd t<\infty.
 \ee
  Then
\begin{equation}\label{eq:res-form-laplace}
\mathcal L_q(f\had g)(z)=-\sum_{a\in \ps{}(f)} \frac{\res{}(f,a)}{a} \mathcal L_qg\Big(\frac za\Big).
\end{equation}
\end{corollary}
\begin{proof}
 Formula~\eqref{eq:res-form-laplace} follows from~\eqref{eq:res-form-hadamart-pr} and linearity of the $q$-Laplace transform, while condition~\eqref{eq:tonelli} allows us to use Tonelli's and then Fubini's theorems for exchanging the order of the summation and the integration.
\end{proof}
\begin{corollary}\label{cor:q-bor-lap-resum-res-form}
Let \(f(u)\) be a simple meromorphic function and \(g(u)\) a \(q\)-Borel--Laplace resummable formal power series, such that $\mathcal{B}_q g$ satisfies the conditions of Corollary~\ref{lq-cor}. Then \((f\had g)(u)\) is \(q\)-Borel--Laplace resummable, and for every \(z\in\mathbb C^\times\) satisfying the hypotheses of Corollary 2, the \(q\)-Borel--Laplace resummation \(\hat{g}(z/a)\) is well defined for every \(a\in \ps{}(f)\) and the \(q\)-Borel--Laplace resummation \(\widehat{f\had g}(z)\) can be computed according to the formula
\begin{equation}\label{eq:res-form-q-borel-laplace-resum}
\widehat{f\had g}(z)=-\sum_{a\in \ps{}(f)} \frac{\res{}(f,a)}{a} \hat{g}\left(\frac za\right).
\end{equation}
\end{corollary}
\begin{proof}
 This follows from~\eqref{eq:q-borel-of-had-prod} and~\eqref{eq:res-form-laplace}.
\end{proof}
\begin{example}\label{ex:me_q(z)}
The following divergent power series will be particularly useful:
 \be\label{eq:me_q(z)}
 \me{q}(z):=\sum_{n\ge0}\frac{z^n}{\gamma_q(-n)}. 
 \ee
Through the $q$-Borel--Laplace resummation, it can be converted into an analytic function in the $z$-complex plane cut along the negative real half-axis:
 \be\label{eq:q-b-l-resum-mu_q}
 \rme{q}(z)=\sum_{m\in\Z}\frac{\gamma_q(-m)}{(q;q)_\infty^3}\frac{\frac{\log z}{\log q}+m}{1-zq^m}.
 \ee

Indeed, applying the $q$-Borel transform, we have
 \be\label{eq:B_qmeqz}
 \mathcal{B}_q\me{q}(z)=(-z;q)_\infty^{-1}=\int_{\ri\R-\epsilon}\frac{1}{(q;q)_s}\frac{z^s}{2\ri\sin(-\pi s)}\rd s,
 \ee
 where the second equality is Watson's contour integral representation valid for $|\arg z|<\pi$. Applying to the latter the $q$-Laplace transform and using~\eqref{eq:t-s-int}, we obtain
 \begin{align}
 \mathcal{L}_q\mathcal{B}_q\me{q}(z)&=\int_{\ri\R-\epsilon}\frac{R_s(q)}{(q;q)_s}\frac{z^s}{2\ri\sin(-\pi s)}\rd s
 =\int_{\ri\R-\epsilon}\frac{1}{(q;q)_{s,-s-1}}\frac{z^s}{\frac{2\ri}{\pi}\sin^2(\pi s)}\rd s\nonumber\\
& =\int_{\ri\R-\epsilon}\frac{\theta_q(q^{-s})}{(q;q)_\infty^3}\frac{z^s}{\frac{2\ri}{\pi}\sin^2(\pi s)}\rd s
 =\sum_{m\in\Z}\frac{\gamma_q(-m)}{(q;q)_\infty^3}\int_{\ri\R-\epsilon}\frac{(zq^m)^s}{\frac{2\ri}{\pi}\sin^2(\pi s)}\rd s\nonumber\\
& =\sum_{m\in\Z}\frac{\gamma_q(-m)}{(q;q)_\infty^3}\frac{\log(zq^m)}{zq^m-1},
\end{align}
where, in the third and fourth equalities, we used the definition and the Laurent series 
expansion of the theta function~\eqref{eq:theta-q1}, while in the last step, we used the integration formula
\be\label{eq:frac{log a}{a-1}}
\int_{\ri\R-\epsilon}\frac{a^s}{\frac{2\ri}{\pi}\sin^2(\pi s)}\rd s=\frac{\log a}{a-1}.
\ee
Using the definition
\be\label{eq:def-rme}
\rme{q}(z):=\frac{-1}{\log q}\mathcal{L}_q\mathcal{B}_q\me{q}(z),
\ee
we obtain~\eqref{eq:q-b-l-resum-mu_q}. \hfill $\diamond$
\end{example}

The following lemma is a useful property of the \(q\)-Borel--Laplace resummation, which also holds for the usual Borel resummation.

\begin{lemma}\label{lem:f(z)=zg(z)}
Let $g(z)$ be a formal power series  and $p(z)$ a polynomial. Define a formal power series
\be
f(z)=p(z)g(z).
\ee
Assume that the $q$-Borel--Laplace resummation $\hat{g}(z)$ of $g(z)$ is well defined for $z\in\C^\times$.
Then the $q$-Borel--Laplace resummation $\hat{f}(z)$ of $f(z)$ is also well defined, and
\be\label{eq:hat-f=ehat-g}
\hat f(z)=p(z)\hat g(z).
\ee
\end{lemma}

\begin{proof}
By linearity, it is enough to prove the claim for monomials $p(z)=z^m$.
By iteration, it therefore suffices to consider the case $p(z)=z$.

In this case,
\[
f_0=0,
\qquad
f_{n+1}=g_n,
\qquad n\geq 0.
\]
Hence, we have for the $q$-Borel transform
\begin{align}
\tilde f(z):=\mathcal{B}_q f(z)
&=
\sum_{n\geq 1}
\frac{q^{n(n+1)/2}}{(q;q)_n}f_n z^n
=
\sum_{n\geq 1}
\frac{q^n}{1-q^n}
\frac{q^{n(n-1)/2}}{(q;q)_{n-1}}f_n z^n
=
z\sum_{n\geq 0}
\frac{q^{n+1}}{1-q^{n+1}}
\frac{q^{n(n+1)/2}}{(q;q)_n}g_n z^n
\nonumber\\
&=
z(h(z)\had_z \tilde g(z)),
\label{eq:tilde-f=z-sum-tilde-g}
\end{align}
where
\be
h(z):=
\sum_{n\geq 1}\frac{q^n}{1-q^n}z^{n-1}.
\ee
For $|qz|<1$, we may rewrite
\begin{align}
h(z)
=
\sum_{n\geq 1}\sum_{k\geq 1}q^{kn}z^{n-1}=
\sum_{k\geq 1}\frac{q^k}{1-q^kz}.
\end{align}
Thus $h$ is Taylor's expansion of an sm-function. By the residue formula~\eqref{eq:res-form-hadamart-pr},
\be\label{eq:res-sum-h-had-tilde-g}
h(z)\had_z\tilde g(z)
=
\sum_{k\geq 1}q^k\tilde g(q^kz).
\ee

Combining \eqref{eq:tilde-f=z-sum-tilde-g} and
\eqref{eq:res-sum-h-had-tilde-g}, we obtain
\begin{align}
\hat f(z)
=
\frac{-1}{\log q}\mathcal L_q\tilde f(z)
=
\frac{-1}{\log q}
\int_0^\infty
\frac{\tilde f(zt)}{(-t;q)_\infty}\,\rd t=
z\sum_{k\geq 1}
\frac{-1}{\log q}
\int_0^\infty
\frac{q^k\tilde g(q^kzt)t}{(-t;q)_\infty}\,\rd t,
\end{align}
where the interchange of the sum and the integral is justified by the
assumed $q$-Borel--Laplace summability of $g(z)$.

Using
\be
(-t;q)_\infty=(1+t)(-qt;q)_\infty,
\ee
we have
\be
\frac{t}{(-t;q)_\infty}
=
\frac{1}{(-qt;q)_\infty}
-
\frac{1}{(-t;q)_\infty}.
\ee
Therefore
\begin{align}
\hat f(z)
&=
z\sum_{k\geq 1}
q^k\frac{-1}{\log q}
\int_0^\infty
\frac{\tilde g(q^kzt)}{(-qt;q)_\infty}\,\rd t
-
z\sum_{k\geq 1}q^k\hat g(q^kz).
\end{align}
In the first sum, make the change of variables $t\mapsto s=qt$. Then
\begin{align}
q^k
\int_0^\infty
\frac{\tilde g(q^kzt)}{(-qt;q)_\infty}\,\rd t
&=
q^{k-1}
\int_0^\infty
\frac{\tilde g(q^{k-1}zs)}{(-s;q)_\infty}\,\rd s.
\end{align}
Hence
\begin{align}
\hat f(z)
=
z\sum_{k\geq 0}q^k\hat g(q^kz)
-
z\sum_{k\geq 1}q^k\hat g(q^kz)=
z\hat g(z).
\end{align}
\end{proof}
\subsection{Resummation of ${_r\phi_s}$ basic hypergeometric series}\label{sec:resummation-r-fi-s}
The basic hypergeometric series~\cite[(1.2.22)]{MR2128719},
\be\label{eq:r-fi-s}
{_r\phi_s}\Bigl[
\begin{matrix}
 a_1,\dots,a_r\\
 b_1,\dots,b_s
\end{matrix};q,z
\Bigr]=\sum_{m\ge0}\frac{(a_1,\dots,a_r;q)_m}{(b_1,\dots,b_s,q;q)_m}\gamma_q(m)^{1+s-r}z^m
\ee
is generically divergent for $r>s+1,$ and it can be converted into an analytic function by applying the $p$-Borel--Laplace resummation with $p:=q^{r-s-1}.$
\begin{theorem}\label{thm:r-fi-s}
Let integers $r,s\ge0$ and complex numbers $a_1,\dots, a_r, b_1,\dots, b_s$ satisfy the conditions
\be\label{eq:cond-par-r-fi-s}
r>s+1,\quad \max(|a_1|,\dots, |a_r|)<1,\quad (b_1,\dots,b_s;q)_\infty\ne0.
\ee
Then, denoting
\be
p:=q^{r-s-1},\quad \kappa:=(-1)^{r-s},
\ee
 the $p$-Borel--Laplace resummation ${_r\hat{\phi}_s}$ of the $ {_r{\phi}_s}$ basic hypergeometric series~\eqref{eq:r-fi-s} is of the form
 \be\label{eq:resummed-r-fi-s}
 {_r\hat{\phi}_s}\Bigl[
\begin{matrix}
 a_1,\dots,a_r\\
 b_1,\dots,b_s
\end{matrix};q,z
\Bigr]=\frac{(a_1,\dots,a_r;q)_\infty}{(b_1,\dots,b_s, q;q)_\infty}\sum_{l_1,\dots,l_r\ge0}\rme{p}(\kappa pzq^{l_1+\dots+l_r})\prod_{i=1}^r\frac{\prod_{j=0}^{l_i-1}(a_i-b_i q^j)}{(q;q)_{l_i}},
 \ee
 where
 \be
b_{i}=q\delta_{i,r},\quad s+1\le i\le r,
\ee
and the function $\rme{p}(z)$ is given in~\eqref{eq:q-b-l-resum-mu_q}.
\end{theorem}
\begin{proof}
With notation~\eqref{eq:fabq(z)} and definition~\eqref{eq:me_q(z)},
 we write
\be\label{eq:S_n(z)}
{_r\phi_s}\Bigl[
\begin{matrix}
 a_1,\dots,a_r\\
 b_1,\dots,b_s
\end{matrix};q,z
\Bigr]=f_{a_1,b_1,q}(z)\had_z\cdots\had_zf_{a_r,b_r,q}(z)\had_z \me{p}(\kappa pz).
\ee
Under conditions~\eqref{eq:cond-par-r-fi-s},
all the factors  $f_{a_i,b_i,q}(z)$ in~\eqref{eq:S_n(z)} are sm-functions, see Example~\ref{ex:fabq(z)}, so that we can apply $r$ times the residue formula~\eqref{eq:res-form-q-borel-laplace-resum} to obtain
\be\label{eq:r-fi-s-residue-sum}
{_r\hat{\phi}_s}\Bigl[
\begin{matrix}
 a_1,\dots,a_r\\
 b_1,\dots,b_s
\end{matrix};q,z
\Bigr]=\sum_{l_1,\dots,l_r\ge0}\rme{p}(\kappa pzq^{l_1+\dots+l_r})\prod_{i=1}^r\frac{(a_i;q)_\infty (\frac{b_i}{a_i};q)_{l_i} a_i^{l_i}}{(b_i;q)_\infty(q;q)_{l_i}},
\ee
which is an equivalent form of~\eqref{eq:resummed-r-fi-s}, when $a_1\cdots a_r\ne0$.
\end{proof}
\begin{remark}\label{rem:choice-of-n-for-r-fi-s}
 If the parameters $a_1,\dots,a_r$ in~\eqref{eq:r-fi-s} do not satisfy conditions~\eqref{eq:cond-par-r-fi-s}, there always exists a positive integer
 $n$ such that
\be
\max(|a_1|,\dots, |a_r|)<q^{-n}.
\ee
Subtracting the first $n$ terms from~\eqref{eq:r-fi-s}, we obtain
\begin{align}\label{eq:r-fi-s-n-terms}
{_r\phi_s}\Bigl[
\begin{matrix}
 a_1,\dots,a_r\\
 b_1,\dots,b_s
\end{matrix};q,z
\Bigr]-\sum_{m=0}^{n-1}\frac{(a_1,\dots,a_r;q)_m}{(b_1,\dots,b_s,q;q)_m}\gamma_q(m)^{s+1-r}z^m\nonumber\\
=\frac{(a_1,\dots,a_r;q)_n}{(b_1,\dots,b_s,q;q)_n}\gamma_q(n)^{s+1-r}z^n{_{r+1}\phi_{s+1}}\Bigl[
\begin{matrix}
 a_1q^n,\dots,a_rq^n,q\\
 b_1q^n,\dots,b_sq^n,q^{1+n}
\end{matrix};q,zq^{n(s+1-r)}
\Bigr],
\end{align}
where the parameters of ${_{r+1}\phi_{s+1}}$ basic hypergeometric series satisfy conditions~\eqref{eq:cond-par-r-fi-s}, so that Theorem~\ref{thm:r-fi-s} applies, and by using Lemma~\ref{lem:f(z)=zg(z)}, we obtain
\begin{align}\label{eq:r-hat-fi-s-n-terms}
{_r\hat{\phi}_s}\Bigl[
\begin{matrix}
 a_1,\dots,a_r\\
 b_1,\dots,b_s
\end{matrix};q,z
\Bigr]=\sum_{m=0}^{n-1}\frac{(a_1,\dots,a_r;q)_m}{(b_1,\dots,b_s,q;q)_m}\gamma_q(m)^{s+1-r}z^m\nonumber\\
+\frac{(a_1,\dots,a_r;q)_n}{(b_1,\dots,b_s,q;q)_n}\gamma_q(n)^{s+1-r}z^n{_{r+1}\hat{\phi}_{s+1}}\Bigl[
\begin{matrix}
 a_1q^n,\dots,a_rq^n,q\\
 b_1q^n,\dots,b_sq^n,q^{1+n}
\end{matrix};q,zq^{n(s+1-r)}
\Bigr].
\end{align}
\end{remark}

In order to illustrate the utility of Remark~\ref{rem:choice-of-n-for-r-fi-s}, here we provide a proof of Theorem~\ref{thm:4_1-full}.

\begin{proof}[Proof of Theorem~\ref{thm:4_1-full}]
We start from the formal power series~\eqref{eq:4_1z}, where
we assume that $x\in\R$. It can be written as a special case of the ${_3\phi_1}$ basic hypergeometric series:
\be\label{eq:4_1-jones-3-fi-1}
J_{{\bf 4}_1,x}(q,z)=
{_3\phi_1}\Bigl[
\begin{matrix}
 q^{1+x},q^{1-x},q\\
 0
\end{matrix};q,z/q
\Bigr]. 
\ee
We now take into account condition~\eqref{eq:n-dependent-cond} and specialize~\eqref{eq:r-hat-fi-s-n-terms}, to obtain
\begin{align}
\hat{J}_{{\bf 4}_1,x}(q,z)&=\sum_{m=0}^{n-1}(q^{1+ x},q^{1-x};q)_m\frac{z^m}{\gamma_q(-m)}\nonumber\\
&+\frac{(q^{1+ x},q^{1-x};q)_nz^n}{\gamma_q(-n)}{_3\hat{\phi}_1}\Bigl[
\begin{matrix}
 q^{n+1+x},q^{n+1-x},q\\
 0
\end{matrix};q,zq^{-n-1}
\Bigr]. \label{eq:4_1-begin-proof}
\end{align}
By using~\eqref{eq:resummed-r-fi-s}, we can write, 
\begin{align}
&\frac{{_3\hat{\phi}_1}\Bigl[
\begin{matrix}
 q^{n+1+x},q^{n+1-x},q\\
 0
\end{matrix};q,zq^{-n-1}
\Bigr]}{(q^{n+1+x},q^{n+1-x};q)_\infty\gamma_q(-n)}=\sum_{l_1,l_2\ge0}\rme{q}(q^{l_1+l_2-n}z)\frac{q^{(n+1+x)l_1+(n+1
-x)l_2}}{\gamma_q(-n)(q;q)_{l_1,l_2}}\nonumber\\
&=\sum_{l_1,l_2\ge0}\sum_{m\in\Z}\frac{\frac{\log z}{\log q}+m}{1-zq^{m}}\frac{\gamma_q(l_1+l_2-m)}{(q;q)_\infty^3}\frac{q^{mn+(1+x)l_1+(1
-x)l_2}}{(q;q)_{l_1,l_2}}\nonumber\\
&=\sum_{m\in\Z}\frac{\frac{\log z}{\log q}+m}{1-zq^{m}}\sum_{l_1\ge0}\frac{\gamma_q(-m)\gamma_q(l_1)(q^{1+l_1-m
-x};q)_\infty}{(q;q)_\infty^3}\frac{q^{mn+(1-m+x)l_1}}{(q;q)_{l_1}}\nonumber\\
&=\sum_{m\in\Z}\frac{\frac{\log z}{\log q}+m}{1-zq^{m}}\frac{\gamma_q(-m)q^{mn}(q^{1-m
-x};q)_\infty}{(q;q)_\infty^3}
{_1\phi_1}\Bigl[
\begin{matrix}
0\\
 q^{1-m-x}
\end{matrix};q,q^{1-m+x}
\Bigr],\label{eq:4_1-3-fi-1-simp}
\end{align}
where, in the third equality, we converted the summation over $l_2$ into the infinite $q$-Pochhammer symbol, and in the last equality we absorbed the summation over $l_1$ into the definition of the basic hypergeometric series ${_1\phi_1}$. Now, substituting~\eqref{eq:4_1-3-fi-1-simp} into~\eqref{eq:4_1-begin-proof}, after simplification, we obtain~\eqref{eq:hat-J-x-z-q-1-fi-1}.

\end{proof}
\section{Generating series for Habiro polynomials}\label{sec:gen-hab-ser}

We define the generating formal power series for Habiro polynomials as
\begin{equation}
\label{gHab}
\mH_K(q,z):=\sum_{m\ge0} \frac{H_{K,m}(q) }{\mg{q}{-m}}z^m,
\end{equation}
which can have zero or strictly positive radius of  convergence, as we will see in examples.

\begin{example} {\it The knot $K={\bf 3}_1^+$.} The Habiro polynomials of the positive trefoil knot are given by
\be
H_{{\bf 3}_1^{+},m}(q)=q^m\mg{q}{-m}=(-1)^mq^{m(m+3)/2}.
\ee
The series
\be
\mH_{{\bf 3}_1^+}(q,z)=\sum_{m\ge0} q^mz^m=\frac{1}{1-qz}
\ee
has radius of convergence $q^{-1}$ and no resummation is needed.\hfill$\diamond$
\end{example}

\begin{example} {\it The knot $K={\bf 3}_1^-$.} For the negative trefoil knot, we have
\be
H_{{\bf 3}_1^{-},m}(q)=H_{{\bf 3}_1^{+},m}(1/q)=\frac{1}{q^m\mg{q}{-m}}=(-1)^mq^{-m(m+3)/2}
\ee
so that the series
\be\label{eq:mh31-qz}
\mH_{{\bf 3}_1^-}(q,z)=\sum_{m\ge0} \frac{z^m}{q^{m(m+2)}}=\sum_{m\ge0} \frac{(-z)^m}{q^m\mg{q^2}{-m}}=\me{q^2}(-z/q)
\ee
has zero radius of convergence. The last equality in~\eqref{eq:mh31-qz} corresponds to Example~\ref{ex:me_q(z)}. Using the result of resummation in~\eqref{eq:q-b-l-resum-mu_q}, we obtain
%
\begin{equation}\label{eq:Borel--Laplace-resummation-3-1-minus}
\hat{\mH}_{{\bf 3}_1^-}(q,z)=\rme{q^2}(-z/q)= \frac{1}{2(q^2;q^2)_\infty^3}\sum_{m\in\Z} \mg{q^2}{-m}\frac{\frac{\log(-z)}{\log q}+2m-1}{1+zq^{2m-1}}.
\end{equation}
\hfill$\diamond$
\end{example}

\begin{example} {The knot $K={\bf 4}_1$.} The Habiro polynomials of the figure-eight knot are given in (\ref{hab-41}),
so that the series $\mH_{{\bf 4}_1}(q,z)$ has a zero radius of convergence:
\be
\mH_{{\bf 4}_1}(q,z)=\sum_{m\ge0} \frac{z^m}{\gamma_q(-m)}=\me{q}(z).
\ee
Again, by using Example~\ref{ex:me_q(z)}, we obtain
%
\be
\label{hatH41}
\hat{\mH}_{{\bf 4}_1}(q,z)=\rme{q}(z)=\frac{1}{(q;q)_\infty^3}\sum_{n\in \Z}\gamma_q(-n)\frac{\frac{\log z}{\log q}+n}{1-zq^{n}}.
\ee
\hfill$\diamond$
\end{example}

\subsection{Analytic Jones invariant} The generating series of Habiro polynomials~\eqref{gHab} 
can be regarded as a simpler object in the theory, since the Jones interpolating series (\ref{eq:cJHi}) can be written as
\begin{equation}\label{eq:J_{K,x}(q,z)-hadamard}
J_{K,x}(q,z)= \phi_{x}(q,z)\had_z \mH_K(q,z),
\end{equation}
where
\begin{align}
\phi_{x}(q,z)&:=\sum_{m\ge0}(q^{1+x},q^{1-x};q)_mz^m=
{_3\phi_2}\Bigl[
\begin{matrix}
 q^{1+x},q^{1-x},q\\
 0,0
\end{matrix}
;q,z\Bigr]\nonumber\\
&=f_{q^{1+x},0,q}(z)\had_zf_{q^{1-x},0,q}(z),\label{eq:had-fac-varphi_x}
\end{align}
and in the last equality we use the notation from~\eqref{eq:fabq(z)}.
\begin{proposition} \label{Jxb} Let $x\in\R$ be such that $|x|<1$, and assume that the $p$-Borel--Laplace resummation $\hat{\mH}_K(q,z)$ of $\mH_K(q,z)$ is well defined for some $p$. Then, the $p$-Borel--Laplace resummation $\hat{J}_{K,x}(q,z)$ of $J_{K,x}(q,z)$ is also well defined and is given by
\be\label{eq:hatJ_Kxqz-res}
\hat{J}_{K,x}(q,z)=(q^{1+x},q^{1-x};q)_\infty \sum_{r \ge 0}  q^r b_r(q,x) \hat{\mH}_K(q,q^r z),
\ee
where
\begin{equation}\label{eq:brqx}
b_r(q,x)
=\sum_{\ell=0}^r
\frac{q^{(2\ell-r)x}}{(q;q)_\ell (q;q)_{r-\ell}}.
\end{equation}
\end{proposition}
\begin{proof}
Under the assumptions on $x$, both Hadamard factors in~\eqref{eq:had-fac-varphi_x} are sm-functions, so that we can use the residue formula~\eqref{eq:res-form-q-borel-laplace-resum}, applied to~\eqref{eq:J_{K,x}(q,z)-hadamard}, to obtain
\begin{align}
\hat{J}_{K,x}(q,z)= \sum_{l_1,l_2\ge0} \hat{\mH}_K(q,zq^{l_1+l_2})\frac{(q^{1+x},q^{1-x};q)_\infty q^{(1+x)l_1+(1-x)l_2}}{(q;q)_{l_1,l_2}}\nonumber\\
= (q^{1+x},q^{1-x};q)_\infty\sum_{r\ge0} \hat{\mH}_K(q,zq^{r})q^{(1-x)r}\sum_{l_1=0}^r \frac{ q^{2xl_1}}{(q;q)_{l_1,r-l_1}},
\end{align}
which is equivalent to~\eqref{eq:hatJ_Kxqz-res}.
\end{proof}
 Below, we evaluate the analytic Jones invariant in many examples by using appropriate $p$-Borel--Laplace resummations of the Habiro generating series, together with Proposition \ref{Jxb}. We will also present alternative formulae in some examples.


\subsection{The figure-eight knot, revisited}
We can calculate $\hat J_{ {\bf 4}_1, x} (q,z)$ directly from Proposition \ref{Jxb}. We find,
\begin{equation}
\label{41-revis}
\hat J_{ {\bf 4}_1, x} (q,z)
=
\frac{(q^{1+x};q)_\infty(q^{1-x};q)_\infty}{(q;q)_\infty^3}
\sum_{r \geq0} q^r b_r(q,x)
\sum_{m\in\mathbb Z}\gamma_q(r-m)
\frac{\frac{\log z}{\log q}+m}{1-zq^{m}}.
\end{equation}
By using the equality
\be
\sum_{r\ge0}q^r b_r(q,x)\gammaq{r-m}= \gammaq{-m}(q^{1-m-x};q)_\infty
\phioneone\!\left[
\begin{matrix}
  0\\
  q^{1-m-x}
\end{matrix}
;q,q^{1-m+x}\right],
\ee
we find that \eqref{41-revis} is precisely the result obtained previously in (\ref{eq:hat-J-x-z-q-1-fi-1}) for $n=0$.

\subsection{The ${\bf 3}_1^-$ knot}
We can calculate $\hat J_{ {\bf 3}^-_1, x} (q,z)$ again from Proposition \ref{Jxb}. We obtain
\begin{equation}
\label{31-revis}
\hat J_{ {\bf 3}^-_1, x} (q,z)
=
\frac{(q^{1+x},q^{1-x};q)_\infty}{2 (q^2;q^2)_\infty^3}
\sum_{r \geq0} q^r b_r(q,x)
\sum_{m\in\mathbb Z}
\gamma_{q^2}(-m)
\frac{\frac{\log (-z)}{\log q}+r+2m-1}{1+zq^{r+2m-1}}.
\end{equation}
\subsubsection{Other representations for $\hat J_{ {\bf 3}^-_1, x} (q,z)$}
Formula~\eqref{31-revis} comes from the general approach developed in Section~\ref{sec:gen-for}. However, as in the case of the figure-eight knot, considered in detail in Section~\ref{sec:41}, there are additional representations for $\hat J_{ {\bf 3}^-_1, x} (q,z)$ that turn out to be very interesting and useful.

Let us apply the $q^2$-Borel transform directly to the series $J_{ {\bf 3}^-_1, x} (q,z)$. We have
\begin{align}
\label{31-qtransform}
\tilde{J}_{{\bf 3}_1^-,x}(q,z)&:=(\mathcal{B}_{q^2} J_{{\bf 3}_1^-,x}(q,\cdot))(z)=\sum_{m=0}^\infty \frac{(q^{1+x},q^{1-x};q)_m}{(q,-q;q)_m}(z/q)^m\nonumber\\
&={_2\phi_1}\!\left[
\begin{matrix}
 q^{1+x},q^{1-x}\\
 -q
\end{matrix}
 ;q,z/q\right].
\end{align}
We note that $\tilde{J}_{{\bf 3}_1^-,x}(q,z)$ has poles at $z=q^{1-m}$, $m\in\Z_{\ge0}$. Therefore, the $q^2$-Borel--Laplace resummation is ambiguous for $z>0$. This ambiguity will manifest itself in the choice of branch of $\log(-z)$.

We now proceed as in the case of the figure-eight knot in Section~\ref{sec:41}. Using Watson's contour integral formula~\eqref{eq:watson-ac}, we obtain
\begin{equation}\label{eq:watson-trefoil}
\tilde{J}_{{\bf 3}_1^-,x}(q,z)=\frac{(q^{1+x},q^{1-x};q)_\infty}{(q,-q;q)_\infty}\left(\frac{-1}{2\pi \ri}\right)\int_C\frac{(q^{1+s},-q^{1+s};q)_\infty}{(q^{1+s+x},q^{1+s-x};q)_\infty}\frac{\pi(-z/q)^s}{\sin(\pi s)}\rd s,
\end{equation}
where the contour $C$ runs upward along the imaginary axis, deviating in the finite part to separate the poles at $s\in\Z_{\ge 0}$ from all other poles of the integrand. The condition for the convergence of the integral is $|\arg(-z)|<\pi$. Thus,
\begin{align}
 &\hat{J}_{{\bf 3}_1^-,x}(q,z):=\frac{-1}{2\log q}(\mathcal{L}_{q^2}\tilde{J}_{{\bf 3}_1^-,x}(q,\cdot))(z)=\frac{-1}{2\log q}\int_0^\infty\frac{\tilde{J}_{{\bf 3}_1^-,x}(q,tz)}{(-t;q^2)_\infty}\rd t\nonumber\\
& =
 \frac{-1}{2\log q}\frac{(q^{1+x},q^{1-x};q)_\infty}{(q^2;q^2)_\infty}\left(\frac{-1}{2\pi \ri}\right)\int_C\frac{(q^{2+2s};q^2)_\infty}{(q^{1+s+x},q^{1+s-x};q)_\infty}\frac{\pi R_s(q^2)(-z/q)^s}{\sin(\pi s)}\rd s,
\end{align}
where we assume that $|\arg(-z)|<\pi$ and use notation~\eqref{eq:t-s-int}.
Substituting the explicit form of $R_s(q^2)$ from~\eqref{eq:t-s-int}, we obtain
\begin{align}
 \hat{J}_{{\bf 3}_1^-,x}(q,z)&= \frac{\ri\pi}{4\log q}\frac{(q^{1+x},q^{1-x};q)_\infty}{(q^2;q^2)_\infty^2}\int_C\frac{(q^{2+2s},q^{-2s};q^2)_\infty}{(q^{1+s+x},q^{1+s-x};q)_\infty}\frac{(-z/q)^s}{(\sin(\pi s))^2}\rd s\nonumber\\
&= \frac{\pi}{4\log q}\frac{(q^{1+x},q^{1-x};q)_\infty}{(q^2;q^2)_\infty^2}\int_{\R-\ri\epsilon}\frac{(q^{2\ri s},q^{2-2\ri s};q^2)_\infty}{(q^{\ri s+x},q^{\ri s-x};q)_\infty}\frac{(-z/q)^{\ri s-1}}{(\sinh(\pi s))^2}\rd s,\label{eq;3-1-minus-cont-int}
 \end{align}
 where, in the last equality, we also assume that $0<\epsilon<1$ and $|x|<\epsilon$.

  %
 %
 \begin{theorem} We have
 \be
 \label{anal3J}
 \ba
 \hat{J}_{{\bf 3}_1^-,x}(q,z)&=\frac{q\,U_q(z,q^x)}{1-q^x}-
\frac{\pi ^2 q}{2z(\log q)^2(q;q)_{\infty } (q^2;q^2)_{\infty }^3}\frac{ \theta_q(q^x) \theta_{q^2}(q^{2
   x})  }{
   (1-q^x)
   (\frac{q}{z};q)_{\infty } \theta_q(q^{2x})} \\
  & \times\sum_{a\in\{\pm1\}}\fn{2,q}(-z/q,ax)
{_1\phi _1}\bigl[
\begin{matrix}
  q/z\\
 0
\end{matrix}
 ;q,q^{1+2ax}\bigr],
   \ea
   \ee
  where
  \be
  U_q(z,X)=
  \frac{X}{z(1-X)}\,
  {_3\phi_2}\Bigl[
\begin{matrix}
 0,0,q\\
 qX,q/X
\end{matrix}
 ;q,\frac{q}{z}\Bigr]=
  \frac{X}{z(1-X)}\sum_{n\geq0}\frac{(q/z)^n}{(qX,q/X;q)_n}\label{Uq}
  \ee
and  the function $\fn{2,q}(z,x)$ is given in (\ref{ident2}).
\end{theorem}

\begin{proof} The proof is almost identical to that of Theorem \ref{thm2}. 
We evaluate the integral in~\eqref{eq;3-1-minus-cont-int} by residues in the upper half-plane.
The contribution from poles at $s\in\ri \Z_{\ge 0}$ is
\be
\frac{q\,U_q(z,q^x)}{1-q^x}, \label{contrib1b}
\ee
with $U_q(z,X)$ defined by \eqref{Uq},
while for the contributions from the poles $s=\pm \ri x +\ri n +\frac{2\pi m}{h}$, we use the identity
\be
{_2\phi_1}\Bigl[
\begin{matrix}
 0,0\\
c
\end{matrix}
 ;q,z\Bigr]=\frac{1}{(c,z;q)_\infty}\,
{_1\phi_1}\Bigl[
\begin{matrix}
 z\\
 0
\end{matrix}
;q,c\Bigr],
\ee
obtained as the limit $a,b\to0$ of formula~(III.1) of Appendix III in \cite{MR2128719}, where we substitute $c\mapsto q^{1\pm 2x}$ and $z\mapsto q/z$. \end{proof}

As in the case of the figure-eight knot, let us specialize~\eqref{anal3J} to $z=1$. Define the functions
\be
H_q(z)= {_1\phi_1}\Bigl[
\begin{matrix}
 {q}\\0
\end{matrix}
;q,z\Bigr]=\sum_{n\geq0}\gamma_q(n)z^n,\quad T_q(z)=H_q(z)-\frac{1}{2}\theta_q(z).\label{HT-def}
\ee

\begin{corollary} \label{cor-31}   
  \be
  \label{finalJ3}
  \hat{J}_{{\bf 3}_1^-,x}(q)=A_{{\bf 3}^-_1,x}(q) \pm {\pi \ri \over 2 h} B_{{\bf 3}^-_1,x}(q),
  \ee
where
\be
\label{AJ3}
 A_{{\bf 3}^-_1,x}(q)=
\frac{qU_q(q^x)}{1-q^x}-\frac{q T_q(q^{1-2x})}{(1-q^x)P_q(q^x)},
\ee
with
\be
P_q(X)=(X,q/X;q)_\infty,\label{P-def}
\ee
\be\label{eq:Uq1X}
U_q(X):= U_q(1,X),
\ee
see \eqref{Uq}, 
  and
  \be
  \label{BJ3}
  B_{{\bf 3}^-_1,x}(q)=\frac{ (q^{1+x},q^{1-x},q;q)_\infty\theta_q(-q^x)^2}{2 (q^2;q^2)_\infty^4 q^{x-1}}.
  \ee
  \end{corollary}
  See Appendix~\ref{app:B} for the proof.
  
Let us note that the first term in the r.h.s. of (\ref{AJ3}) is the inverted Habiro series of the ${\bf 3}_1^-$ knot introduced in \cite{park-inverted}. As in the case of the figure-eight knot, this is obtained in our calculation from the poles at $s=\ri \IZ_{\ge 0}$. The singularities at $x=n$, $n \in \IZ_{\ge 0}$, are removed by the second term in the r.h.s. of (\ref{AJ3}).

In \figref{31-fig} we plot the real and imaginary parts of (\ref{finalJ3}) for $q=1/2$, and compare it to the values of the colored Jones polynomial for $n=1,2,3$.  
 \begin{figure}[!ht]
\leavevmode
\begin{center}
\includegraphics[height=4cm]{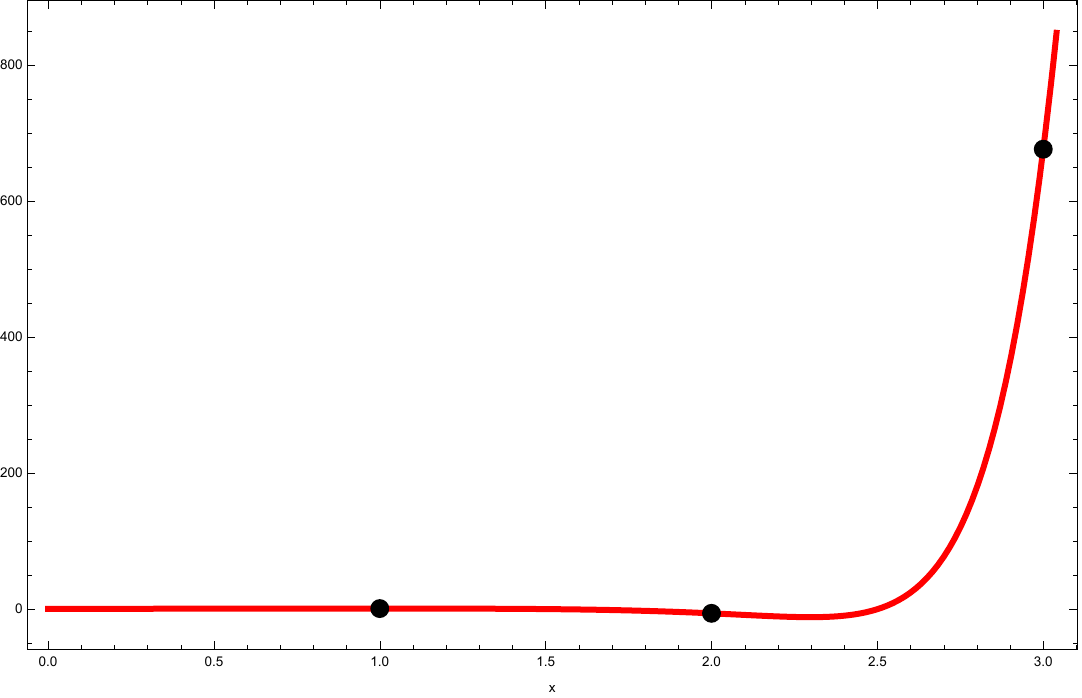}\qquad \qquad  \includegraphics[height=4cm]{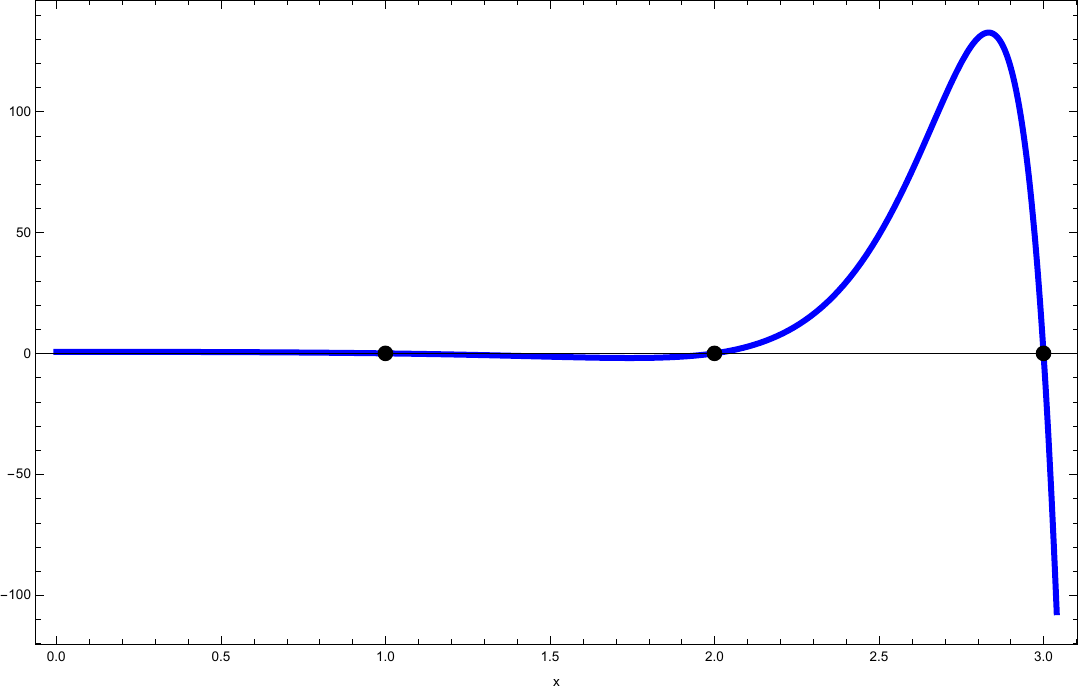}
\end{center}
\caption{Left-hand side: the red line is the real part of the interpolating Jones invariant $\hat J_{{\bf 3}_1^-,x}(q)$ as a function of $x$ for $q=1/2$. The black dots are the corresponding values $1, -6, 676$ of the colored Jones polynomial for $n=1,2,3$. Right-hand side: the blue line is the imaginary part of the interpolating Jones invariant $\hat J_{{\bf 3}_1^-,x}(q)$ for $q=1/2$. The black dots at $x=1,2,3$ are on the $x$-axis and show that the imaginary part vanishes, as required.}
\label{31-fig}
\end{figure}


\subsubsection{Formal asymptotic expansion as $q\to 1$ and comparison to Borel resummation}
As in the case of the figure-eight knot, the interpolating series (\ref{eq:cJHi}) leads to a formal power 
series in $h$ for the trefoil knot, $\Phi^{{\bf 3}^-_1}_0 \left( h; w ,z \right)$. Its very first terms are
 \be
\Phi^{{\bf 3}_1^-}_0(h;w,z)= {1\over 1+ (X+ X^{-1}-2) z} -\frac{\left(X+X^{-1}-2\right) z \left(\left(X+X^{-1}-2\right) z+2\right)}{\left(\left(X+X^{-1}-2\right) z+1\right)^3} h + \CO(h^2).
  \ee
When $z=1$, $\Phi^{{\bf 3}_1^-}_0(h;w)\equiv \Phi^{{\bf 3}_1^-}_0(h;w,1)$ gives the usual MMR series of the trefoil knot.
We can use the integral expression (\ref{eq;3-1-minus-cont-int}) and the asymptotic formula (\ref{qpoch-z}) to obtain the formal asymptotic
expansion of $ \hat{J}_{{\bf 3}_1^-,x}(q,z)$ as $h\to 0^+$, where $q=\re^{-h}$, as we did for the figure-eight knot in Subsection \ref{prop-as}. One finds,
\be
\hat{J}_{{\bf 3}_1^-,x}(q,z) \sim \Phi^{{\bf 3}_1^-}_0(h;w,z), \qquad h \to 0,
\ee
where $w$, given by~\eqref{uu}, is fixed and
\be
\label{trefoil-MMR}
\ba
\Phi^{{\bf 3}_1^-}_0(h;w,z)&=\sum_{k \ge 0} z^{k} \left( 2 \sinh(w/2) \right)^{2k} \gamma_{\re^h}(-k)\re^{kh}
 \\
 & \times \exp \left\{ 2 \sum_{\ell \ge 1} {B_{2 \ell +1}(k+1) \over (2\ell+1)!} h^{2\ell} \left( {\rd \over \rd w} \right)^{2 \ell} \log\left( 2 \sinh(w/2) \right) \right\},
 \ea
 \ee
 which is the deformed MMR series defined by the procedure explained in 
 Subsection \ref{prop-as}.

Let us now compare the result of the $q^2$-Borel--Laplace resummation to the standard Borel resummation of the MMR series, which was studied in some detail in \cite{MR4233201}. We first note the following useful expression \cite{rozansky,MR4233201}:
\be
\label{phi-GM}
\ba
\Phi^{{\bf 3}_1^-}_0 (h; w)&= {q^{23 \over 24} \over X^{1\over 2} - X^{-{1\over 2}}}  \sum_{m \ge 0} {1\over m!} \left(-{h\over 24} \right)^m
\left( {\partial \over \partial y} \right)^{2m} \left[ {\re^y X^{1\over 2} - \re^{-y} X^{-{1\over 2}} \over 1+ \left(\re^y X^{1\over 2} - \re^{-y} X^{-{1\over 2}} \right)^2 } \right]_{y=0}.
\ea
\ee
This suggests introducing the modified MMR series
\be
\Psi_0(h;w)=  {1\over {\sqrt{h}}} q^{-{23 \over 24}} \Phi^{{\bf 3}_1^-}_0(h;w) =\sum_{m \ge 0} a_m(w) h^{m-1/2}.
\ee
As noted in \cite{gmp} in a similar context, it is more convenient in this case to use the modified Borel transform
\be
\widehat \Psi_0 (\zeta;w)= \sum_{m \ge 0} {a_m (w) \over \Gamma\left(m+{1\over 2} \right)} \zeta^{m-1/2},
\ee
since it can be computed in closed form as
\be
\widehat \Psi_0(\zeta;w)= {1\over 2 {\sqrt{\pi \zeta}}} {1\over X^{1/2}- X^{-1/2}} \left[ \mathfrak{f}\left( \ri {\sqrt{\zeta \over 6}};X \right) -
\mathfrak{f}\left( \ri {\sqrt{\zeta \over 6}}; X^{-1} \right)  \right],
\ee
where
\be
\mathfrak{f}(y; X)= {\re^y X^{1/2} - \re^{-y}X^{-1/2} \over  1+ (\re^y X^{1/2} - \re^{-y}X^{-1/2})^2}.
\ee
In terms of this modified Borel transform, the lateral Borel resummations are given by
\be
s_{\pm} \left( \Phi^{ {\bf 3}_1^-}_0 \right)(h;w)= q^{{23 \over 24}} {\sqrt{h}} \int_{0}^{\re^{ \pm \ri  \epsilon} \infty}  \re^{- \zeta} \widehat \Psi_0 (h \zeta;w ) \rd \zeta.
\ee
The modified Borel transform has simple poles at the points
\be
\label{tre-poles}
\zeta^\pm_k= {\pi^2 \over 6} \left(1+ 6k \pm {3 \ri w \over \pi}  \right)^2, \qquad k \in \IZ.
\ee
We will assume that $w=0$ or is purely imaginary but small in absolute value, so that 
the poles (\ref{tre-poles}) lie on the positive real axis in the $\zeta$ plane (when $w=0$ we have double poles). Let us 
then consider the lateral Borel resummations of the MMR series. The real part of these resummations is unambiguous and it can 
be computed in terms of a series obtained in \cite{MR4233201}:
\be
\label{gm-trefoil}
f(q,X)= {q\over 2} \sum_{m \ge 0}  \varepsilon (m) { X^{m/2}-X^{-m/2} \over X^{1/2}-X^{-1/2}} q^{m^2-1 \over 24},
\ee
where
\be\label{eq:varepsilon(m)}
\varepsilon(m)= \begin{cases} -1, & \text{if $m\equiv 1, 11$ mod $12$}, \\
1 , & \text{if $m\equiv 5, 7$ mod $12$},\\
0 &\text{otherwise}
\end{cases}.
\ee
We can represent nonzero contributions to $\varepsilon(m)$ as $m=12k\pm1$ and $m=12k\pm5$, $k,m\geq0$.
Summing over the admissible values of $k$, one can  obtain an alternative formula for $f(q,X)$:
\be
f(q,X)=\frac{q}{1-X}\left(T_{q^3}(qX^3)+X\,T_{q^3}(q^2X^3)-1\right),\label{GM-exp}
\ee
where $T_q(X)$ is defined by \eqref{HT-def}.

One then has,
\be
{\rm Re}\, s_\pm \left( \Phi^{ {\bf 3}_1^-}_0 \right)(h;w)=f(q, X).
\ee
There is in addition a non-zero imaginary part whose sign depends on the choice of lateral resummation. It can be obtained by summing over the residues of the poles (\ref{tre-poles}), and one obtains, for $|w|$ sufficiently small,
\be
\label{ims}
{\rm Im}\,  s_\pm \left( \Phi^{ {\bf 3}_1^-}_0 \right)(h;w)=\mp {\ri \over 2} q^{23\over 24} {\sqrt{2 \pi \over h}} \re^{{3 w^2 \over 2h}} \sum_{m \ge 0} \varepsilon(m) {\widetilde X^{m/2} - \widetilde X^{-m/2}\over X^{1/2}- X^{-1/2}} \tilde q^{m^2 \over 24},
\ee
where $\widetilde X$ is defined in (\ref{tX-def}).

Let us now compare these results with the result (\ref{finalJ3}), (\ref{AJ3}), (\ref{BJ3}) for the interpolating Jones invariant.
\begin{proposition}\label{prop3}
\be
A_{{\bf 3}^-_1, x}(q)= f(q, X), \quad X=q^x, \label{A-prop}
\ee
where $A_{{\bf 3}^-_1, x}(q)$ is defined in \eqref{AJ3} and $f(q,X)$ in~\eqref{gm-trefoil}.
\end{proposition}
The proof is given in Appendix~\ref{app:B}.

Therefore, the $(q,X)$-series $f(q, X)$, which was proposed in \cite{MR4233201} as the Gukov--Manolescu 
invariant of the knot ${\bf 3}_1^-$, agrees in this case with the real part of the interpolating Jones invariant. The series 
(\ref{BJ3}) gives an additional contribution to this invariant depending on the choice of the branch of $\log(-z)$. By
performing an $S$-duality transformation on (\ref{BJ3}), we can write it as
 \be
 B_{{\bf 3}^-_1, x}(q)= 2 \ri q^{23 \over 24} {\sqrt{ h \over 2 \pi}} \re^{{3 w^2 \over 2h}}  \tq^{1\over 24}  {\tX^{1/2}- \tX^{-1/2} \over X^{1/2}- X^{-1/2}} {(\tq^{1/2}  \tX, \tq^{1/2} \tX^{-1} ; \tq)^2_\infty \over (\tq^{1/2}; \tq)_\infty ^4} {(\tq  \tX, \tq \tX^{-1} ; \tq)_\infty \over (\tq ; \tq)_\infty },
 \ee
and the imaginary part of the interpolating Jones invariant (\ref{finalJ3}) is
\be
\pm { \pi \over 2h}  B_{{\bf 3}^-_1, x}(q)=   \pm {\ri \over 2} q^{23 \over 24} {\sqrt{ 2 \pi \over h}} \re^{{3 w^2 \over 2h}}{\tX^{1/2}- \tX^{-1/2} \over X^{1/2}- X^{-1/2}} \tq^{1\over 24}\left( 1+ \CO(\tq^{1/2}) \right).
 \ee
This agrees with (\ref{ims}) at leading order in $\tq$ but differs at higher orders. This is particularly clear in the case $w=0$, where one has
\be
{\rm Im}\,  s_\pm \left( \Phi^{{\bf 3}^-_1}_0 \right)(h)=\mp {1\over 2} q^{{23 \over 24}} \left({2 \pi \over h} \right)^{3/2}  \sum_{m \ge 1} m\, \varepsilon(m) \tilde q^{ m^2 \over 24},
\ee
see~\eqref{eq:varepsilon(m)} for the definition of $\varepsilon(m)$, while the imaginary part of the interpolating Jones invariant is given by 
\be
\label{amb-qborel}
\mp {1\over 2} q^{{23 \over 24}} \left({2 \pi \over h} \right)^{3/2}  \sum_{m \ge 1}  \varepsilon(m) \tilde q^{ m^2 \over 24}, 
\ee
for two choices of branch of $\log(-1)$. We finally note that, just as in standard Borel resummation, (\ref{BJ3}) can also be calculated by summing the residues of the poles of the $q^2$-Borel transform (\ref{31-qtransform}) along the positive real axis.

\subsection{The ${\bf 5}_2$ knot}
This subsection is devoted to another example, the knot $K={\bf 5}_2$, which is more complicated than the examples we have already analyzed. 

The Habiro invariants of the ${\bf 5}_2$ knots were calculated in~\cite{MR1997328}, with the result
\be
H_{{\bf 5}_2^\pm,m}(q)=(-1)^m q^{\pm\frac{m(m+3)}{2}}
h_m(q^{\pm1}),
\ee
where
\be
\label{hm-def}
h_m(q):=\sum_{k=0}^mq^{k(k+1)}\qbinom{m}{k} 
\ee
and, as in the case of the trefoil knots, the sign~$\pm$ corresponds to the sign of the knot signature, that is, $\sigma({\bf 5}_2^{\pm}) = \pm 2$. 

The series $\mH_{{\bf 5}_2^+}(q,z)$ has radius of convergence $q^{-1}$ which can be seen from the formula:
\begin{equation}
\mH_{{\bf 5}_2^{+}}(q,z)=\sum_{k\ge0} \frac{q^{k(k+2)}}{(qz;q)_{k+1}}z^{k}.
\end{equation}
On the other hand, for the knot ${\bf 5}_2^{-}$ we have
\be
\mH_{{\bf 5}_2^-}(q,z)=\sum_{m\ge0}q^{-m(2m+3)} g_{m}(q)z^m,
\ee
where
\begin{equation}\label{eq:g2m(q)-def}
g_{m}(q):=\sum_{k=0}^m\qbinom{m}{k}q^{k(m+1)}.
\end{equation}
The series $\mH_{{\bf 5}_2^-}(q,z)$ has zero radius of convergence. Its coefficients grow $\sim q^{-2 m^2}$, so one requires a $q^4$-Borel--Laplace resummation.

For the $q^4$-Borel transform, we have
\begin{equation}\label{eq:tildeH52-qz}
\tilde{\mH}_{{\bf 5}_2^-}(q,z):=\mathcal B_{q^4}\mH_{{\bf 5}_2^-}(q,z)=\sum_{m\ge0} \frac{g_{m}(q)(z/q)^m}{(q^4;q^4)_m}=\sum_{m\ge l\ge0} \qbinom{m}{l}\frac{q^{(m-l)(m+1)}(z/q)^m}{(q^4;q^4)_m}.
\end{equation}
\begin{proposition}
 The following Hadamard decomposition formula holds
 \begin{equation}\label{eq:triple-decomp-5-2}
\tilde{\mH}_{{\bf 5}_2^-}(q,z)=\frac{1}{(\frac zq;q^4)_{\infty}}\had_z\Big(\sum_{n\ge0}(q;q)_nz^n\Big)\had_zF_q(z),
\end{equation}
where $F_q(z)$ was introduced in (\ref{eq:main-non-sm-func}).
\end{proposition}
\begin{proof}
 We start with the immediate factorisation
 $$
 \tilde{\mH}_{{\bf 5}_2^-}(q,z)=\Big(\sum_{k\ge0}\frac{(z/q)^k}{(q^4;q^4)_k}\Big)\had_z\Big(\sum_{n\ge0}(q;q)_nz^n\Big)\had_z\Big(\sum_{m\ge l\ge0} \frac{q^{(m-l)(m+1)}}{(q;q)_{l,m-l}}z^m\Big).
 $$
 Then, for the last factor, assuming that $|z|<1$, we write
 $$
\sum_{m\ge l\ge0} \frac{q^{(m-l)(m+1)}}{(q;q)_{l,m-l}}z^m=\sum_{l,m\ge0} \frac{q^{m(l+m+1)}}{(q;q)_{l,m}}z^{l+m}=\sum_{m\ge0} \frac{q^{m(m+1)}}{(q;q)_m(q^mz;q)_{\infty}}z^{m}=F_q(z).
 $$
\end{proof}
We now construct a Watson-type contour integral representation for the \(p\)-Laplace transform of \(F_q(z)\), leaving, for the time being, the relation between \(p\) and \(q\) unspecified.
\begin{lemma}\label{lem:div-lap-int} The $p$-Laplace transform of $F_q(z)$ can be expressed as
 \begin{equation}\label{eq:div-ser-expr}
 \mathcal L_{p} F_q(z)=\int_{\ri\R-\epsilon}\frac{S_{p,q}\left(p^{1+s},z (q/p)^{s}\right)(-z)^{s}\rd s}{(p;p)_{-1-s}(q;q)_{s}\frac{2\ri}{\pi}\sin^2(\pi s)},
\end{equation}
where
 \begin{equation}\label{eq:cpqxy-function}
 S_{p,q}(x,y):=\sum_{m\ge0}\frac{(x;p)_{m}y^m}{(q;q)_{m}(p/q^2)^{m(m+1)/2}}.
\end{equation}
\end{lemma}
See Appendix~\ref{app:B} for the proof.

It follows from~\eqref{eq:cpqxy-function} that, for $p=q^4$, the series $S_{p,q}(x,y)$ is divergent. We therefore apply the same deformed Borel--Laplace resummation technique to $S_{q^4,q}(x,y)$. Let us note that this second application is performed in a non-canonical way, so we cannot regard this iteration of the resummation as part of a well-defined general procedure.

\begin{lemma}\label{lem:bl-resum-c}
The $q^2$-Borel--Laplace resummation of $S_{q^4,q}(x,y)$ is as follows:
\begin{equation}\label{eq:hat-cq4qxy}
\hat S_{q^4,q}(x,y)=\sum_{n\ge0}\sum_{l\in\Z}\sum_{k\ge 4n-2l}\frac{(x;q^4)_{\infty}x^n\mg{q^2}{l}\mg{q}{4n-2l-k}}{(q^2;q^2)_{\infty}^3(q^4;q^4)_n(q;q)_{k+2l-4n,\infty}}\frac{\frac{\log(-yq^{k})}{2\log q}}{yq^{k}+1}.
\end{equation}
\end{lemma}
See Appendix~\ref{app:B} for the proof.

\subsubsection{The rectified $q^4$-Laplace transform of $F_q(z)$}
Putting together Lemmata~\ref{lem:div-lap-int} and \ref{lem:bl-resum-c}, we define the ``rectified'' $q^4$-Laplace transform of $F_q(z)$ as follows:
 \begin{equation}\label{eq:div-ser-expr-rectified}
 \mathcal L^{\text{rec}}_{q^4} F_q(z)=\int_{\ri\R-\epsilon}\frac{\hat{S}_{q^4,q}\bigl(q^{4+4s},z q^{-3s}\bigr)(-z)^{s}\rd s}{(q^4;q^4)_{-1-s}(q;q)_{s}\frac{2\ri}{\pi}\sin^2(\pi s)},
\end{equation}
where $\hat{S}_{q^4,q}(x,y)$ is given by~\eqref{eq:hat-cq4qxy}. It can be rewritten in the form
 \begin{equation}
 \mathcal L^{\text{rec}}_{q^4} F_q(z)=\int_{\ri\R-\epsilon}\frac{(q^{1+s};q)_{\infty}D_q(s,z)\theta_{q^4}(q^{-4s})
 (-z)^{s}\rd s}{2(q;q)_{\infty}^2(q^2;q^2)_{\infty}^3(q^4;q^4)_\infty^2\frac{2\ri}{\pi}\sin^2(\pi s)},\label{Lrecq4}
\end{equation}
where
\begin{equation}
D_q(s,z):=\sum_{k,l \in \IZ,n \ge 0}\frac{q^{4n(1+s)}\mg{q^2}{l+2n}\mg{q}{-2l-k}}{(q^4;q^4)_n(q;q)_{k+2l}}
\frac{\frac{\log(-z)}{\log q}+k-3s}{zq^{k-3s}+1}.\label{Dqsz}
\end{equation}
\begin{lemma}\label{Lq4F}
The ``rectified'' Laplace transform ${\mathcal L}^{\text{rec}}_{q^4} F_q(z)$ can be written as
\begin{equation}
{\mathcal L}^{\text{rec}}_{q^4} F_q(z)=\mathcal P_q\sum_{e,k\in\mathbb Z}A_{e,k}(q)I_{e,k}(z),
\label{eq:LF_AJ}
\end{equation}
where
\be
\mathcal P_q:=
\frac{1}{2(q;q)_\infty^2(q^2;q^2)_\infty^3(q^4;q^4)_\infty^2},
\ee
\begin{equation}
A_{e,k}(q)=
\sum_{n\ge0}\sum_{j\in\mathbb Z}\sum_{l\in\mathbb Z}
\frac{
\gamma_q(e+4j)q^{e+4j+4n}
\gamma_{q^4}(j+n)
\gamma_{q^2}(l+2n)
\gamma_q(-2l-k)
}{
(q;q)_{e+4j}(q^4;q^4)_n(q;q)_{k+2l}
},
\label{eq:Atriple}
\end{equation}
\begin{equation}
I_{e,k}(z)=
\frac{\pi}{2\ri}
\int_{-\varepsilon-\ri\infty}^{-\varepsilon+\ri\infty}
\frac{(L+k-3s)(-z)^s q^{es}}
{(1+zq^{k-3s})\sin^2(\pi s)}\,\rd s,
\label{eq:Jekintegral}
\end{equation}
and
\be\label{eq:def-L}
L:={\log(-z) \over \log(q)}.
\ee
\end{lemma}
\begin{proof}
Let us substitute \eqref{Dqsz} in \eqref{Lrecq4}, expand $(q^{1+s};q)_{\infty}$ and
$\theta_{q^4}(q^{-4s})$ into series and swap summations and integration.

The integral can be rewritten in the form \eqref{eq:Jekintegral} and coefficients $A_{e,k}$ are given by
\begin{equation}
A_{e,k}(q):=
\sum_{\substack{p,n\ge0,\;j,l\in\mathbb Z\\p-4j+4n=e}}
\mathcal C_{p,n,j,k,l},
\label{eq:Adef}
\end{equation}
with
\be
\mathcal C_{p, n, j, k, l}:=
\frac{
\gamma_q(p)q^{p+4n}\,
\gamma_{q^4}(j)\,
\gamma_{q^2}(l+2n)\,
\gamma_q(-2l-k)
}{
(q;q)_p\,(q^4;q^4)_n\,(q;q)_{k+2l}
}.
\ee
Eliminating $p$ by $p=e+4j-4n$ and shifting
$j\to j+n$, we come to \eqref{eq:Atriple}.
\end{proof}
\begin{remark}
 The  integral \eqref{eq:Jekintegral} satisfies the following invariance property:
\begin{equation}
I_{e,k}(q^d z)=I_{e+d,k+d}(z),\quad \forall d\in\mathbb Z.
\label{eq:Jshift}
\end{equation}
\hfill$\diamond$
\end{remark}
Let us consider
\be
{\hat\mH}_{{\bf 5}_2^-}(q,z) =\frac{1}{4h}\mathcal{L}_{q^4}( \tilde{\mH}_{{\bf 5}_2^-}(q,\cdot))(z),
\ee
where, from now on, we drop the superscript ``rect'' in the $q^4$-Laplace transform.

\begin{lemma}\label{lem:hatH52-BI}
The following identity holds:
\be\label{eq:H52B}
{\hat\mH}_{{\bf 5}_2^-}(q,z)=
\mathcal P_q\frac{(q;q)_\infty}{4h(q^4;q^4)_\infty}
\sum_{e,k\in\mathbb Z}B_{e,k}(q)I_{e,k}(z),
\ee
where
\begin{equation}
B_{e,k}(q)=
\sum_{j,l\in\mathbb Z}
\frac{
q^{e+1+4j}
\gamma_{q^4}(j)
\gamma_{q^2}(l)
\gamma_q(e-k+4j-2l)
}{
(q;q)_{e+1+4j}(q;q)_{k+1+2l}
}.
\label{eq:Bdouble}
\end{equation}
\end{lemma}
See Appendix~\ref{app:B} for the proof.

Now, we evaluate the contour integral~\eqref{eq:Jekintegral}.
In the following formulas, $\log$ denotes the branch used to define $v^s=\exp(s\log v)$, consistently with~\eqref{eq:def-L}. 

Defining
\begin{equation}
I_n(v):=
\frac{\pi}{2\ri}
\int_{-\varepsilon-\ri\infty}^{-\varepsilon+\ri\infty}
\frac{s^n v^s}{\sin^2(\pi s)}\,\rd s,
\label{eq:Indef}
\end{equation}
we have
\begin{equation}
I_0(v)=-\frac{\log v}{1-v},
\qquad
I_1(v)=-\frac1{1-v}-\frac{v\log v}{(1-v)^2}.
\label{eq:I01}
\end{equation}
It is convenient to define the function
\begin{equation}
\Phi_k(v):=(L+k)I_0(v)-3I_1(v)
=\frac{3-(L+k)\log v}{1-v}
+\frac{3v\log v}{(1-v)^2}.
\label{eq:Phik}
\end{equation}
Let us assume that
\be
q^{1-3\varepsilon}<|z|\leq1. \label{modz}
\ee
Then, on the contour $\Re s=-\varepsilon$,
\be
|zq^{k-3s}|<1\quad \mbox{for}\quad k\geq0,\quad \mbox{and} \quad |zq^{k-3s}|>1\quad \mbox{for}\quad k<0.
\ee
Expanding the denominator in \eqref
{eq:Jekintegral} in geometric series, we obtain
\begin{equation}
I_{e,k}(z)=
\sum_{r=0}^{\infty}(-zq^k)^r
\Phi_k\!\left((-z)q^{e-3r}\right),
\qquad k\ge0
\label{eq:Jplusphi}
\end{equation}
and
\begin{equation}
I_{e,k}(z)=
-\sum_{r=1}^{\infty}(-zq^k)^{-r}
\Phi_k\!\left((-z)q^{e+3r}\right),
\qquad k<0.
\label{eq:JminusPhi}
\end{equation}

Note that if the variable $z$ does not satisfy \eqref{modz}, then we need to split the sum
at some value $k=k_0$ depending on the  value of $|zq^{-3s}|$. However, in the next sections
we will combine two pieces \eqref{eq:Jplusphi} and \eqref{eq:JminusPhi} back in one sum using the method
of generating functions. So the condition \eqref{modz} can be removed in the final answer.

\subsubsection{Analytic expression for $\hat{J}_{{\bf 5}_2^-,x}(q,z)$}
From Proposition~\ref{Jxb} we have
\be
\hat{J}_{{\bf 5}_2^-,x}(q,z)=(q^{1+x},q^{1-x};q)_\infty \sum_{r \ge 0}  q^r b_r(q,x) \hat{\mH}_{{\bf 5}_2^-}(q,q^r z),\label{J52sum}
\ee
where $b_r(q,x)$ is given by~\eqref{eq:brqx}.
%
Combining \eqref{eq:H52B}, \eqref{J52sum} and using \eqref{eq:Jshift}, we obtain
\begin{equation}
\hat{J}_{{\bf 5}_2^-,x}(q,z)
=\mathcal P_q\frac{(q;q)_\infty}{4h(q^4;q^4)_\infty}
(q^{1+x},q^{1-x};q)_\infty
\sum_{e,k\in\mathbb Z}C_{e,k}(x,q)I_{e,k}(z),
\label{eq:J52C}
\end{equation}
where
\begin{equation}
C_{e,k}(x,q)=
\sum_{r=0}^{\infty}\sum_{m=0}^{r}
\frac{q^{r(1-x)+2xm}}
{(q;q)_m(q;q)_{r-m}}
B_{e-r,k-r}(q).
\label{eq:CdoubleB}
\end{equation}

Substituting the explicit double sum \eqref{eq:Bdouble} gives
\begin{equation}
\begin{aligned}
C_{e,k}(x,q)
={}&\sum_{r=0}^{\infty}\sum_{m=0}^{r}\sum_{j,l\in\mathbb Z}
\frac{
q^{e+1+4j+x(2m-r)}
\gamma_{q^4}(j)\gamma_{q^2}(l)
\gamma_q(e-k+4j-2l)
}{
(q;q)_m(q;q)_{r-m}
(q;q)_{e-r+1+4j}(q;q)_{k-r+1+2l}
}.
\end{aligned}
\label{eq:CexplicitB}
\end{equation}
\begin{theorem}\label{thm:analytic-form-hatJ52} Let complex variables $U$ and $V$ satisfy the conditions
\be\label{eq:condUV}
q<\bigl|\frac{z}{U^3}\bigr|<1<|V|<q^{-1},\quad |UV|<q^{|x|-1}.
\ee
Then the following equality holds:
 \begin{equation}
\begin{aligned}
&\hat{J}_{{\bf 5}_2^-,x}(q,z)
=
\frac{(q^{1+x},q^{1-x};q)_\infty}{8(q;q)_\infty(q^2;q^2)_\infty^3(q^4;q^4)_\infty^3}
\CT_{U,V}
\bigl[
\CC_{q,x}(U,V)\Theta_q(U,V;z)
\\
&\times
\bigl\{
\frac{3}{h}+
\bigl(L-\Psi_q\bigl(\frac{1}{U},\frac{-z}{U}\bigr)\bigr)\bigl(L-\Psi_q\bigl(\frac{1}{V},\frac{-z}{U^3V}\bigr)-3\Psi_q\bigl(\frac{-z}{U},-z\bigr)\bigr)
-3\Lambda_q\bigl(\frac{-z}{U}\bigr)
\bigr\}
\bigr], 
\end{aligned}
\label{eq:J52unexpanded}
\end{equation}
where the operator $\CT_{U,V}[f(U,V)]$ stands for the constant term in the Laurent expansion in variables $U$ and $V$ of a
function $f(U,V)$ in the domain determined by the conditions~\eqref{eq:condUV},
\begin{equation}
\mathcal C_{q,x}(U,V)=
\frac{
U^{-1}V^{-1}
\theta_{q^4}(U^{-4})\theta_{q^2}(V^{-2})
(qU,qV;q)_\infty
}{
(qUV;q)_\infty
(q^{1+x}UV,q^{1-x}UV;q)_\infty
},
\label{eq:CUVproduct}
\end{equation}
\begin{equation}
\Theta_q(U,V;z)
=
(q;q)_\infty^6
\frac{
\theta_q(-z/U)\theta_q(-zU^{-3}/V)
}{
\theta_q(-z)\theta_q(U^{-1})
\theta_q(-zU^{-3})\theta_q(V^{-1})
},
\label{Thetaprod}
\end{equation}
\begin{equation}
\Psi_q(u,v):=\Psi_q(u)-\Psi_q(v),\quad
\Psi_q(u)
:=
u\frac{\theta_q'(u)}{\theta_q(u)},
\quad
\Lambda_q(u)
:=
\bigl(u\frac{\partial}{\partial u}\bigr)^2\log\theta_q(u),
\label{derivat}
\end{equation}
and $L$ is defined in~\eqref{eq:def-L}.
\end{theorem}
See the proof in Appendix~\ref{app:B}.

Expression~\eqref{eq:J52unexpanded} is completely explicit and can be evaluated either numerically or as a $(q,X)$-series. 
At $z=1$, the contour conditions~\eqref{eq:condUV} can be
satisfied only for $|x|<1$, so \eqref{eq:J52unexpanded} can be used only in that strip. To evaluate the analytic Jones invariant 
beyond this strip, we write the constant term as a double contour integral, we continue it in $x$, deform the contours, and
we include the residues of the poles they cross. An alternative numerical evaluation consists of calculating first 
\be
\hat{\mH}_{{\bf 5}_2^-}(q,q^r)= {1\over 8h(q;q)_\infty(q^2;q^2)_\infty^3(q^4;q^4)_\infty^3} \CT_{U,V}
 \left[(UV)^r\mathcal B(U,V)\mathcal I (U,V;1)\right], 
 \label{eq:H-diagonal}
\end{equation}
where $\mathcal B(U,V)$, $\mathcal I (U,V;z)$ are defined respectively in \eqref{eq:BUVpre}, 
\eqref{IUV-def}. Then, (\ref{eq:H-diagonal}) is evaluated by using a double contour integral, and finally by using 
\eqref{J52sum} one can calculate the analytic Jones invariant. We should note however that the sum over $r \ge 0$ appearing in \eqref{J52sum} diverges outside the strip $|x|<1$. It can be analytically continued outside this strip by subtracting finitely many terms of the small-$z$ expansion of
$\hat{\mH}_{{\bf 5}_2^-}(q,z)$, and adding their contributions back in
closed form. This is the method we used to plot the real and imaginary parts of the interpolating Jones invariant 
$\hat J_{{\bf 5}_2^-,x}(q)$ at $q=1/2$ as a function of $x$, in \figref{52-fig}. 

\begin{figure}[!ht]
\leavevmode
\begin{center}
\includegraphics[height=4cm]{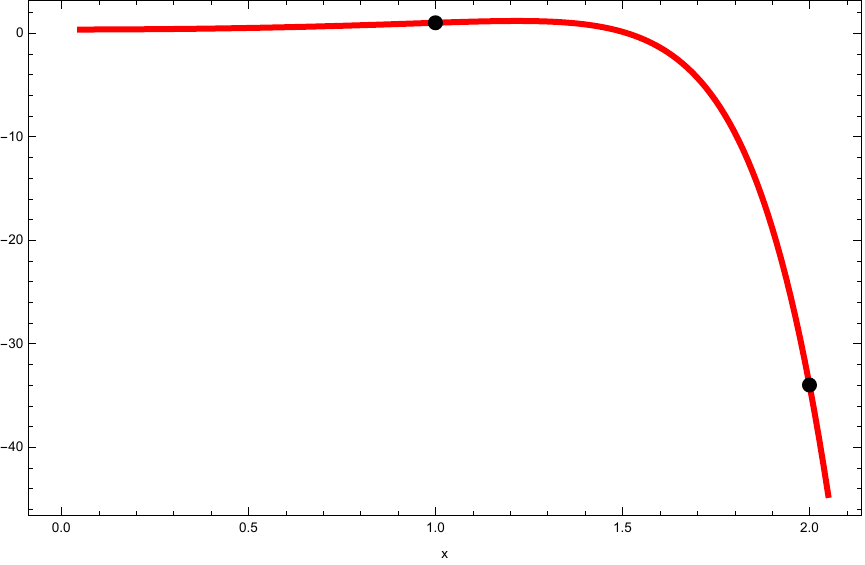}\qquad \qquad \includegraphics[height=4cm]{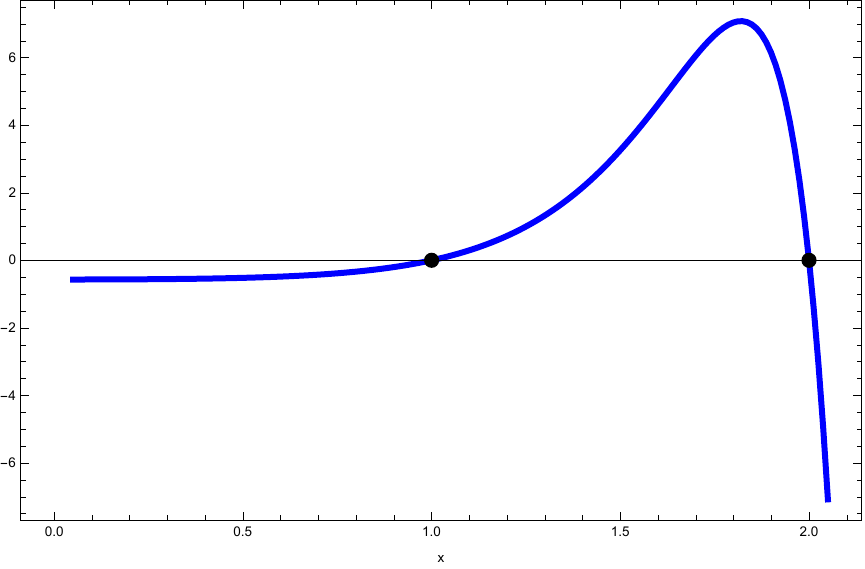}
\end{center}
\caption{Left-hand side: the red line is the real part of the interpolating Jones invariant $\hat J_{{\bf 5}_2^-,x}(q)$ as a function of $x$ for $q=1/2$. The black dots are the corresponding values $1, -34$ of the colored Jones polynomial for $n=1,2$. Right-hand side: the blue line is the imaginary part of the interpolating Jones invariant $\hat J_{{\bf 5}_2^-,x}(q)$ for $q=1/2$. The black dots at $x=1,2$ are on the $x$-axis and show that the imaginary part vanishes, as required.}
\label{52-fig}
\end{figure}

In terms of $(q,X)$-series, at $z=1$ one finds the structure 
\begin{equation}
\hat{J}_{{\bf 5}_2^-,x}(q)
=\mathcal A_x(q)+L\mathcal B_x(q)
+\left(L^2+\frac{3}{h}\right)\mathcal C_x(q),
\qquad
L=\frac{\log(-1)}{\log q}.
\label{eq:general-qX-decomposition}
\end{equation}
If we introduce the symmetric Laurent polynomials
\begin{equation}
S_j=X^j+X^{-j}, 
\label{eq:general-qX-Sj}
\end{equation}
we explicitly find from the above formulae
\begin{equation}
\begin{aligned}
\mathcal A_x(q)
={}&-\frac14q
+\frac{S_2+3S_1+10}{8}q^2+\frac{3S_3+7S_2+5S_1+26}{8}q^3
\\
&+\frac{2S_4-8S_3-11S_2-22S_1+38}{8}q^4
\\
&+\frac{2S_5+4S_4+S_3+36S_2-23S_1+105}{8}q^5
+\mathcal O(q^6),
\end{aligned}
\label{eq:general-qX-A-series}
\end{equation}
\begin{equation}
\begin{aligned}
\mathcal B_x(q)
={}&\frac{3S_1+7}{32}q
+\frac{5S_2-2S_1-18}{32}q^2+\frac{-11S_2+4S_1+38}{32}q^3
\\
&+\frac{11S_3+30S_2+2S_1-58}{32}q^4
\\
&+\frac{-4S_5-8S_4-36S_3-73S_2-24S_1+98}{32}q^5
+\mathcal O(q^6),
\end{aligned}
\label{eq:general-qX-B-series}
\end{equation}
and
\begin{equation}
\begin{aligned}
\mathcal C_x(q)
={}&\frac{S_1+3}{16}q
+\frac{S_2-2S_1}{16}q^2+\frac{-2S_3-3S_2-4S_1+2}{16}q^3
\\
&+\frac{S_3+6S_2-2S_1+4}{16}q^4
\\
&+\frac{2S_4-6S_3-7S_2-6S_1+14}{16}q^5
+\mathcal O(q^6).
\end{aligned}
\label{eq:general-qX-C-series}
\end{equation}

\subsubsection{Formal asymptotic expansion} As in previous cases, the Jones interpolating series (\ref{eq:cJHi}) leads to a
$z$-deformed MMR series for the ${\bf 5}_2^{-}$ knot,
\be
\Phi_0^{{\bf 5}_2^{-}}(h; w, z) = \frac{1}{1+2az}-\frac{az(6+5az)}{(1+2az)^3}h + \CO(h^2),
\ee
where
\begin{equation}
a=\left(2\sinh\frac{w}{2}\right)^2
=X+X^{-1}-2.
\label{eq:a-from-w}
\end{equation}
We now show that the formal asymptotic expansion of $\hat{J}_{{\bf 5}_2^-,x}(q,z)$ when $q \rightarrow 1^{-}$ is given by this series. The key idea is to use modularity of the answer, since after an $S$-duality transformation we have to study the comparatively simpler limit $\tq \rightarrow 0$.

The first step is to write the constant term operator in (\ref{CTMaster}) and subsequent equations
in terms of a double contour integral. We first write
\begin{equation}
U=q^{-u},\qquad V=q^{-v}.
\label{eq:UV-uv}
\end{equation}
Then,
\begin{equation}
\CT_{U,V}\left[ f(U,V) \right]=\frac{h^2}{(2\pi\ii)^2}
\int_{\mathcal C_u}\int_{\mathcal C_v}
f(q^{-u},q^{-v})\,\dd u\,\dd v,
\label{eq:CT-contour}
\end{equation}
where \(\mathcal C_u\) and \(\mathcal C_v\) are vertical cycles of period
\(2\pi\ii/h\). For the formal residue calculation below, their precise positions are irrelevant as long as they 
remain in the same pole-free region, with the poles at nonpositive integers on the same side. 

Let us introduce the additional notation
\be
Z= \log(-z), \qquad \xi=\re^{-2\pi\ri u},\qquad
\eta=\re^{-2\pi\ri v},\qquad \widetilde z=\re^{2\pi\ri Z/h}.
\ee
 By using the $S$-duality transformation of the theta functions and Pochhammer symbols, we find
\begin{equation}
\ba
&\frac{(q^{1+x},q^{1-x};q)_\infty}{8(q;q)_\infty(q^2;q^2)_\infty^3(q^4;q^4)_\infty^3}\CC_{q,x}(q^{-u},q^{-v})\Theta_q(q^{-u},q^{-v};z)
=\\
& \qquad \qquad\,\re^{-(Z+\ri\pi)(u+v)}\re^{h\Phi(u,v)}
\mathcal G_q(u,v)\mathcal R_X(u,v;q)E^\vee(\tq)R_z^\vee(u,v;\tq).
\ea
\label{eq:dual-product}
\end{equation}
In this equation, we have introduced
\be
\ba
\Phi(u,v)&=2u^2+3uv+v^2-3u-2v, \\
\mathcal G_q(u,v)&={\qp{q}{q}\left( q^{1-u}, q^{1-v}; q\right)_\infty   \over \left(q^{1-u-v}; q \right)^3_\infty},\\
\mathcal R_X(u,v;q)&=
{ \left( qX,q/X; q\right)_\infty \qp{q^{1-u-v}}{q}^{2} \over \qp{q}{q}^{2}\left( q^{1-u-v}X, q^{1-u-v}/X; q\right)_\infty }, \\
E^\vee(\tq)&=
\frac{\qp{\tq}{\tq}^{6}}
{\qp{\tq^{1/2}}{\tq^{1/2}}^{3}\qp{\tq^{1/4}}{\tq^{1/4}}^{3}},\\
R_z^\vee(u,v;\tq)&=
\frac{
\theta_{\tq^{1/4}}(\tq^{1/4}\xi)
\theta_{\tq^{1/2}}(\tq^{1/2}\eta)
\theta_\tq(\tq\widetilde z\,\xi)
\theta_\tq(\tq\widetilde z\,\xi^3\eta)
}{
\theta_\tq(\tq\widetilde z)
\theta_\tq(\tq\xi)
\theta_\tq(\tq\widetilde z\,\xi^3)
\theta_\tq(\tq\eta)
}.
\ea
\label{eq:Phi}
\end{equation}
The second line in \eqref{eq:J52unexpanded} can also be transformed. One finds that it is equal to 
\be
{4 \pi^2 \over h^2}{\mathcal{D}}^\vee(u,v;\tq),
\ee
where
\begin{equation}
\begin{aligned}
\mathcal D^\vee(u,v;\tq)
={}&
\left[
\Psi_\tq(\tq\widetilde z\,\xi)-\Psi_\tq(\tq\xi)
\right]
\\
&\times
\left[
-3\Psi_\tq(\tq\widetilde z)
+3\Psi_\tq(\tq\widetilde z\,\xi)
-\Psi_\tq(\tq\widetilde z\,\xi^3\eta)
+\Psi_\tq(\tq\eta)
\right]
\\
&+3\Lambda_\tq(\tq\widetilde z\,\xi).
\end{aligned}
\label{eq:Ddual}
\end{equation}
We can now expand around $\tq=0$ to find 
\be
\ba
E^\vee(\tq)&=1+\mathcal O(\tq^{1/4}),\\
R_z^\vee(u,v;\tq)\mathcal D^\vee(u,v;\tq)
&=-\frac{1}{(\re^{-2\pi\ii u}-1)(\re^{-2\pi\ii v}-1)}
+\mathcal O(\tq^{1/4}).
\ea
\ee
Therefore, in order to obtain the formal asymptotic expansion as $h\to 0^+$, we can write
\be
\hat J_{{\bf 5}_2^{-},x}(q,z)
\sim -\frac14
\int_{\mathcal C_u}\int_{\mathcal C_v}
\frac{
\re^{-Z(u+v)}\re^{h\Phi(u,v)}
\mathcal G_q(u,v)\mathcal R_X(u,v;q)
}{
\sin(\pi u)\sin(\pi v)
}\,\dd u\,\dd v.
\label{eq:pert-double-integral}
\end{equation}
The r.h.s. is a double Mellin--Barnes integral, whose formal asymptotic expansion
can be evaluated as in Subsection \ref{prop-as}. The integrand has poles
at $u=-m$, $v=-n$, with $m,n\in\mathbb Z_{\ge0}$. By using again the asymptotics
(\ref{qpoch-z}), we find
\be
\begin{aligned}
\hat J_{{\bf 5}_2^{-},x}(q,z)
& \sim
\sum_{m,n\geq0}
(-z)^M\binom{M}{m}
\left(2\sinh\frac{w}{2}\right)^{2M}
\\
&\times
\exp\Biggl\{
\frac{h}{2}
\left(3m^2+3mn+n^2+5m+3n\right)
\\
&\quad+
\sum_{\ell\geq1}\frac{h^{2\ell}}{(2\ell+1)!}
\Biggl[
2B_{2\ell+1}(M+1)
\left(\frac{\dd}{\dd w}\right)^{2\ell}
\log\left(2\sinh\frac{w}{2}\right)
\\
&\qquad\qquad+
\frac{B_{2\ell}}{2\ell}
\Bigl(
B_{2\ell+1}(M+1)
-B_{2\ell+1}(m+1)
-B_{2\ell+1}(n+1)
\Bigr)
\Biggr]
\Biggr\},
\end{aligned}
\label{eq:McIntosh-double-residue}
\end{equation}
where $M=m+n$. It is easy to see that this agrees with $\Phi_0^{{\bf 5}_2^{-}}(h; w, z)$ after using the expansion of $h_m(q)$ in \eqref{hm-def} as $q\to 1^{-}$. As in the case of the figure-eight knot and the trefoil knot, numerical calculations are consistent with the $z$-deformed MMR series giving the asymptotic expansion of $\hat J_{{\bf 5}_2^{-},x}(q,z)$ in the Poincar\'e sense.

\section{Interpolation of a sequence of two-variable polynomials}
\label{sec:2variable}

In this section, we apply the deformed Borel--Laplace resummation to construct an analytic invariant associated with a sequence of two-variable polynomials that are invariants of the figure-eight knot. These invariants arise from a solution of the quantum Yang--Baxter equation on a vector space of dimension $2^n$, $n\in\mathbb{Z}_{>0}$, constructed in~\cite{MR4992645}.

It was conjectured that these invariants coincide with the two-variable Links--Gould polynomials, associated with a one-parameter family of $2^n$-dimensional representations of the quantum supergroup $U_q(\mathfrak{gl}(n|1))$. Using the results of~\cite{MR1822139}, this conjecture was verified for several knots when $n=2,3,4$.

For the figure-eight knot $K=\mathbf{4}_1$, the resulting invariant is given by the two-variable Laurent polynomial
\be\label{eq:4_1-lg-n-1}
M_{\mathbf{4}_1,n}(q,t)=\sum_{m=0}^n
{n\brack m}_q
(t,q^{1-n}t^{-1};q)_m.
\ee
Since this sum has a structure similar to Habiro's cyclotomic expansion of the colored Jones polynomials~\eqref{eq:cJH}, we can apply our machinery of deformed Borel--Laplace resummation to extend it to non-integer values of $n$. Indeed, the divergent formal power series
\be\label{eq:TN}
M_{\mathbf{4}_1,y}(q,t,z):=\sum_{m\ge 0}
\frac{(t,q^{-y},q^{1-y}t^{-1};q)_m}{(q;q)_m\gamma_q(m)}
q^{my}z^m
={}_3\phi_1\!\Bigl[
\begin{matrix}
t,q^{1-y}t^{-1},q^{-y}\\
0
\end{matrix};q,q^yz
\Bigr]
\ee
extends~\eqref{eq:4_1-lg-n-1} through the equality
\be
M_{\mathbf{4}_1,n}(q,t)=M_{\mathbf{4}_1,n}(q,t,1),\quad \forall n\in\Z_{>0}.
\ee
Moreover, it specializes to $J_{\mathbf{4}_1,x}(q,z)$ given by~\eqref{eq:4_1-jones-3-fi-1}:
\be
J_{\mathbf{4}_1,x}(q,z)=M_{\mathbf{4}_1,-1}(q,q^{1+x},z).
\ee
Assuming $y\in\R$ and  $t\in\C$ such that
\be
q^{-y}<|t|<1,
\ee
formula~\eqref{eq:resummed-r-fi-s}, applied to~\eqref{eq:TN}, gives
\begin{align}
&\frac{(q;q)_\infty^4}{(t,q^{1-y}t^{-1},q^{-y};q)_\infty}\widehat{M}_{\mathbf{4}_1,y}(q,t,z)=(q;q)_\infty^3\sum_{l_1,l_2,l_3\ge0}\rme{q}(zq^{1+y+l_1+l_2+l_3})\frac{(q^{1+y};q)_{l_3}t^{l_1-l_2}q^{(1-y)l_2-yl_3}}{(q;q)_{l_1,l_2,l_3}},\nonumber\\
&=\sum_{n\in\Z}\frac{\frac{\log z}{\log q}+y+n}{1-zq^{y+n}}\sum_{l_2,l_3\ge0}\gamma_q(1+l_2+l_3-n)(tq^{1+l_2+l_3-n};q)_\infty\frac{(q^{1+y};q)_{l_3}t^{-l_2}q^{(1-y)l_2-yl_3}}{(q;q)_{l_2,l_3}},\nonumber\\
&=\sum_{n\in\Z}\frac{\frac{\log z}{\log q}+y+n}{1-zq^{y+n}}\sum_{r\ge0}\gamma_q(1+r-n)(q^{1-y}/t)^{r}(tq^{1+r-n};q)_\infty\sum_{l=0}^r\frac{(q^{1+y};q)_{l}(t/q)^{l}}{(q;q)_{l,r-l}},\label{eq:LH-pole-sum}
\end{align}
where, in the second equality, we converted the summation over $l_1$ into the infinite $q$-Pochhammer symbol. 

As in previous examples, it is also useful to have a Mellin--Barnes representation for $\widehat{M}_{\mathbf{4}_1,y}(q,t,z)$. 
The coefficients in \eqref{eq:TN} grow as $q^{-m^2/2}$ and, by applying  the $q$-Borel transform with respect to variable $z$, we obtain
\begin{equation}
({\mathcal B}_q{M}_{\mathbf{4}_1,y}(q,t,\cdot))(z)=
{}_3\phi_2\Bigl[
\begin{matrix}
t,q^{1-y}t^{-1},q^{-y}\\
q,0
\end{matrix};q,-q^{1+y}z
\Bigr]. 
\label{eq:BT}
\end{equation}

We can then proceed with a Mellin--Barnes representation for ${}_3\phi_2$.
Repeating calculations similar to those in Section~\ref{sec:4_1-int-sub-sec1}, we obtain
\begin{equation}
\widehat{M}_{\mathbf{4}_1,y}(q,t,z)
=
\frac{(t,q^{1-y}/t,q^{-y};q)_\infty}
{h(q;q)_\infty^3}
\frac1{2\ri}
\int_{-\varepsilon-\ri\infty}^{-\varepsilon+\ri\infty}
\frac{(q^{1+s},q^{1+s},q^{-s};q)_\infty\,
\pi(q^{1+y}z)^s}
{(tq^s,q^{1-y+s}/t,q^{s-y};q)_\infty
\sin^2(\pi s)}\,\rd s.
\label{eq:That-MB}
\end{equation}
We can calculate the integral \eqref{eq:That-MB} by closing the contour of integration to the left.
Far to the left the integrand decays as $q^{N(N+1)/2}$ for $s=-N+u$, where $u$ belongs to a fixed vertical strip.
So the left closing contribution tends to zero super-exponentially.

For simplicity, let us assume that all variables are real and satisfy the following conditions:
\begin{equation}
\quad 0<\varepsilon<1,
\quad
y<-\varepsilon,
\quad
1-\varepsilon-y>v>\varepsilon,
\quad z>0,
\label{eq:common-region}
\end{equation}
where $t=q^v$.
In this case, all three series of poles of the integrand are located to the left of the contour:
\be
s_{n,m}(a)=a-n+\frac{2\pi \ri m}{\log q},\quad n\ge0,\quad m\in\mathbb Z,\quad a\in A,
\ee
where the set $A$ is composed of three elements 
\be\label{eq:setA}
A=\{y,-v,y+v-1\},
\ee
and we assume that for any two distinct elements $a,b\in A$, $a-b\not\in\Z$. 

Put
\be
\rho=q^{1+y}z.
\ee
The lattice sum generated by the imaginary translates of a pole family is given by $\fn{2,q}(1/\rho,a)$, defined by~\eqref{def-fqn} with $n=2$,
where $a\in A$.

Using the explicit form~\eqref{ident2} of $\fn{2,q}$, we obtain
\begin{equation}
\fn{2,q}(1/\rho,a)
=
\frac{h(q;q)_\infty^3}{\pi^2} T(a),
\label{eq:Ktheta}
\end{equation}
with the theta-lattice factors
\be
T(a)
=
\frac{\theta_q(q^{a}/\rho)}
{\theta_q(q^a)\theta_q(\rho^{-1})}
\left\{
 h\bigl[\Psi_q(q^{a}/\rho)-\Psi_q(q^a)\bigr]
 +\log \rho
\right\},
\ee
where $\Psi_q$ is defined in~\eqref{derivat}.
For each $a\in A$, denoting
\be
\{a',a''\}=A\setminus\{a\},
\ee
and evaluating the residues for all $n\ge0$, we obtain the following result
for the $n$-sums
\be\label{n-sums}
\Phi(a)=
\frac{(q^{1+a};q)_\infty^2(q^{-a};q)_\infty}
{(q^{a-a'},q^{a-a''};q)_\infty}
{}_2\phi_2\Bigl[
\begin{matrix}q^{-a},0\\q^{1+a'-a},q^{1+a''-a}\end{matrix};q,\frac{q^{1+y-a}}z
\Bigr].
\ee
Summing the $m$-lattices with \eqref{eq:Ktheta} and using \eqref{n-sums}, we finally obtain
\begin{equation}
\widehat{M}_{\mathbf{4}_1,y}(q,q^v,z)
=
\frac{\prod_{b\in A}(q^{-b};q)_\infty}
{h(q;q)_\infty}
\sum_{a\in A}
T(a)\Phi(a).
\label{eq:That-theta}
\end{equation}
Since ${}_2\phi_2$ is entire in its last argument and the theta/product factors are meromorphic on $\mathbb C^\times$, \eqref{eq:That-theta} gives analytic continuation in $y,v,z$ beyond \eqref{eq:common-region}.
In the limit $z\to 1$ and then $y\to n$, only the first term with $a=y$ survives in the sum \eqref{eq:That-theta}, all poles cancel, 
and we arrive at
\eqref{eq:4_1-lg-n-1}.

\section{Deformed Borel--Laplace resummation and the WRT invariant}
\label{sec:wrt}
In this section we apply the deformed Borel--Laplace resummation procedure to $q$-series which appear naturally in the theory
of WRT invariants of three-manifolds \cite{MR990772, MR1091619}. The analytic objects that we obtain are defined by a surgery presentation of the three-manifold and might depend on it. They lead however to interesting connections to $q$-series associated to three-manifolds which have appeared in the literature.

\subsection{General formalism}

Let $M_K=S_{-p/r}(K)$ be the three-manifold obtained by performing a $-p/r$ surgery on a knot $K$. In \cite{bbl}, it was found that
a universal formula for the WRT invariant of $M_K$ can be obtained by using the Habiro form
of the colored Jones polynomial of $K$. This formula employs a Laplace transform $\mL_b$, $b\in\Q$, acting linearly on the ring of Laurent polynomials in $X^{1/d}$, $d \ge 1$,
defined as
\be
\mL_{b} (X^u)= q^{-u^2/b}.
\ee
Then, up to an overall factor of the form $Aq^\alpha$, the WRT invariant of $M_K$ is given by the Laplace transform
\be
\label{wrt-laplace}
\tau _{M_K} (q)= \mL_{-p/r} \left[ (X^{1\over 2} - X^{-{1\over 2}})(X^{1\over 2r} - X^{-{1\over 2r}}) J_{K,x}(q) \right],
\ee
where the colored Jones polynomial $J_{K,x}(q)$ is expressed in terms of the variable $X=q^x$.
By using Habiro's formula~\eqref{eq:cJH}, this can be written as
\be
\label{tau-gen}
\tau _{M_K} (q)=\sum_{m \ge 0} {H_{K,m}(q) \over \mg{q}{-m}}
\CF_{-p/r, m}(q),
\ee
where
\be
 \CF_{-p/r, m}(q)= \mL_{-p/r} \left[ (X^{1\over 2} - X^{-{1\over 2}})(X^{1\over 2r} - X^{-{1\over 2r}})  (q X,q/X;q)_m\right].
\ee

The $q$-series~\eqref{tau-gen} is defined when $q$ is a root of unity, since only in that case the series truncates and 
one obtains a well-defined
invariant of $M_K$, independent of the surgery presentation. We now promote~\eqref{tau-gen}
to a power series in $z$, as we did in the
case of the colored Jones invariant, assuming $0<q<1$:
%
\be
\label{tauz-1}
\tau_{M_K} \left(q,z \right)=\sum_{m \ge 0} {H_{K,m}(q) \over \mg{q}{-m}} {\mathcal F}_{-p/r, m}(q) z^m.
\ee
Note that this series depends on the knot $K$ and the surgery data $p,r$, and it is not in general an invariant of the three-manifold $M_K$. Let us now introduce the auxiliary series
\be
\Fsurg_{-p/r}(q,z)=\sum_{m \ge 0} {\mathcal F}_{-p/r, m}(q) z^m
\ee
and extend by linearity the definition of Laplace transform $\mL_b$ to formal power series in $z$ whose coefficients are
Laurent polynomials in $X^{1/d}$:
\be
\mL_b\left[\sum_{m\geq0}P_m(X)z^m\right]
:=
\sum_{m\geq0}\mL_b[P_m(X)]z^m,
\qquad
P_m(X)\in\mathbb C[X^{\pm1/d}].
\ee
With this extension one has
\be
\Fsurg_{-p/r}(q,z)=  \mL_{-p/r} \left[ (X^{1\over 2} - X^{-{1\over 2}})(X^{1\over 2r} - X^{-{1\over 2r}})  \phi_x(q,z)\right],
\ee
where $\phi_x(q,z)$ is defined in \eqref{eq:had-fac-varphi_x}. We can then write (\ref{tauz-1}) as a Hadamard product:
\be
\label{tauz-2}
\tau_{M_K} \left(q,z \right)=\Fsurg_{-p/r}(q,z) \had_z \mH_{K}(q,z).
\ee
This is the WRT counterpart of formula (\ref{eq:J_{K,x}(q,z)-hadamard}). 

\begin{lemma} \label{sm-wrt} For any pair of positive coprime integers $p,r$, $\Fsurg_{-p/r}(q,z)$ is Taylor's expansion of a simple meromorphic function.
\end{lemma}
\begin{proof}
We denote
\begin{equation}
 \lambda:=\frac rp,
 \qquad
 a_+:=\frac{r+1}{2r},
 \qquad
 a_-:=\frac{r-1}{2r},
 \label{eq:parameters}
\end{equation}
and consider the following completion of the vector space of
Laurent polynomials:
\begin{equation}
 \CA_{\lambda,r}(q):=
 \bigl\{
 \sum_{u\in(2r)^{-1}\mathbb Z}c_uX^u:\
 \sum_{u\in(2r)^{-1}\mathbb Z}|c_u|q^{\lambda u^2}<\infty
 \bigr\}.
 \label{eq:laplace-domain}
\end{equation}
For $F(X)=\sum_uc_uX^u\in \CA_{\lambda,r}(q)$, set
\begin{equation}
 \|F\|_\lambda:=\sum_u|c_u|q^{\lambda u^2},
 \qquad
 \mL_{-p/r}[F]:=\sum_uc_uq^{\lambda u^2}.
 \label{eq:laplace-continuous-extension}
\end{equation}
This extended Laplace transform acts continuously in the sense that
$|\mL_{-p/r}[F]|\leq\|F\|_\lambda$.

We start with the Mittag--Leffler representation of $\phi_x(q,z)$:
\begin{equation}
 \phi_x(q,z)
 =(qX,q/X;q)_\infty
 \sum_{k\geq0}\frac{q^k b_k(q,x)}{1-zq^k},
 \qquad
 b_k(q,x)=\sum_{\ell=0}^k
 \frac{X^{2\ell-k}}{(q;q)_\ell(q;q)_{k-\ell}}.
 \label{eq:phi-ML}
\end{equation}
Let us introduce the theta function
\begin{equation}
 \Theta_\lambda(c;q)
 :=\sum_{n\in\Z}\gamma_q(n)
 q^{\lambda(n+c)^2}.
 \label{eq:theta-def}
\end{equation}
Using the equality
\begin{equation}
\ba
&(X^{1/2}-X^{-1/2})(X^{1/(2r)}-X^{-1/(2r)})
 (qX,q/X;q)_\infty \\
  & \qquad =\frac{X^{-a_+}-X^{-a_-}}{(q;q)_\infty}
 \sum_{n\in\Z}\gamma_q(n)X^n,
 \ea
  \label{eq:cancellation}
\ee
and multiplying
\eqref{eq:cancellation} by $b_k(q,x)$ from \eqref{eq:phi-ML}, we obtain
\begin{equation}
 \begin{aligned}
 & (X^{1/2}-X^{-1/2})(X^{1/(2r)}-X^{-1/(2r)})
 (qX,q/X;q)_\infty b_k(q,x)\\
 &\quad=\frac{1}{(q;q)_\infty}
 \sum_{\ell=0}^k\frac{1}{(q;q)_\ell(q;q)_{k-\ell}}
 \sum_{n\in\Z}\gamma_q(n)
 \left(
 X^{n+2\ell-k-a_+}-X^{n+2\ell-k-a_-}
 \right).
 \end{aligned}
 \label{eq:shifted-powers}
\end{equation}
Let $G_k(X)$ denote the left-hand side of \eqref{eq:shifted-powers}.
Completing the square in $n$ shows that
\begin{equation}
 C_{\lambda,q}:=
 \sup_{c\in\mathbb R}\sum_{n\in\mathbb Z}
 q^{n(n-1)/2+\lambda(n+c)^2}<\infty.
 \label{eq:uniform-theta-bound}
\end{equation}
Since $(q;q)_j\geq(q;q)_\infty$, equation
\eqref{eq:shifted-powers} therefore gives
\begin{equation}
 \|G_k\|_\lambda
 \leq \frac{2C_{\lambda,q}}{(q;q)_\infty^3}(k+1).
 \label{eq:Gk-bound}
\end{equation}
Thus, $G_k\in\CA_{\lambda,r}(q)$. Applying term-wise the continuous Laplace
transform $ \mL_{-p/r}$ to the right hand side of~\eqref{eq:shifted-powers} turns the inner sums into
$\Theta_\lambda(2\ell-k-a_+;q)$ and
$\Theta_\lambda(2\ell-k-a_-;q)$.  Moreover, for every compact set
$K\subset\mathbb C\setminus q^{-\mathbb Z_{\geq0}}$,
\begin{equation}
 \sum_{k\geq0}\sup_{z\in K}
 \frac{q^k\|G_k\|_\lambda}{|1-zq^k|}<\infty.
 \label{eq:Gk-local-uniform}
\end{equation}
This justifies that the Laplace transform $ \mL_{-p/r}$ can be applied term by term to
\eqref{eq:phi-ML}. We then obtain
\begin{equation}
 \Fsurg_{-p/r}(q,z)
 =\sum_{k\geq0}\frac{q^k D_k^{(p,r)}(q)}{1-zq^k}, \label{eq:main-ML}
\end{equation}
where
\begin{equation}
\ba
 D_k^{(p,r)}(q)&= \mL_{-p/r} \left[ (X^{1/2}-X^{-1/2})(X^{1/(2r)}-X^{-1/(2r)})
 (qX,q/X;q)_\infty b_k(q,x) \right]\\
 &=\frac{1}{(q;q)_\infty}
 \sum_{\ell=0}^k
 \frac{
 \Theta_\lambda(2\ell-k-a_+;q)
 -\Theta_\lambda(2\ell-k-a_-;q)
 }{(q;q)_\ell(q;q)_{k-\ell}}.
 \ea
 \label{eq:Dk}
\end{equation}
Thus, the possible poles of $ \Fsurg_{-p/r}(q,z)$ are $z=q^{-k}$, $k\geq0$, with the residues
\begin{equation}
 {\rm Res}_{z=q^{-k}}\mathscr F_{-p/r}(q,z)
 =-D_k^{(p,r)}(q).
 \label{eq:residue}
\end{equation}
Inequality~\eqref{eq:Gk-bound} also implies that
$|D_k^{(p,r)}(q)|=\mathcal O(k)$.  Together with the factor $q^k$ in
\eqref{eq:main-ML}, this proves normal convergence away from
$q^{-\mathbb Z_{\geq0}}$. Hence $\Fsurg_{-p/r}(q,z)$ is simple meromorphic.
\end{proof}

\begin{example} When $p=r=1$, one has $\lambda=1$, $a_+=1$ and $a_-=0$, as defined in~\eqref{eq:parameters}. 
The theta functions appearing in~\eqref{eq:Dk} either vanish or are given by a monomial in $q$ times $(q;q)_\infty$.
One eventually finds
\be
 D_k^{(1,1)}(q)
 =-\frac{2(-1)^k}{(q;q)_k}\mathscr{P}_k(q),
 \label{eq:D11-closed}
 \ee
 with
 \be
 \mathscr{P}_k(q)=\sum_{j=0}^{\lfloor k/2\rfloor}
 q^j\genfrac{[}{]}{0pt}{}{k}{2j}_{q}(q;q^2)_j.
 \label{eq:Pk-definition}
\ee

When $p=1$, $r=2$, one has $\lambda=2$, $a_+=3/4$ and $a_-=1/4$, see~\eqref{eq:parameters}. Define
\begin{equation}
 A_m:=q^{-1/8}\frac{\Theta_2(m-3/4;q)}{(q;q)_\infty},
 \qquad
 B_m:=q^{-1/8}\frac{\Theta_2(m-1/4;q)}{(q;q)_\infty},
 \label{eq:AB-definition}
 \ee
 see also~\eqref{eq:theta-def}.
 Then,
 \begin{equation}
D_k^{(1,2)}(q)=q^{1/8}\sum_{\ell=0}^k
 \frac{A_{2\ell-k}-B_{2\ell-k}}
 {(q;q)_\ell(q;q)_{k-\ell}}.
 \label{eq:residue-12}
\end{equation}
$A_m$ and $B_m$ are quasiperiodic modulo $5$ and their values either vanish or are given by a monomial in $q$ times a Rogers--Ramanujan product.\hfill$\diamond$
\end{example}

In many cases, the series (\ref{tauz-1}) is divergent, so it is natural to apply a deformed Borel--Laplace resummation to obtain an analytic object.

\begin{definition} 
Let us suppose that the coefficients of the generating series (\ref{tauz-1}) grow as $q^{-\beta n^2/2}$. Then, we define the analytic function $\hat \tau_{M_K}(q,z)$ as the $q^\beta$-Borel--Laplace resummation of $\tau_{M_K} (q,z)$.
\end{definition}

\begin{remark} An important caveat is that the construction above associates the analytic function $\hat \tau_{M_K}(q,z)$ to the surgery presentation $(K,-p/r)$, rather than to the three-manifold $M_K$ itself. This is because invariance under changes of surgery presentation for $0<q<1$ has not been established, and this is probably the most important open problem concerning this function. Nonetheless, we will denote the function by $\hat \tau_{M_K}(q,z)$ to simplify the notation. We should note however that in some cases different surgery presentations of the same manifold lead to identical generating functions $\hat \tau_{M_K}(q,z)$. For example, the Brieskorn homology sphere $\Sigma(2,3,7)$ can be obtained either as $-1$ surgery on ${\bf 3}_1^-$ or as $+1$ surgery on ${\bf 4}_1$. In both cases, one obtains the same generating function (up to trivial overall factors of the form $Aq^{\alpha}$). 
\end{remark}

 In appropriate circumstances, the following result reduces the calculation of this analytic function to the $q^\beta$-Borel--Laplace resummation of the generating series of Habiro polynomials.

\begin{proposition} \label{tauhat-gen} Let $p,r$ be two positive coprime integers and $z\in \mathbb{C}^\times$. Let us suppose that the $q^\beta$-Borel transform $\tilde \mH_K (q,z)$ admits an analytic continuation along the ray $z \IR_{\ge 0}$, and that 
\begin{equation}\label{eq:H-uniform-Laplace}
\sup_{\ell\geq0}
\int_0^\infty
\frac{\bigl|\widetilde{\mH}_K(q,q^\ell zt)\bigr|}
     {(-t;q^\beta)_\infty}\,\rd t<\infty.
\end{equation}
Then,
\be
\hat \tau_{M_K}\left(q,z \right)= \sum_{\ell \ge 0}  q^\ell D_\ell ^{(p,r)}(q) \hat{\mH}_K(q,q^\ell z).
\ee
We also have
\be
\label{lcb}
\hat \tau_{M_K}\left(q,z \right)= \mL_{-p/r} \left[ (X^{1\over 2} - X^{-{1\over 2}})(X^{1\over 2r} - X^{-{1\over 2r}}) \hat J_{K,x}(q,z) \right],
\ee
i.e. $q^\beta$-Borel--Laplace resummation commutes with the Laplace transformation $\mL_{-p/r}$.
\end{proposition}

\begin{proof} 
By~\eqref{eq:H-uniform-Laplace} and
$D_\ell^{(p,r)}(q)=\mathcal O(\ell)$, condition~\eqref{eq:tonelli}
holds. In addition, by Lemma \ref{sm-wrt}, we can apply Corollary~\ref{lq-cor} and we obtain the first formula. To obtain the second, we note that
\begin{equation}
 (X^{1/2}-X^{-1/2})(X^{1/(2r)}-X^{-1/(2r)})
 \hat J_{K,x}(q,z)
 =\sum_{\ell\geq0}q^\ell G_\ell(X)
   \hat{\mH}_K(q,q^\ell z),
\end{equation}
where $G_\ell$ is defined below~\eqref{eq:shifted-powers}.
The condition~\eqref{eq:H-uniform-Laplace} gives a
bound on $|\hat{\mH}_K(q,q^\ell z)|$ independent of $\ell$.
Together with~\eqref{eq:Gk-bound}, this shows that the
right-hand side converges in $\CA_{\lambda,r}(q)$, and we may thus apply $\mL_{-p/r}$ term by
term. Since $\mL_{-p/r}[G_\ell]=D_\ell^{(p,r)}(q)$ by
\eqref{eq:Dk}, the result is the first formula.
\end{proof}
Our use of the Laplace transform as a bridge between the analytic invariants 
of knots and $q$-series attached to three-manifolds is inspired by similar constructions in \cite{gmp,MR4233201}. Proposition~\ref{tauhat-gen} is similar in spirit to Conjecture 1.7 of \cite{MR4233201}.

In the next sections we will illustrate the general picture with various examples. In many cases there are
two equivalent procedures to obtain the analytic functions of three-manifolds: we can use Proposition \ref{tauhat-gen} and rely on known results for the $q^\beta$-Borel--Laplace resummation of the underlying knot, or we can find a Mellin--Barnes representation
directly for the resummation of (\ref{tauz-1}). The procedures are complementary and lead to different viewpoints on the analytic function. In all the examples below, the $q^\beta$-Borel--Laplace resummation of the WRT invariant leads to an expression built from $q$-series, with two properties.
First, its formal asymptotic expansion as $h \to 0$ reproduces the perturbative expansion of the
WRT invariant, or Ohtsuki series of the three-manifold. Numerical calculations are consistent with 
this expansion being an asymptotic expansion in the Poincar\'e sense. Second,
the resulting $q$-series are closely related to $q$-series that have appeared before in the literature 
\cite{lz, hikami1, hikami2, gmp, 3dmod, MR4233201, wheeler}. Similar to what
happens for the analytic invariants of knots, which lead to corrected versions of the inverted Habiro series, 
in some of the examples below the analytic functions of three-manifolds lead to ``corrected'' versions of their postulated 
$\hat Z$-invariants. These corrections are necessary in order to preserve their asymptotic 
relationship to the Ohtsuki series. However, explicit computations show that the resulting $q$-series do not have well-defined radial limits at other roots of unity. This is reminiscent of the fact that radial limits of mock theta functions require cusp-dependent modular corrections.

\subsection{Brieskorn homology spheres}

We first consider some examples of Brieskorn homology spheres, focusing on
manifolds constructed by surgery on knots that we have
already studied, so we can exploit our previous results and illustrate the general theory. 
In addition, these are manifolds for which $q$-series invariants have been already proposed 
from different perspectives, and we will be able to compare our results to these
proposals.

\subsubsection{Manifold $\Sigma(2,3,7)$} Our first example is $\Sigma(2,3,7)$ with its canonical orientation. We realize it as a $-1$ surgery on the trefoil knot ${\bf 3}_1^-$. Since
\be
\CF_{-1,m}(q)= -2 (q^{m+1};q)_{m+1},
\ee
the WRT invariant is given by \cite{bbl, hikami2}
\be
\tau_{\Sigma(2,3,7)}(q)= -2 \sum_{m=0}^{\infty}  q^{-m(m+2)-1}  (q^{m+1};q)_{m+1},
\ee
where for convenience we have included an overall factor $q^{-1}$ in the definition (\ref{tau-gen}). The corresponding generating series is
\be
\label{z-z}
\tau_{\Sigma(2,3,7)}(q,z)= -2 \sum_{m=0}^{\infty}  q^{-m(m+2)-1}  (q^{m+1};q)_{m+1} z^m.
\ee
This series is clearly divergent when $q<1$. To obtain an analytic object, we have to perform a $q^2$-Borel--Laplace resummation. The following result is proven in Appendix~\ref{app:B}.

\begin{lemma}\label{lem:analytic-sigma237} The analytic function of the manifold $\Sigma(2,3,7)$, presented as $-1$ surgery on the
trefoil knot ${\bf 3}_1^-$, is given by
\be
\hat \tau_{\Sigma(2,3,7)}(q,z)= A_{\Sigma(2,3,7)}(q,z)-{\log(-z) \over h} B_{\Sigma(2,3,7)}(q,z),
\ee
where
\be
\ba
 A_{\Sigma(2,3,7)}(q,z)
 &=-{q^{-1}\over (q^2;q^2)_\infty^3}
 \sum_{k\geq0}{(-1)^kq^k\mathscr{P}_k(q)\over(q;q)_k}
 \sum_{n\in\mathbb Z}\gamma_{q^2}(-n)
 {k+2n-1\over1+z q^{k+2n-1}},
 \\
 B_{\Sigma(2,3,7)}(q,z)
 &=-{q^{-1}\over (q^2;q^2)_\infty^3}
 \sum_{k\geq0}{(-1)^kq^k\mathscr{P}_k(q)\over(q;q)_k}
 \sum_{n\in\mathbb Z}{\gamma_{q^2}(-n)\over
 1+z q^{k+2n-1}},
 \ea
 \label{eq:sigma-AB-series}
\ee
and $\mathscr{P}_k(q)$ is defined in~\eqref{eq:Pk-definition}. 

We also have the equalities for $z=1$:
 \be\label{A237M}
 A_{\Sigma(2,3,7)}(q) = \CM(q),
 \ee
 where
 \be
\label{mq}
\CM(q)= \sum_{n\geq 0}
  {\gamma_{q}(-n)
  \over
  \qpochn{q^{n+1}}{q}{n}},
  \ee
  and
 \begin{equation}
  B_{\Sigma(2,3,7)}(q)
  =
  {(q;q)_\infty^2\over 2 (q^2;q^2)_\infty^5}
  \left[
  \theta_{q^6}(q)
  \left(q\,\theta_{q^{21}}(q^4)-\theta_{q^{21}}(q^{10})\right)
  -
  q\,\theta_{q^6}(q^3)\theta_{q^{21}}(q^3)
  \right].
  \label{eq:H-theta-constant-form}
\end{equation}

Finally, let
\be
\label{phi07}
\varphi_{\Sigma(2,3,7)}(h)=-2 h - 13 h^2  -{415 h^3 \over 3} - {24721 \over 12} h^4+\cdots
\ee
be the Ohtsuki series for this manifold. Then, one has the formal asymptotic expansion
\be
\hat \tau_{\Sigma(2,3,7)} (q) \sim \varphi_{\Sigma(2,3,7)}(h), \qquad h\rightarrow 0^+.
\ee
\end{lemma}
 %

 %

This result of the $q^2$-Borel--Laplace resummation is closely related to previous results. Let us introduce the false
  theta functions:
\be
\tilde \Psi^{(a)}_p (q)  :=  \sum_{n=0}^\infty \psi^{(a)}_{2p}(n) q^{\frac{n^2}{4p}} \,  \,  \in q^\frac{a^2}{4p}\,\IZ[[q]],
\label{Psis}
\ee
which are Eichler integrals of weight-$\tfrac{3}{2}$ vector-valued modular forms.
Here,
\begin{equation}
 \psi^{(a)}_{2p}(n)  =  \left\{
\begin{array}{cl}
 \pm 1, & n\equiv \pm a\mod 2p\,, \\
0, & \text{otherwise}.
\end{array}\right.
\end{equation}
Let us also introduce the notation \cite{hikami2}
\be
\tilde\Psi^{n_a (a) + n_b (b) + \cdots}_p (q) \; := \;
n_a \tilde\Psi^{(a)}_p (q)
+ n_b \tilde\Psi^{(b)}_p (q)
+ \ldots.
\label{linearPsitilde}
\ee
Then, for $\CM(q)$, defined in~\eqref{mq}, we have
\be\label{eq:cmq=psi-mock}
\CM(q) =q^{-{1 \over 168}} \tilde\Psi^{ (1) - (13) -(29) +(41)}_{42} (q),
\ee
and the $q$-series in the right hand side of~\eqref{eq:cmq=psi-mock} is precisely the false theta function associated to this
three-manifold in~\cite{hikami1}. The same theta function is the $\hat Z$-invariant of this manifold, see, for example~\cite{MR4233201}.

As in the previous knot examples, it is interesting to compare the result of the $q^2$-Borel--Laplace resummation of~\eqref{z-z} with the conventional Borel resummation
of the Ohtsuki series. When $h>0$, $\varphi_{\Sigma(2,3,7)}(h)$ is not Borel summable and one has to consider lateral
Borel resummations. An explicit formula for these was obtained in \cite{gmp}.
Their real parts agree exactly with the real part of the $q^2$-Borel--Laplace resummation:
\be
{\rm Re}\, s_\pm (\varphi_{\Sigma(2,3,7)})(h)={\mathcal M}(q).
\ee
The ambiguous, imaginary parts can also be written in terms of false theta functions, as follows:
\be
\label{p-discont}
\ba
{\rm Im} \, s_\pm (\varphi_{\Sigma(2,3,7)})(h)&=\pm  2 {\sqrt{2 \pi \over h}}q^{-1/168} \biggl[ \frac{\sin \left(\frac{\pi }{42}\right)-\cos \left(\frac{4 \pi
   }{21}\right)}{\sqrt{21}}  \tilde \Psi^{(1) - (13) - (29) + (41)}_{42} (\tq)\\
& \qquad \qquad\qquad+\frac{\sin \left(\frac{5 \pi }{42}\right)+\cos \left(\frac{\pi
   }{21}\right)}{\sqrt{21}}   \tilde\Psi^{(5) + (19) + (23) + (37)}_{42} (\tq) \\
   & \qquad \qquad\qquad+ \frac{\cos \left(\frac{2 \pi }{21}\right)+\cos \left(\frac{5 \pi
   }{21}\right)}{\sqrt{21}} \tilde\Psi^{(11) + (17) + (25) + (31)}_{42} (\tq) \biggr].
  \ea
 \ee
This is an exponentially small correction in $h$. By considering the $S$-duality transformation of (\ref{eq:H-theta-constant-form}) 
one can verify that, as in the case of the trefoil knot, the r.h.s. of (\ref{p-discont}) agrees with $\pm (\pi /h)B_{\Sigma(2,3,7)}(q)$ at leading order in $\tq$, but it is a different function of $h$.

\subsubsection{Manifold $\overline{\Sigma(2,3,7)}$} Our next example is the same manifold $\Sigma(2,3,7)$, but with opposite orientation.
The invariant is obtained simply by changing $h \to -h$, or equivalently, $q\to q^{-1}$. It can also be obtained by doing $-1$ surgery on the figure-eight knot, and one finds
\begin{equation}
  \tau_{\overline{\Sigma(2,3,7)}}(q)
  =
  2\sum_{m\geq0}
 \frac{(q^{m+1};q)_{m+1}}{\gamma_{q}(-m)},
  \label{eq:extra-f1-def}
\end{equation}
where we have introduced an overall $-1$ sign for convenience.
The corresponding series (\ref{tauz-1}) is divergent when $q<1$ and requires a
deformed Borel--Laplace resummation. The proof of the following result is contained in Appendix~\ref{app:B}.

\begin{lemma}\label{lem:analytic-bar-sigma237}  The analytic function of $\overline{\Sigma(2,3,7)}$, obtained by $-1$ surgery on the figure-eight knot, is given by
\be
\hat \tau_{\overline{\Sigma(2,3,7)}}(q,z)= A_{\overline{\Sigma(2,3,7)}}(q,z)+ {\log z \over h} B_{\overline{\Sigma(2,3,7)}}(q,z),
\ee
where
\be
\ba
 A_{\overline{\Sigma(2,3,7)}}(q,z)
 &={2\over(q;q)_\infty^3}
 \sum_{k\geq0}{(-1)^kq^k\mathscr{P}_k(q)\over(q;q)_k}
 \sum_{n\in\mathbb Z}\gamma_q(-n)
 {k+n\over1-zq^{k+n}},\\
 B_{\overline{\Sigma(2,3,7)}}(q,z)
 &=-{2\over(q;q)_\infty^3}
 \sum_{k\geq0}{(-1)^kq^k\mathscr{P}_k(q)\over(q;q)_k}
 \sum_{n\in\mathbb Z}{\gamma_q(-n)\over1-zq^{k+n}},
 \ea
 \label{eq:bar237AB}
 \ee
 with $\mathscr{P}_k(q)$ defined in~\eqref{eq:Pk-definition}.

 At $z=1$, this specializes to
 \begin{equation}
 \hat \tau_{\overline{\Sigma(2,3,7)}}(q)=A_{\overline{\Sigma(2,3,7)}}(q)+{1\over h}B_{\overline{\Sigma(2,3,7)}}(q),
 \label{eq:A-B-at-one}
\end{equation}
where
\be
\ba
A_{\overline{\Sigma(2,3,7)}}(q)
 &={2\over(q;q)_\infty^3}
 \sum_{k\geq0}{(-1)^kq^k\mathscr{P}_k(q)\over(q;q)_k}
 \sum_{\substack{n\in\mathbb Z\\n\neq-k}}
 \gamma_q(-n){k+n\over1-q^{k+n}}\\
 B_{\overline{\Sigma(2,3,7)}}(q)
 &={2\over(q;q)_\infty^3}
 \sum_{k\geq0}{q^{k(k+1)/2}\mathscr{P}_k(q)\over(q;q)_k}.
 \ea
 \label{eq:AB-residue-series}
 \ee
 Additionally, one has
 \be
 \label{Brogers}
 B_{\overline{\Sigma(2,3,7)}}(q)=2{\theta_{q^7}(q^5)\over(q;q)_\infty^4}.
 \ee

 Finally, one has the formal asymptotic expansion
\be
\label{as237bar}
\hat \tau_{\overline{\Sigma(2,3,7)}}(q) \sim \varphi_{\overline{\Sigma(2,3,7)}}(h), \qquad h\rightarrow 0^+,
\ee
where $\varphi_{\overline{\Sigma(2,3,7)}}(h)=\varphi_{\Sigma(2,3,7)}(-h)$ is the Ohtsuki series of $\overline{\Sigma(2,3,7)}$.
   \end{lemma}

The following conjecture relates the $q$-series $A_{\overline{\Sigma(2,3,7)}}(q)$ to a mock theta function. 
  \begin{conjecture} The $q$-series $A_{\overline{\Sigma(2,3,7)}}(q)$ is given by
 \be
 A_{\overline{\Sigma(2,3,7)}}(q)= {\mathcal F}_0(q)+\Sigma_{\overline{\Sigma(2,3,7)}}(q),
 \ee
 where
\begin{equation}
  \mathcal F_0(q)
  =
  \sum_{n\geq0}{q^{n^2}\over(q^{n+1};q)_n}
  \label{eq:extra-F0-def}
\end{equation}
is Ramanujan's seventh-order mock theta function, and
 \be
 \label{sigma1conj}
  \Sigma_{\overline{\Sigma(2,3,7)}}(q)
  =
\left[ {1-25E_2(q)\over48} -{7 \over 2} \,q{\rd \over\rd q} \right]  B_{\overline{\Sigma(2,3,7)}}(q).
  \ee
\end{conjecture}
This conjecture can be motivated in various ways. First of all, the appearance of the mock modular form 
$\CF_0(q)$ is easy to see from the Mellin--Barnes representation of $ \hat \tau_{\overline{\Sigma(2,3,7)}} (q,z)$, 
given by
\begin{equation}
 \hat \tau_{\overline{\Sigma(2,3,7)}} (q,z)
  =
  {\pi \over\log q\,(q;q)^2_\infty}
   \int_{\IR-\ii\epsilon}
  z^{\ri s-1}
  {
  (q^{1-\ri s};q)_\infty(q^{\ri s};q)_\infty
  \over
  \sinh^2(\pi s)
  (-q^{\ri s};q)_\infty(q^{1+2\ri s};q^2)_\infty}
  \,\rd s.
  \label{eq:extra-f1-vertical-integral}
\end{equation}
We can evaluate (\ref{eq:extra-f1-vertical-integral}) by closing the contour
in the upper half-plane. There are poles at $s=\ri \IZ_{\ge 0}$, and two additional pole families at
\be
\ba
  s&=-{(2a+1)\pi\over h}+\ri r,
  &
  a&\in\IZ,\quad r\geq0,
  \\
  s&=-{\pi a\over h}
  +\ri\left(r+{1\over2}\right),
  &
  a&\in\IZ,\quad r\geq0.
  \ea
  \label{eq:extrapoles7}
\ee
 The contribution of the poles at $s=\ri \IZ_{\ge 0}$ gives precisely (\ref{eq:extra-F0-def}). To obtain a conjecture for $\Sigma_{\overline{\Sigma(2,3,7)}}(q)$ it is better to use the Laplace transform. The result (\ref{lcb}) implies in this case
\be
\ba
- \mL_{-1} \left[ (X+X^{-1}-2) A_{{\bf 4}_1,x}(q) \right]&= {\mathcal F}_0(q)+\Sigma_{\overline{\Sigma(2,3,7)}}(q), \\
 -  \mL_{-1} \left[ (X+X^{-1}-2) B_{{\bf 4}_1,x}(q) \right] &= B_{\overline{\Sigma(2,3,7)}}(q).
  \ea
  \ee
As shown in (\ref{a-easy}), the $(q,X)$-series $A_{{\bf 4}_1,x}(q)$ is the sum of two $(q,X)$-series, but they 
are separately singular at $X=1$. The problem of calculating the Laplace transform of such singular $(q,X)$-series 
was considered in \cite{MR4233201, park-inverted} and follow-up works, with the goal of constructing 
$\hat Z$-invariants of manifolds obtained by surgeries on the figure-eight knot. In that case, one has to first 
expand in $X$ around $X=0$. In the case at hand, the resulting formal series are in the space $\CA_{1,1}$ defined in \eqref{eq:laplace-domain}, so one can apply the Laplace transform. One has \cite{park-inverted}
 \be
 -\mL_{-1} \left[ (X+X^{-1}-2) \CC(q,X) \right]= \CF_0(q).
 \ee
 To determine $\Sigma_{\overline{\Sigma(2,3,7)}}(q)$, one has to calculate
 \be
 \mL_{-1} \left[ (X+X^{-1}-2)\Phi(q,X) B_{{\bf 4}_1,x}(q) \right].
 \ee
 Since $\Sigma_{\overline{\Sigma(2,3,7)}}(q)$ is obtained by applying the Laplace transform to $B_{{\bf 4}_1,x}(q)$, after insertion of
  the quasi-modular Jacobi form $\Phi(q,X)$, this suggests that $\Sigma_{\overline{\Sigma(2,3,7)}}(q)$ is a linear combination of $B_{\overline{\Sigma(2,3,7)}}(q)$, $q \rd B_{\overline{\Sigma(2,3,7)}}(q)/\rd q$ and $E_2 (q) B_{\overline{\Sigma(2,3,7)}}(q)$. After fixing the coefficients with the very first terms of the $q$ expansion, one finds (\ref{sigma1conj}).

  Let us now comment on the relationship to other approaches. It is commonly stated that the $\hat Z$-invariant of $\overline{\Sigma(2,3,7)}$
  is $\CF_0(q)$ \cite{3dmod,MR4233201}. This $q$-series does not have a well-defined radial limit when $q\to 1^{-}$, and, in particular, does not reproduce the Ohtsuki series for this manifold. As we have seen above,
  the deformed Borel--Laplace resummation of~\eqref{eq:extra-f1-def} differs from $\CF_0(q)$ by a non-trivial correction. In particular, this correction guarantees the asymptotic expansion (\ref{as237bar}). From the point of view of the Laplace transform, this correction is inherited from the correction to $\CC(q,X)$ in the analytic invariant of the figure-eight knot.

 \subsubsection{Manifold $\overline{\Sigma(2,3,5)}$} Another interesting example is the Poincar\'e sphere with reversed orientation,
 $\overline{\Sigma(2,3,5)}$. This manifold can be obtained by $+1$ surgery on the trefoil
 knot ${\bf 3}_1^-$, so in this case $p=-1$, $r=1$ and Proposition \ref{tauhat-gen} cannot be
 used to calculate the corresponding analytic function. In fact, (\ref{lcb}) cannot be used, since
 the Laplace transform of $\hat J_{ {\bf 3}_1^-,x}(q)$ with a positive surgery is not well-defined as
 a $q$-series. However, one can proceed by direct calculation, as follows. 
 
 The generating series
 (\ref{tauz-1}) is in this case
 \be
\tau_{\overline{\Sigma(2,3,5)}}(q, z)
  =
  2 q^{-2}
  \sum_{m\geq0}
  (-1)^m q^{-(3m^2+7m)/2}(q^{m+1};q)_{m+1}z^m,
  \label{eq:extra-f3-def}
\end{equation}
where we have included an extra factor of $q^{-1}$ for convenience, as compared to the definition in (\ref{wrt-laplace}). 
This series is clearly divergent for $q<1$. By performing a $q^3$-Borel--Laplace resummation,
one obtains the following Mellin--Barnes representation: 
\begin{equation}
\begin{aligned}
\hat \tau_{\overline{\Sigma(2,3,5)}}(q,z)
={}&{\pi q^{-2}(-q;q)_\infty(q;q^2)_\infty
\over 3\log q\,(q^3;q^3)_\infty^3}
\int_{\IR-\ii\epsilon}
\left({z\over q^2}\right)^{\ri s-1}
\\
&\quad\times
{\theta_{q^3}(q^{3\ri s})
\over
\sinh^2(\pi s)(-q^{\ri s};q)_\infty
(q^{1+2\ri s};q^2)_\infty}
\,\rd s.
\end{aligned}
\label{eq:poincare-MB}
\end{equation}
When $z=1$ one finds the following $q$-series expression, 
\be
\label{poincare}
\hat \tau_{\overline{\Sigma(2,3,5)}}(q)=A _{\overline{\Sigma(2,3,5)}}(q)  + {1\over h} B _{\overline{\Sigma(2,3,5)}}(q),
\ee
where
\be
A _{\overline{\Sigma(2,3,5)}}(q)=\chi_0(q)+  \Sigma_{\overline{\Sigma(2,3,5)}}(q),
  \ee
and
 \begin{equation}
  \chi_0(q)
  =
  \sum_{n\geq0}{q^n\over(q^{n+1};q)_n}
  \label{eq:chi0-def}
\end{equation}
is Ramanujan's fifth-order mock theta function. As in the examples above, it is obtained 
from the residues of the integrand of \eqref{eq:poincare-MB} at the poles
$s=\ri n$, $n\in \IZ_{\ge 0}$. The
series $\Sigma_{\overline{\Sigma(2,3,5)}}(q)$ and $B _{\overline{\Sigma(2,3,5)}}(q)$ come from additional poles, and at the very first orders one finds
\be
\ba
 \Sigma_{\overline{\Sigma(2,3,5)}}(q)
  &=
  -{1\over3}-{11\over3}q-{5\over3}q^2
  -{28 \over3}q^3-{41\over3}q^4-17 q^5+\CO(q^6), \\
 B _{\overline{\Sigma(2,3,5)}}(q)  &=
  {2\over3}+{4\over3}q+{4\over3}q^2
  +{14\over3}q^3+{22\over3}q^4+{28\over3}q^5+\CO(q^6).
  \ea
  \label{eq:sigma3s}
\ee
The series $\chi_0(q)$ (or rather $2-\chi_0(q)$) is the $\hat Z$-invariant for 
$\overline{\Sigma(2,3,5)}$ \cite{3dmod}. However, as in the example of $\overline{\Sigma(2,3,7)}$, 
it does not have a well-defined radial limit as $q \to 1^{-}$. It can be checked that the corrections in (\ref{poincare}) 
make it possible to obtain the correct perturbative series $2h-9 h^2+ \CO(h^3)$ as $h\to 0^+$.

\subsection{A hyperbolic closed three-manifold}

Our final example is the hyperbolic three-manifold $M_{-1/2}$ obtained by $-1/2$ surgery on the figure-eight knot. This manifold was considered, for example, in~\cite{MR4233201} and studied in great detail in~\cite{wheeler}. In this case, one can show that 
\be
 \CF_{-1/2,m}(q)
 =-2q^{1/8}(q;q)_{2m+1}
 \sum_{\ell=0}^{m}
 \frac{q^{\ell(\ell+1)}}
 {(q;q)_\ell(q;q)_{m-\ell}},
 \ee
 and the WRT invariant is then (see e.g. \cite{wheeler}):
\be
  \tau_{M_{-1/2}}(q)
  =
  2\sum_{m\geq0}\frac{(q;q)_{2m+1}}{ \gamma_q(-m)}\sum_{\ell=0}^m
  {q^{\ell(\ell+1)}\over
 (q;q)_\ell(q;q)_{m-\ell}},
  \label{eq:corr-start}
\end{equation}
where we have included an overall factor $-q^{-1/8}$ for convenience.

\begin{lemma}\label{lem:analytic-m-1/2} The analytic function of $M_{-1/2}$, obtained by $-1/2$ surgery on the figure-eight knot, is given by
the Mellin--Barnes integral
 \be
 \label{mb-M12}
  \hat \tau_{M_{-1/2}}(q,z)=
  {\pi\over \log q\,(q;q)_\infty^2}
  \sum_{\ell\geq0}
  {q^{\ell(\ell+1)}\over(q;q)_\ell}
  \int_{\IR-\ri\epsilon} {(q^{1-\ell-\ri s};q)_\infty(q^{\ri s};q)_\infty z^{\ri s -1+\ell}
  \over
  \sinh^2(\pi s)
  (-q^{\ell+\ri s};q)_\infty
  (q^{2\ell+1+2\ri s};q^2)_\infty}
  \rd s,
  \ee
which can be evaluated to
\be
\hat \tau_{M_{-1/2}}(q,z)= A_{M_{-1/2}}(q,z)+{\log z \over h}  B_{M_{-1/2}}(q,z),
\ee
where
\be
\ba
 A_{M_{-1/2}}(q,z)
 &=-{1\over(q;q)_\infty^3}
 \sum_{k\geq0} q^k \mathscr{Q}_k(q)
 \sum_{n\in\mathbb Z}\gamma_q(-n)
 {k+n\over1-zq^{k+n}},\\
 B_{M_{-1/2}}(q,z)
 &={1\over(q;q)_\infty^3}
 \sum_{k\geq0} q^k \mathscr{Q}_k(q)
 \sum_{n\in\mathbb Z}{\gamma_q(-n)\over1-zq^{k+n}},
 \ea
 \label{eq:M12AB}
 \ee
\be
\mathscr{Q}_k(q)= q^{-1/8} D_k^{(1,2)}(q),
\ee
and the coefficient $D_k^{(1,2)}(q)$ is given in (\ref{eq:residue-12}). 

When $z=1$, we have
 \be
 \hat \tau_{M_{-1/2}}(q)
 =A_{M_{-1/2}}(q)+{1\over h}B_{M_{-1/2}}(q),
 \label{eq:hyper-AB}
\end{equation}
where
\be
\ba
A_{M_{-1/2}}(q)
 &=-{1\over (q;q)_\infty^3}
 \sum_{k\geq0}q^k \mathscr{Q}_k(q)
 \sum_{\substack{n\in\mathbb Z\\n\neq-k}}
 \gamma_q(-n){k+n\over1-q^{k+n}},\\
 B_{M_{-1/2}}(q)
 &=-{1\over (q;q)_\infty^3}
 \sum_{k\geq0}\gamma_q(-k)\mathscr{Q}_k(q).
 \ea
 \label{eq:hyper-AB-series}
 \ee

One has the formal asymptotic expansion
\be
\label{asM12}
\hat \tau_{M_{-1/2}}(q) \sim \varphi_{M_{-1/2}}(h), \qquad h\rightarrow 0^+,
\ee
where
\be
\varphi_{M_{-1/2}}(h)=2h- 25 h^2+ {1621 \over 3} h^3+ \CO(h^4)
\ee
is the Ohtsuki series of $M_{-1/2}$.
   \end{lemma}
See Appendix~\ref{app:B} for the proof.

As in the previous examples, calculation of the contour integral \eqref{mb-M12} by residues gives 
additional information and a comparison with
other analytic invariants. By closing the contour in the upper half-plane, we find poles at $s =\ri \IZ_{\ge 0}$, and at
\be
\ba
  s&=-{(2a+1)\pi\over h}+\ii(\ell+j),
  &
  a&\in\IZ,\quad j\geq0,
   \\
  s&=-{\pi a\over h}
  +\ii\left(\ell+j+{1\over2}\right),
  &
  a&\in\IZ,\quad \ell, j\geq0.
\ea
  \label{eq:corr-half-poles}
  \ee
 By summing over the residues at $z=1$, we obtain a decomposition
 \be
 \label{zm12}
A_{M_{-1/2}}(q)=\hat Z_0 \left(M_{-1/2}\right) +\Sigma_{M_{-1/2}}(q),
  \ee
where the first term comes from the contribution of the poles at $s=\ri \IZ_{\ge 0}$ and is given by
\begin{equation}
\ba
  \hat Z_0 \left(M_{-1/2}\right)&=
  \sum_{\ell,k\geq0}
  (-1)^\ell
  q^{(\ell^2+\ell-2\ell k+2k^2)/2}
  {(q;q)_{\ell+k}\over(q;q)_\ell(q;q)_{2k}}\\
  &= 1-q+2q^3-2q^6+q^9+3q^{10}+q^{11}
  -q^{14}-3q^{15}-q^{16}+\cdots.
  \ea
   \label{z0gm}
  \ee
The second term can be written explicitly but its structure seems to be complicated. For the first few terms, one finds
\be
 \Sigma_{M_{-1/2}}(q)=-1-5 q-22 q^2-78 q^3 -236 q^4-634 q^5 +\CO(q^6).
  \label{hyper-sigma}
\ee
The $q$-series (\ref{z0gm}) is the $\hat Z$-invariant for this manifold obtained in \cite{MR4233201}. The closed form above was obtained in \cite{wheeler}, where it was also noted that it does not have a radial limit as $q \to 1^{-}$. In contrast, the formal asymptotic expansion of our analytic function is given by the Ohtsuki series; as in the previous examples, 
the corrections coming from the
additional poles (\ref{eq:corr-half-poles}) are crucial for this. The series  $ \Sigma_{M_{-1/2}}(q)$, in principle, can be
computed by the Laplace transform, as in the case of $\overline{\Sigma(2,3,7)}$. It is given by
\be
  \Sigma_{M_{-1/2}}(q)= -q^{-1/8}  \mL_{-1/2} \left[ (X^{1\over2} -X^{-{1 \over 2}} )(X^{1\over4} -X^{-{1 \over 4} })
\Phi(q, X) B_{{\bf 4}_1,x}(q) \right].
 \ee
However, in this case we have not been able to find a closed formula for $\Sigma_{M_{-1/2}}(q)$. 

\section*{Acknowledgements}

We would like to thank Stavros Garoufalidis, Jie Gu and Campbell Wheeler for discussions. V.M. gratefully acknowledges the hospitality of the University of Geneva during September--October 2026, when part of this work was completed. We acknowledge the use of Mathematica for assistance with symbolic calculations, and of ChatGPT (Open AI) for assistance with computations, derivations, and text editing.

\noindent{\bf Funding.} This work is partially supported by the SNSF research program NCCR The Mathematics of Physics (SwissMAP), the SNSF grants no.~200021-232258, no.~200020-200400 and no.~2000-1-240084.

\appendix
\section{Some useful formulas}
\label{app:A}

\be
(a;q)_n=\prod_{i=0}^{n-1}(1-aq^i), \quad  (a;q)_{-n}=(aq^{-n};q)_n^{-1},\quad n\geq0.
\ee
\begin{equation}\label{eq:gamma}
\mg{q}{n}:=(-1)^n q^{n(n-1)/2}.
\end{equation}
\be
(a_1,\ldots,a_l;q)_k=\prod_{i=1}^l(a_i;q)_k,\quad
(a;q)_{k_1,\ldots,k_l}=\prod_{i=1}^l(a;q)_{k_i}.
\ee
\begin{equation}\label{eq:theta-q2}
\theta_q(z):=\sum_{n\in\Z}\gamma_q(n)z^n=(q,z,q/z;q)_\infty.
\end{equation}
\be\label{eq:phirs}
{_r\phi_s}\Bigl[
\begin{matrix}
 a_1,\dots,a_r\\
 b_1,\dots,b_s
\end{matrix};q,z
\Bigr]=\sum_{m\ge0}\frac{(a_1,\dots,a_r;q)_m}{(b_1,\dots,b_s,q;q)_m}\gamma_q(m)^{1+s-r}z^m.
\ee
\begin{equation}\label{Rsq}
R_s(q):=\int_0^\infty\frac{ t^{s}}{(-t;q)_\infty}\rd t=-\frac{(q^{-s};q)_\infty}{(q;q)_\infty}\frac{\pi}{\sin(\pi s)}=\frac{1}{(q;q)_{-s-1}}\frac{\pi}{\sin(-\pi s)}.
\end{equation}
\be
\int_{\ri\R-\epsilon}\frac{a^s}{\frac{2\ri}{\pi}\sin^2(\pi s)}\rd s=\frac{\log a}{a-1}.\label{integralA}
\ee

\section{Proofs}
\label{app:B}

\begin{proof}[Proof of Theorem~\ref{thm2}]

Let us write
\be\label{eq:41cont-int1}
\hat J_{{\bf 4}_1,x}(q,z)=
\frac{1}{2\log q}\frac{(q^{1+x},q^{1-x};q)_\infty}{(q;q)_\infty^2}I_{{\bf 4}_1,x}(q,z),
\ee
where
\be
I_{{\bf 4}_1,x}(q,z)=
\int_{\R-\ri\epsilon}\frac{(q^{\ri s},q^{1-\ri s};q)_\infty}{(q^{\ri s+x},q^{\ri s-x};q)_\infty}\frac{\pi z^{\ri s-1}}{(\sinh(\pi s))^2}\rd s. \label{integ1}
\ee
We can close the contour in the upper half-plane because the integrand on the contour $C_R=\{s| s=R\re^{\ri\phi}, \pi \ge \phi \ge 0\}$ behaves like
\be
\exp\left(-\frac{h}{2}(\Im s)^2+O(\Im s)\right),
\ee
which follows from the asymptotics
\be
\Re [\log(q^s;q)_\infty]=\frac{h}{2}(\Re s)^2-\frac{h}{2}\Re s +O(1). 
\ee
We use the corrected form of \cite[(4.3.6)]{MR2128719}; the sign of the linear term in the cited formula appears to be a typo.

There are three series of simple poles:
\be
s=\pm \ri x+\ri n+\frac{2\pi m}{h}\quad\mbox{and} \quad s=\ri n, \quad   n\in\Z_{\ge0},  \quad m\in\mathbb{Z}.
\ee
For the first two series of poles, $q^{\ri s}$ does not depend on $m$ and the sums over $n$ and $m$ factorize.
The poles $s=\ri x+\ri n+\frac{2\pi m}{h}$ give the following contribution to $I_{{\bf 4}_1,x}(q,z)$:
\be
-2\pi^2 \ri \left\{\sum_{n=0}^\infty
\frac{(q^{-x-n},q^{1+x+n};q)_\infty}{(q^{-2x-n};q)_\infty z^{n+1}} \mbox{Res}\left[\frac{1}{(q^{\ri s+x};q)_\infty}\right]_{s=\ri x+\ri n}\right\}\sum_{m\in\mathbb{Z}}\frac{z^{-x+\frac{2\pi \ri m}{h}}}
{\sin^2(\pi(x-\frac{2\pi \ri m}{h}))}.\label{contrib1}
\ee
Using
\be
\mbox{Res}\left[\frac{1}{(q^{\ri s+x};q)_\infty}\right]_{s=\ri x+\ri n}=\frac{\ri}{\log q} \frac{\gamma_q(-n)}{(q;q)_n(q;q)_\infty},
\ee
we rewrite \eqref{contrib1} as
\be
\frac{2\pi^2}{z\log q }\frac{\theta_q(q^{-x})}{(q;q)_\infty^2(q^{-2x};q)_\infty}\,
{_1\phi_1}\Bigl[
\begin{matrix}
 0\\
q^{1+2x}
\end{matrix}
;q,\frac{q^{1+x}}{z}\Bigr]\fn{2,q}(z,x). \label{Contrib1a}
\ee
The second contribution from the poles  $s=-\ri x+\ri n+\frac{2\pi m}{h}$ is obtained from \eqref{Contrib1a} by replacing
$x\to-x$. Finally, we have the contribution from the poles $s=\ri n$. These give
\be
2\pi \ri \sum_{n=0}^\infty \frac{(q^{1+n};q)_\infty}{(q^{-n+x},q^{-n-x};q)_\infty}\frac{\pi}{z^{n+1}}
\mbox{Res}\left[\frac{(q^{\ri s};q)_\infty}{\sinh^2(\pi s)}\right]_{s=\ri n}.\label{Contrib2}
\ee
Using
\be
\mbox{Res}\left[\frac{(q^{\ri s};q)_\infty}{\sinh^2(\pi s)}\right]_{s=\ri n}=-\frac{\ri\log q}{\pi^2}\frac{(q;q)_n(q;q)_\infty}{\gamma_q(-n)},
\ee
we rewrite \eqref{Contrib2} as
\be
\frac{2\log q}{z}\frac{(q;q)_\infty^2}{(q^x,q^{-x};q)_\infty}\,
{_2\phi_2}\Bigl[
\begin{matrix}
 0,q\\
q^{1+x},q^{1-x}
\end{matrix}
;q,\frac{q}{z}\Bigr].
\ee
Putting these contributions together, we obtain
\begin{align}
&I_{{\bf 4}_1,x}(q,z)=\frac{2\log q}{z}\frac{(q;q)_\infty^2}{(q^x,q^{-x};q)_\infty}\,
{_2\phi_2}\Bigl[
\begin{matrix}
 0,q\\
q^{1+x},q^{1-x}
\end{matrix}
;q,\frac{q}{z}\Bigr]+\\
\frac{2\pi^2\theta_q(q^{-x})}{z\log q(q;q)_\infty^2}
&\left\{ \frac{\fn{2,q}(z,x)}{(q^{-2x};q)_\infty}\,
{_1\phi_1}\Bigl[
\begin{matrix}
 0\\
q^{1+2x}
\end{matrix}
;q,\frac{q^{1+x}}{z}\Bigr]-
\frac{q^x\fn{2,q}(z,-x)}{(q^{2x};q)_\infty}\,
{_1\phi_1}\Bigl[
\begin{matrix}
 0\\
q^{1-2x}
\end{matrix}
;q,\frac{q^{1-x}}{z}\Bigr]\right\},\nonumber
\end{align}
and this leads to (\ref{J41-res}), which we have written in a form which is manifestly symmetric under $x \leftrightarrow -x$.
\end{proof}

\begin{proof}[Proof of Corollary~\ref{cor:hatJ4_1-evaluates}] In order to evaluate (\ref{J41-res}), we use the explicit result~\eqref{ident2} for the function $\fn{2,q}(z,x)$
for evaluation at $z=1$. To resolve the removable singularities 
appearing in the r.h.s. of~\eqref{ident2}, we set $z=1-\epsilon$ and we take the
limit $\epsilon\to0$. After simple calculations we obtain
\be
\label{f2z1}
\pi^2 \fn{2,q}(1,x)=h +\psi_q^{(1)}(x)+\psi_q^{(1)}(1-x),
\ee
where $\psi_q^{(1)}(x)$ is the first derivative of the $q$-digamma function and can be presented by the following
Lambert series
\be
\psi_q^{(1)}(x)=h^2\sum_{k=1}^\infty \frac{k q^{kx}}{1-q^k}.
\ee
The sum of the two Lambert series appearing in (\ref{f2z1}) can be expressed in terms of (\ref{quasi-mod})
\be
\label{Phi-mod}
\pi^2 \fn{2,q}(1,x)= h + h^2 \Phi(q, X)= (2 \pi \ri)^2 \Phi(\tq, \tX),
\ee
where $\tq$ is defined in (\ref{tq-def}) and
\be
\label{tX-def}
\tX = \re^{-{2 \pi \ri  \over h}\log X}=\re^{2\pi\ri x}.
\ee
After some simple algebra, we arrive at~\eqref{jcz}. 
\end{proof}

\begin{proof} [Proof of Corollary~\ref{cor-31}]

First we evaluate
$\fn{2,q}(-q^{-1},x)$ from (\ref{ident2}). By using the functional equation $\theta_q(q z)= -z^{-1} \theta_q(z)$, we find
\be
 \fn{1,q}(1/q,x) =
  \frac{h (q;q)_\infty^3}{\pi\theta_q(-1)} R(q^x), \qquad R(X)=
 X \frac{\theta_q(-X)}{\theta_q(X)}.
  \ee
From Jacobi's triple product identity for the theta function we find
\be
  X\frac{\partial}{\partial X}\log R(X)=1+2S(q,X),
  \ee
  where
\be
 S(q,X)
  =X\frac{(q^2;q^2)_\infty^4}{(q;q)_\infty^2}\frac{\theta_{q^2}(qX^2)}{\theta_{q^2}(X^2)}\label{S-func},
  \ee
  and we conclude that
  \be
  \label{ex-fq2}
  \fn{2,q}(-q^{-1},x)
  =
  \frac{h(q;q)_\infty^3}{\pi^2\theta_q(-1)}
  \frac{q^x\theta_q(-q^x)}{\theta_q(q^x)}
  \left(
    2h S(q,q^x)\pm \ri\pi \right),
    \ee
    where we have explicitly written down two choices of branch in
    \be
    \label{log-branch}
      \log(-1/q)=h\mp \ri\pi.
      \ee
It is easy to check that the functions $H_q(z)$, introduced in \eqref{HT-def}, satisfy
\be
H_q(z)+H_q(q/z)-1=\theta_q(z), \quad H_q(z)+zH_q(qz)=1.\label{H-prop}
\ee
Combining these equations together, we get
\be\label{H-prop1}
X^2H_q(qX^2)-H_q(q/X^2)+\theta_q(X^2)=0.
\ee
Using that
\be
R(X^{-1})=-X^{-2} R(X),\quad S(X^{-1})= -S(X),\quad \theta_q(-1)=2(q;q)_\infty(-q;q)_\infty^2,
\ee
and \eqref{S-func}, \eqref{HT-def},
we can write the second term of (\ref{anal3J}) as
      \be
      M(q,X) \pm \frac{\ri \pi}{2h}N(q,X),
\ee
where
\be
  M(q,X)
  =-\frac{q}{2(1-X)}\frac{\theta_q(-X)\theta_{q^2}(qX^2)}
{(q^2;q^2)_\infty\theta_q(X^2)}
 \left(X^2 H_q(qX^2)
  +H_q(q/X^2)
  \right)\label{Mq-def}
\ee
 and
 \be
  N(q,X)
  =
  -\frac{
    (q;q)_\infty^2\theta_q(-X)\theta_{q^2}(X^2)}
 {2(q^2;q^2)^5_\infty X(1-X)\theta_q(X^2)} \left(X^2 H_q(qX^2)
  -H_q(q/X^2)\right).\label{Nq-def}
\ee
To simplify \eqref{Mq-def} and \eqref{Nq-def},
we substitute $H_q(qX^2)$ from \eqref{H-prop1}, use definition \eqref{HT-def}
and the identity
\be
\frac{\theta_q(-X)\theta_{q^2}(qX^2)}
 {(q^2;q^2)_\infty\theta_q(X^2)}=\frac{1}{(X,q/X;q)_\infty}=\frac{1}{P_q(X)}.\label{TH-ident}
\ee
This leads to (\ref{finalJ3}). 
\end{proof}
\begin{proof}[Proof of Proposition~\ref{prop3}]

First we construct a telescoping representation for the
function $U_q(X)$, defined in~\eqref{eq:Uq1X} and~\eqref{Uq},
\be
U_q(X)=\frac{X}{1-X}\sum_{n\geq0}\frac{q^n}{(qX,q/X;q)_n}.
\ee
Let
\be
A_n(X)=\frac{1}{(X,q/X;q)_n},\quad G_n(X)=A_n(X)\frac{X(X-q^n)}{1-Xq^n}.
\ee
One can check the following identity
\be
\frac{q^nX}{(X;q)_{n+1}(q/X;q)_n}+\frac{q^{n+2}X^4}{(qX;q)_{n+1}(1/X;q)_n}=
(1-qX^2)\left(G_{n+1}(X)-G_n(X)\right).\label{ident3}
\ee
Summing \eqref{ident3} over $n\geq0$ and taking into account that
\be
G_0(X)=-X,\quad G_\infty(X)=\frac{X^2}{P_q(X)},
\ee
with $P_q(X)$ defined by \eqref{P-def},
we come to the recurrence relation for $U_q(X)$
\be
U_q(X)+qX^3U_q(qX)=X(1-qX^2)\bigl(1+\frac{X}{P_q(X)}\bigr).\label{ident4}
\ee
Iterating \eqref{ident4} $k$ times, we obtain
\be
U_q(X)=\gamma_{q^3}(k)q^kX^{3k}U_q(q^kX)+
\sum_{s=0}^{k-1}\gamma_{q^3}(-s)q^{-s}X^{3s+1}(1-q^{1+2s}X^2)\bigl[1+\frac{q^s X}{P_q(q^s X)}\bigr].
\ee
Now taking the limit $k\to\infty$, and using the relation $U_q(0)=0$, together with
\be
P_q(q^kX)=\gamma_{1/q}(k)X^{-k} P_q(X),
\ee
we  obtain
\be
U_q(X)=H_{q^3}(qX^3)-1+X H_{q^3}(q^2X^3)+\frac{X^2}{P_q(X)}H_q(qX^2).\label{Uq-new}
\ee
Now consider Watson's quintuple product identity (see~\cite{Watson} or \cite[Ex. (5.6)]{MR2128719})
\be
\sum_{m\in\mathbb{Z}}\gamma_{q^3}(m)q^mz^{3m}(1+zq^m)=(q,-z,-q/z;q)_\infty(qz^2,q/z^2;q^2)_\infty.\label{Watson}
\ee

Setting $z=X$ and using the identity \eqref{TH-ident} together with two definitions of the theta-function,
we rewrite \eqref{Watson} as
\be
\theta_{q^3}(qX^3)+X\,\theta_{q^3}(q^2X^3)=\frac{\theta_q(X^2)}{P_q(X)}.\label{Watson1}
\ee
From \eqref{H-prop} we can also get
\be
X^2H_q(qX^2)=T_q(q/X^2)-\frac{1}{2}\theta_q(X^2).\label{ident5}
\ee
Substituting \eqref{ident5} in \eqref{Uq-new} and using \eqref{Watson1} we obtain
\begin{align}
&U_q(X)=H_{q^3}(qX^3)-1+X H_{q^3}(q^2X^3)+\frac{T_q(q/X^2)}{P_q(X)}-\frac{1}{2}\frac{\theta_q(X^2)}{P_q(X)}\nonumber\\
&=\left(H_{q^3}(qX^3)-\frac{1}{2}\theta_{q^3}(qX^3)\right)+X\left(H_{q^3}(q^2X^3)-\frac{1}{2}\theta_{q^3}(q^2X^3)\right)
-1+\frac{T_q(q/X^2)}{P_q(X)}\label{Uq-new1}\\
&=T_{q^3}(qX^3)+X\,T_{q^3}(q^2X^3)-1+\frac{T_q(q/X^2)}{P_q(X)}.\nonumber
\end{align}
Substituting \eqref{Uq-new1} into \eqref{AJ3}, we conclude that it coincides with  \eqref{GM-exp}.
\end{proof}

\begin{proof}[Proof of Lemma~\ref{lem:div-lap-int}]
 Using Watson's integral formula for the inverse of infinite $q$-Pochhammer symbol
 $$
 \frac{1}{(z;q)_\infty}=\int_{\ri\R-\epsilon}\frac{(-z)^s\rd s}{(q;q)_s2\ri\sin(-\pi s)},
 $$
 we write
 $$
 F_q(z)=\sum_{m\ge0}\int_{\ri\R-\epsilon}\frac{q^{m(m+1)}z^m(-zq^m)^s\rd s}{(q;q)_{m,s}2\ri\sin(-\pi s)},
 $$
 so that, for $0<p<1$,
 \be
 \mathcal L_p F_q(z)=\sum_{m\ge0}\int_{\ri\R-\epsilon}\frac{q^{m(m+1)}(zq^s)^m(-z)^s R_{s+m}(p)\rd s}{(q;q)_{m,s}2\ri\sin(-\pi s)},
 \ee
 where $R_s(p)$ is defined in (\ref{eq:t-s-int}). By using its explicit value, and after some further algebraic simplifications, we find 
 \be
 \ba
  \mathcal L_p F_q(z) &=\sum_{m\ge0}\int_{\ri\R-\epsilon}\frac{(p;p)_{s+m}\theta_p(p^{-s})}{(p;p)_\infty^3p^{sm}}\frac{\mg{q^2/p}{-m}q^{sm}(-z)^{s+m}\rd s}{(q;q)_{m,s}\frac{2\ri}{\pi}\sin^2(\pi s)}
\\
& =\int_{\ri\R-\epsilon}\sum_{m\ge0}\frac{(p^{1+s};p)_{m}\bigl(-z(\frac q p)^{s}\bigr)^m}{(q;q)_{m}\mg{pq^{-2}}{-m}}\frac{(p;p)_{s}\theta_p(p^{-s})(-z)^{s}\rd s}{(p;p)_\infty^3(q;q)_{s}\frac{2\ri}{\pi}\sin^2(\pi s)}
 \\
&=\int_{\ri\R-\epsilon}\frac{S_{p,q}\bigl(p^{1+s},z (q/p)^{s}\bigr)(-z)^{s}\rd s}{(p;p)_{-1-s}(q;q)_{s}\frac{2 \ri}{\pi}\sin^2(\pi s)},
\ea
\ee
thus obtaining~\eqref{eq:div-ser-expr}.
 \end{proof}

\begin{proof}[Proof of Lemma~\ref{lem:bl-resum-c}]
We consider
 \begin{align}
& \tilde S_{q^4,q}(x,y):=(\mathcal B_{q^2}S_{q^4,q}(x,\cdot))(y)=\sum_{m\ge0}\frac{(x;q^4)_{m}y^m}{(q;q)_{m}(q^2;q^2)_{m}}\nonumber\\
& =\Bigl(\sum_{m\ge0}(x;q^4)_{m}y^m\Bigr)\had_y(y;q)_\infty^{-1}\had_y(y;q^2)_\infty^{-1}.
 \end{align}
By applying \eqref{eq:res-form-hadamart-pr} to the first two factors, we obtain
\bea
\tilde S_{q^4,q}(x,y)=&
 \sum_{k,n\ge0}\frac{(x;q^4)_{\infty}x^n\mg{q}{-k}}{(q^4;q^4)_n(q;q)_{k,\infty}(yq^{4n+k};q^2)_\infty}
 =\sum_{n\ge0}\sum_{k\ge4n}\frac{(x;q^4)_{\infty}x^n
 \mg{q}{4n-k}}{(q^4;q^4)_n(q;q)_{k-4n,\infty}(yq^{k};q^2)_\infty}\nonumber\\
& =\sum_{k\ge0}\sum_{0\le n\le\frac{k}{4}}\frac{(x;q^4)_{\infty}x^n\mg{q}{4n-k}}{(q^4;q^4)_n(q;q)_{k-4n,\infty}}(\mathcal{B}_{q^2}\me{q^2})(-yq^k),\label{sum1}
\end{align}
see~\eqref{eq:B_qmeqz}. Applying $\mathcal L_{q^2}$ to \eqref{sum1}, using~\eqref{eq:def-rme} and~\eqref{eq:q-b-l-resum-mu_q}, and shifting the summation variable $k$, we obtain
\be
-2(\log q) \hat S_{q^4,q}(x,y)
=\sum_{n\ge0}\sum_{l\in\Z}\sum_{k\ge 4n-2l}\frac{(x;q^4)_{\infty}x^n\mg{q^2}{l}\mg{q}{4n-2l-k}}
 {(q^2;q^2)_{\infty}^3(q^4;q^4)_n(q;q)_{k+2l-4n,\infty}}\frac{\log(-yq^{k})}{-yq^{k}-1}.
\ee
\end{proof}

\begin{proof}[Proof of Lemma~\ref{lem:hatH52-BI}]
Applying ${\mathcal L}_{q^4}$ to \eqref{eq:triple-decomp-5-2}, substituting \eqref{eq:LF_AJ}
and using \eqref{eq:res-form-laplace} twice, we obtain
\begin{equation}
\begin{aligned}
\hat \mH_{{\bf 5}_2^-}(q,z)
={}&\mathcal P_q\frac{(q;q)_\infty}{4h(q^4;q^4)_\infty}
\sum_{\substack{t,m\ge0\\e,k\in\mathbb Z}}
\frac{(-1)^tq^{2t(t+1)+m}}
{(q^4;q^4)_t(q;q)_m}
A_{e,k}(q)I_{e,k}(q^{m+4t-1}z).
\end{aligned}
\label{eq:H52pre}
\end{equation}
Using the shift identity~\eqref{eq:Jshift}, after reindexing $e,k$, we obtain~\eqref{eq:H52B}, where
\begin{equation}
B_{e,k}(q):=
\sum_{m,t\ge0}
\frac{q^{m+4t}\gamma_{q^4}(t)}
{(q;q)_m(q^4;q^4)_t}
A_{e-m+1-4t,\,k-m+1-4t}(q).
\label{eq:BfromA}
\end{equation}
Substituting~\eqref{eq:Atriple} into~\eqref{eq:BfromA}, we obtain 
\begin{equation}
\begin{aligned}
B_{e,k}
={}&\sum_{\substack{m,t,n\ge0\\j,l\in\mathbb Z}}
\frac{q^{e+1+4j+4n}\gamma_{q^4}(t)
\gamma_{q^4}(j+n)\gamma_{q^2}(l+2n)}
{(q;q)_m(q^4;q^4)_t(q^4;q^4)_n}
\\
&\quad\times
\frac{
\gamma_q(e+1-4t+4j-m)
\gamma_q(m-k-1+4t-2l)
}{
(q;q)_{e+1-4t+4j-m}
(q;q)_{k+1-4t+2l-m}
}.
\end{aligned}
\label{eq:Bfive}
\end{equation}
For any integers $a,b\ge0$, one has the terminating identity
\begin{equation}
\sum_{m\ge0}
\frac{\gamma_q(a-m)\gamma_q(m-b)}
{(q;q)_m(q;q)_{a-m}(q;q)_{b-m}}
=
\frac{\gamma_q(a-b)}{(q;q)_a(q;q)_b},
\label{eq:mcollapse}
\end{equation}
which can be used in~\eqref{eq:Bfive} with
\[
a=e+1-4t+4j,\qquad b=k+1-4t+2l,
\]
to remove the $m$-sum.  Now, shifting
\[
j\mapsto j+t,\quad l\mapsto l+2t,
\]
and putting $u=n+t$, we can use  the identity
\[
\sum_{n=0}^{u}
\frac{\gamma_{q^4}(u-n)}
{(q^4;q^4)_{u-n}(q^4;q^4)_n}
=\delta_{u,0},
\]
which follows from
\[
(x;q^4)_\infty\frac1{(x;q^4)_\infty}=1.
\]
Hence the five-fold sum in~\eqref{eq:Bfive} reduces to the double sum and we come to \eqref{eq:Bdouble}.
\end{proof}

\begin{proof}[Proof of Theorem~\ref{thm:analytic-form-hatJ52}]
Let us define generating functions for the coefficients $B_{e,k}$ and $C_{e,k}$
\begin{equation}
\mathcal B(U,V):=\sum_{e,k\in\mathbb Z}B_{e,k}(q)U^eV^k,
\quad
\mathcal C_{q,x}(U,V):=\sum_{e,k\in\mathbb Z}C_{e,k}(x,q)U^eV^k,
\label{eq:BCUVdef}
\end{equation}
and the one for $I_{e,k}(z)$ with inverted $U$ and $V$,
\begin{equation}
\mathcal I(U,V;z)
:=
\sum_{e,k\in\mathbb Z}I_{e,k}(z)U^{-e}V^{-k}.
\label{eq:Jgen}
\end{equation}

The main observation is that the sum in \eqref{eq:J52C} can be evaluated using the identity
\be
\sum_{e,k\in\mathbb Z}C_{e,k}(x,q)I_{e,k}(z)=
\CT_{U,V}
\left[{\mathcal C}_{q,x}(U,V)\mathcal I(U,V;z)\right].\label{CTMaster}
\ee

We first sum $\mathcal B(U,V)$. In the first definition in~\eqref{eq:BCUVdef}, taking into account~\eqref{eq:Bdouble}, we perform the change of summation variables
$$
e \mapsto a-1-4j,
\quad
k \mapsto b-1-2l,
$$
so that \(a,b\ge 0\). Consequently,
\be
e-k+4j-2l=a-b,
\quad
U^eV^k=U^{-1}V^{-1}U^aV^bU^{-4j}V^{-2l}.
\ee
It follows that
\begin{equation}
\mathcal B(U,V)=
U^{-1}V^{-1}
\theta_{q^4}(U^{-4})\theta_{q^2}(V^{-2})\,
\mathcal S(U,V),
\label{eq:BUVpre}
\end{equation}
where
\begin{equation}
\mathcal S(U,V):=
\sum_{a,b\ge0}
\frac{q^a\gamma_q(a-b)}{(q;q)_a(q;q)_b}U^aV^b=
\frac{(qU,qV;q)_\infty}{(qUV;q)_\infty},
\label{eq:SUVdef}
\end{equation}
where, in the second equality, we first sum over $a$ to obtain $(q^{1-b}U;q)_\infty$, and then sum over $b$ using the $q$-binomial summation formula.
Therefore,
\begin{equation}
\mathcal B(U,V)=
\frac{
U^{-1}V^{-1}
\theta_{q^4}(U^{-4})\theta_{q^2}(V^{-2})
(qU,qV;q)_\infty
}{
(qUV;q)_\infty
}.
\label{eq:BUVproduct}
\end{equation}

Next, we turn to the calculation of $\mathcal C_{q,x}(U,V)$. We substitute~\eqref{eq:CdoubleB} into the second definition in~\eqref{eq:BCUVdef}. In the resulting expression, we first shift the summation variables according to $e\mapsto e+r$ and $k\mapsto k+r$. We then replace $r$ with a new summation variable $n=r-m$.
 This gives
\begin{equation}
\mathcal C_{q,x}(U,V)=
\frac{\mathcal B(U,V)}
{(q^{1+x}UV,q^{1-x}UV;q)_\infty}
\label{eq:CUVfromB}
\end{equation}
under the condition $q^{1-|x|}|UV|<1$.
Substituting \eqref{eq:BUVproduct} into~\eqref{eq:CUVfromB}, we obtain~\eqref{eq:CUVproduct}.

The next step is to evaluate the generating function ${\mathcal I}(U,V;z)$. Collecting the powers of $\log(-z)$, we write
\be
\label{IUV-def}
{\mathcal I}(U,V;z)=\sum_{i=0}^2{\mathcal I}_i(U,V;z) L^i,
\ee
where $L$ is defined in~\eqref{eq:def-L}.

Let us first evaluate 
${\mathcal I}_2(U,V;z)$.
For any integer $n\in\Z$, from~\eqref{eq:Phik} we can write
an expansion
\be
\Phi_k(-zq^n)=\sum_{i=0}^2 \Phi_{k,i}(-zq^n)L^i,\label{Phikexp}
\ee
where the coefficients are rational functions of $z$. 
Using
\be
\Phi_{k,2}(v)=\frac{h}{1-v},\label{Phik2}
\ee
we obtain
\be
{\mathcal I}_2(U,V;z)= h\sum_{k\geq 0}\sum_{e\in \mathbb{Z}}\sum_{r\geq0}(-zq^k)^r\frac{U^{-e}V^{-k}}
{1+zq^{e-3r}}-h\sum_{k<0}\sum_{e\in \mathbb{Z}}\sum_{r\geq 1}(-zq^k)^{-r}\frac{U^{-e}V^{-k}}
{1+zq^{e+3r}}.\label{twosum}
\ee
We now change the sign of $r$ in the second term in \eqref{twosum}, and then shift $e\mapsto a+ 3r$ in both terms.  We can 
either sum over $k$ with the same result for both $r\geq0$ and $r<0$:
\be
\frac{1}{1-q^rV^{-1} }
\ee
under the condition $1<|V|<q^{-1}$, or sum over $r$ with the same result for both $k\geq0$ and $k<0$:
\be
\frac{1}{1+q^kzU^{-3} }
\ee
under the condition $q<|zU^{-3}|<1$.
As a result,
\be
{\mathcal I}_2(U,V;z)= h\sum_{a\in\mathbb{Z}}\frac{U^{-a}}{1+zq^a}\sum_{r\in\mathbb{Z}}
\frac{(-zU^{-3} )^{r}}{1-q^r V^{-1}}=h\Theta_q(U,V;z),\label{prodexp}
\ee
where
\begin{equation}\label{eq:Theta-product}
\Theta_q(U,V;z)=
K(U^{-1},-z)K( -zU^{-3},V^{-1}),
\end{equation}
and
\begin{equation}
K(a,b):=\sum_{n\in\mathbb Z}\frac{b^n}{1-aq^n}
=
(q;q)_\infty^3
\frac{\theta_q(ab)}{\theta_q(a)\theta_q(b)}
\label{Kronecker}
\end{equation}
is the bilateral Kronecker identity \eqref{ident1a}
rewritten in a symmetric form.
Substituting~\eqref{Kronecker} into~\eqref{eq:Theta-product}, we obtain~\eqref{Thetaprod}.

Coefficients of  ${\mathcal I}(U,V;z)$ at powers $L^1$ and $L^0$ are obtained by differentiating~\eqref{prodexp}. 
Introduce
differential operators
\begin{equation}
D_U:=U\frac{\partial}{\partial U},
\quad
D_V:=V\frac{\partial}{\partial V},
\quad
D_z:=z\frac{\partial}{\partial z}.\label{diffops}
\end{equation}
Since
\begin{equation}
\log(-zq^a)=(L+a)\log q=-h(L+a),
\end{equation}
we have
\begin{equation}
I_0(-zq^a)
=
h\frac{L+a}{1+zq^a},
\end{equation}
where $I_0$ is defined in \eqref{eq:I01}.
This gives the exact generating identity
\begin{align}
\mathcal I_0(U;z)
&:=
\sum_{a\in\mathbb Z}U^{-a}I_0(-zq^a)
=
h(L-D_U)K(U^{-1},-z)\nonumber\\
&=hK(U^{-1},-z)(L-\Psi_q(U^{-1},-z/U)).
\label{eq:I0gen}
\end{align}
From~\eqref{eq:I01} we obtain
\begin{equation}
I_1(y)=y\frac{\rd}{\rd y}I_0(y),
\end{equation}
and for $y=-zq^a$, one has, at fixed $a$,
\begin{equation}
y\frac{\rd}{\rd y}=D_z.
\end{equation}
Hence
\begin{align}
\mathcal I_1(U;z)
&:=
\sum_{a\in\Z}U^{-a}I_1(-zq^a)
=
D_z\mathcal I_0(U;z)\nonumber\\
&=-K(U^{-1},-z)
+h(L-D_U)D_zK(U^{-1},-z)\nonumber\\
&=-K(U^{-1},-z)
+\mathcal I_0(U;z)\Psi_q(-z/U,-z)+hK(U^{-1},-z)\Lambda_q(-z/U)
\label{eq:I1gen}
\end{align}
where we used \eqref{eq:I0gen}, $D_zL=-h^{-1}$, and
\be
 h(L-D_U)D_zK(U^{-1},-z)=
\mathcal I_0(U;z)\Psi_q(-z/U,-z)+hK(U^{-1},-z)\Lambda_q(-z/U).
\ee

The factor $L+k$ in $\Phi_k$, see~\eqref{eq:Phik}, is generated by $V$-differentiation, because
\begin{equation}
D_VV^{-k}=-kV^{-k}.
\end{equation}
Thus, using~\eqref{Kronecker},
\begin{equation}
(L-D_V)K(-zU^{-3},1/V)=K(-zU^{-3},1/V)(L-\Psi_q(V^{-1},-zU^{-3}/V)).
\end{equation}
Therefore, the complete generating function of $I_{e,k}$ is
\begin{align}
\mathcal I(U,V;z)
=
\bigl\{\mathcal I_0(U;z)(L-D_V)
-3\mathcal I_1(U;z)\bigr\}K(-zU^{-3},V^{-1})
=
h\Theta_q(U,V;z)\nonumber\\
\times\bigl\{
\frac{3}{h}+
\bigl(L-\Psi_q\bigl(\frac{1}{U},\frac{-z}{U}\bigr)\bigr)\bigl(L-\Psi_q\bigl(\frac{1}{V},\frac{-z}{U^3V}\bigr)-3\Psi_q\bigl(\frac{-z}{U},-z\bigr)\bigr)
-3\Lambda_q\bigl(\frac{-z}{U}\bigr)
\bigr\}.
\label{eq:Jgenoperator}
\end{align}
Substituting~\eqref{eq:Jgenoperator} into~\eqref{CTMaster}, equation \eqref{eq:J52C} becomes~\eqref{eq:J52unexpanded}.
\end{proof}

\begin{proof}[Proof of Lemma~\ref{lem:analytic-sigma237}] The first part of the lemma follows from Proposition \ref{tauhat-gen}: 
in this case $p=r=1$ and the $q^2$-Borel transform of $\mH_{ {\bf 3}_1^-}(q,z)$ satisfies (\ref{eq:H-uniform-Laplace}) away from the Stokes ray going through $z=1$. The value at $z=1$ is to be understood as a lateral limit. To obtain the additional
 identities, the best way to proceed is to use (\ref{lcb}). With our normalization of the invariant, we must have
\be
\ba
\mL_{-1} \left[ (X+X^{-1}-2) A_{{\bf 3}_1^-,x}(q) \right]&= q \CM(q), \\
   \mL_{-1} \left[ (X+X^{-1}-2) B_{{\bf 3}_1^-,x}(q) \right] &=2 q B_{\Sigma(2,3,7)}(q).
  \ea
  \ee
 The first equality follows from Theorem 1.3 in \cite{MR4233201}. To see the second equality,  first use
\begin{equation}
  (qX,q/X;q)_\infty(q;q)_\infty
  =
  {\theta_q(X)\over 1-X}.
  \label{eq:B-theta-rewrite}
\end{equation}
It follows that
\begin{equation}
  \left(X+X^{-1}-2\right)B_{{\bf 3}^-_1,x}(q)
  =
  {q\over2(q^2;q^2)_\infty^4}
  \left(X^{-2}-X^{-1}\right)
  \theta_q(X)\theta_q(-X)^2.
  \label{eq:B-after-Alexander-factor}
\end{equation}
The elementary product identity
\begin{equation}
  \theta_q(X)\theta_q(-X)
  =
  {(q;q)_\infty^2\over(q^2;q^2)_\infty}
  \theta_{q^2}(X^2)
  \label{eq:theta-product-q-q2}
\end{equation}
therefore gives
\begin{equation}
  \left(X+X^{-1}-2\right)B_{{\bf 3}^-_1,x}(q)
  =
  {q(q;q)_\infty^2\over2(q^2;q^2)_\infty^5}
  \left(X^{-2}-X^{-1}\right)
  \theta_{q^2}(X^2)\theta_q(-X).
  \label{eq:B-q2-theta-form}
\end{equation}
Expanding the two theta functions,
\be
  \theta_{q^2}(X^2)
  =
  \sum_{r\in\IZ}(-1)^r q^{r(r-1)}X^{2r}, \qquad
  \theta_q(-X)
  =
  \sum_{m\in\IZ}q^{m(m-1)/2}X^m,
  \label{eq:theta-expand}
\ee
one can easily check that the r.h.s. of \eqref{eq:B-q2-theta-form} is in the space $\CA_{1,1}$ defined in \eqref{eq:laplace-domain}, therefore one can apply $\mL_{-1}$
to obtain 
\begin{equation}
 q {(q;q)_\infty^2\over2(q^2;q^2)_\infty^5}
  \sum_{r,m\in\IZ}
  (-1)^r q^{r(r-1)+m(m-1)/2}
  \left(
  q^{(2r+m-2)^2}-q^{(2r+m-1)^2}
  \right).
  \label{eq:H-double-theta-sum}
\end{equation}
After putting \(k=2r+m\), the inner \(r\)-sum is a unary
theta function:
\begin{equation}
 q {(q;q)_\infty^2\over2(q^2;q^2)_\infty^5}
  \sum_{k\in\IZ}
  q^{k(k-1)/2}
  \left(q^{(k-2)^2}-q^{(k-1)^2}\right)
  \theta_{q^6}(q^{3-2k}).
  \label{eq:H-one-sum-theta}
\end{equation}
Splitting \(k\) into residue classes modulo \(3\), and using
\begin{equation}
 \gamma_Q(n)z^{n} \theta_Q(Q^n z)
  =
\theta_Q(z),
  \qquad
  \theta_Q(Q/z)=\theta_Q(z),
  \label{eq:theta-quasiperiods}
\end{equation}
one finds (\ref{eq:H-theta-constant-form}). To prove the asymptotic expansion, we find a Mellin--Barnes representation for $\hat \tau_{\Sigma(2,3,7)}(q,z)$. The $q^2$-Borel transformation of (\ref{z-z}) gives:
\be
\ba
\tilde \tau_{\Sigma(2,3,7)}(q,t) &=-2 q^{-1}\sum_{m=0}^{\infty}   { (q^{m+1};q)_{m+1}  \over (q, -q; q)_m} \left( {t \over q} \right)^m=
2 (1-q^{-1}) 
{_2 \phi_1} \Bigl[ 
\begin{matrix}
 q^{3/2}, -q^{3/2}\\0
\end{matrix}
 ;q, {t \over q}\Bigr].
  \label{eq:S-simple}
  \ea
\end{equation}
Using Watson's formula, and following the by now standard steps, we obtain the formula
\be
\hat \tau_{\Sigma(2,3,7)}(q,z)= - {\pi q^{-1}  \over 2 \log(q) (q^2;q^2)^2_\infty}
  \int_{\mathbb R-\ri\epsilon}
  \left(-{z\over q}\right)^{\ri s-1}
  {
  \qpoch{q^{\ri s}}{q}
  \qpoch{q^{2-2\ri s}}{q^2}
  \over
  \sinh^2(\pi s)
  \qpoch{q^{1+2\ri s}}{q^2}
  }
  \,\rd s .
  \label{eq:I-vertical}
\end{equation}
One can now evaluate the formal asymptotic expansion by using the same technique as in Subsection \ref{prop-as}.
Let us also note that this integral can be calculated by closing the contour in the upper half-plane and picking the residues, as we did for the figure-eight knot. There are simple poles at $s=\ri n$, $n\in \IZ_{\ge 0}$, and a second family of poles at
\begin{equation}
  s=
  {\pi\ell\over \log q}
  +\ri\left(m+{1\over 2}\right),
  \qquad
  m=0,1,2,\ldots,\quad \ell\in\mathbb Z.
  \label{eq:half-poles}
\end{equation}
The calculation of these contributions is similar to what happens in the case of the trefoil knot. The case $z=1$ requires a choice of branch of the logarithm, as in (\ref{log-branch}). The contribution of the poles at $s = \ri n$ gives (\ref{A237M}), while (\ref{eq:H-theta-constant-form}) is the contribution of the poles at (\ref{eq:half-poles}).
\end{proof}

\begin{proof}[Proof of Lemma~\ref{lem:analytic-bar-sigma237}] The first part of the lemma follows again from Proposition \ref{tauhat-gen},
 since in this case $p=r=1$ and the $q$-Borel transform of $\mH_{ {\bf 4}_1}(q,z)$ 
 satisfies (\ref{eq:H-uniform-Laplace}) away from Stokes rays. The expressions 
 (\ref{eq:AB-residue-series}) are obtained by
 taking the limit $z \to 1$. The
extra $B_{\overline{\Sigma(2,3,7)}}(q)/h$ term at $z=1$ comes from the terms with $k+n=0$ through the limit
\begin{equation}
 \lim_{z\to1}\frac{\log z}{1-z}=-1.
 \label{eq:exceptional-limit}
\end{equation}
 To derive the closed formula (\ref{Brogers}) for
 $B_{\overline{\Sigma(2,3,7)}}(q)$, we use \eqref{eq:Pk-definition} and put
$k=2j+\ell$.  Since
\begin{equation}
 \frac{1}{(q;q)_k} \qbinom{k}{2j}
 =\frac{1}{(q;q)_{2j}(q;q)_\ell},
 \qquad
 (q;q)_{2j}=(q,q^2;q^2)_j,
 \label{eq:qbinomial-cancellation}
\end{equation}
we obtain
\begin{equation}
 (q;q)_\infty^3B_{\overline{\Sigma(2,3,7)}}(q)
 =2 \sum_{j,\ell\geq0}
 \frac{q^{(2j+\ell)(2j+\ell+1)/2+j}}
 {(q^2;q^2)_j(q;q)_\ell}.
 \label{eq:double-sum}
\end{equation}
Euler's identity
\begin{equation}
 (-z;q)_\infty
 =\sum_{\ell\geq0}
 \frac{q^{\ell(\ell-1)/2}z^\ell}{(q;q)_\ell}
 \label{eq:Euler}
\end{equation}
allows one to evaluate the $\ell$-sum in \eqref{eq:double-sum}:
\begin{equation}
 \sum_{\ell\geq0}
 \frac{q^{\ell(\ell+1)/2+2j\ell}}{(q;q)_\ell}
 =(-q^{2j+1};q)_\infty.
 \label{eq:ell-sum}
\end{equation}
Therefore,
\begin{equation}
 (q;q)_\infty^3  B_{\overline{\Sigma(2,3,7)}}(q)
 =2 \sum_{j\geq0}
 \frac{q^{2j(j+1)}(-q^{2j+1};q)_\infty}
 {(q^2;q^2)_j}.
 \label{eq:after-Euler}
\end{equation}
If we now use
\begin{equation}
 (-q^{2j+1};q)_\infty
 =\frac{(-q;q)_\infty}{(-q;q)_{2j}},
 \qquad
 (-q;q)_\infty=\frac{(q^2;q^2)_\infty}{(q;q)_\infty},
 \label{eq:minus-products}
\end{equation}
we arrive at
\begin{equation}
B_{\overline{\Sigma(2,3,7)}}(q)
 =2 \frac{(q^2;q^2)_\infty}{(q;q)_\infty^4}
 \sum_{j\geq0}
 \frac{q^{2j^2+2j}}
 {(q^2;q^2)_j(-q;q)_{2j}}.
 \label{eq:Rogers-Selberg-reduction}
\end{equation}
The r.h.s. can then be evaluated with the second Rogers--Selberg identity,
\begin{equation}
 \sum_{j\geq0}
 \frac{q^{2j^2+2j}}
 {(q^2;q^2)_j(-q;q)_{2j}}
 =\frac{(q^2,q^5,q^7;q^7)_\infty}{(q^2;q^2)_\infty}.
 \label{eq:second-Rogers-Selberg}
\end{equation}
This proves (\ref{Brogers}).

Finally, to derive the asymptotic expansion, we first obtain a Mellin--Barnes representation. The \(q\)-Borel transform is
\begin{equation}
  \tilde \tau_{\overline{\Sigma(2,3,7)}}(q, t)
  =
 2  \sum_{m\geq0}
  (-1)^m{(q^{m+1};q)_{m+1}\over(q;q)_m}t^m
  =
  2 (1-q)\,
  {_3\phi_2}\Bigl[
  \begin{matrix}
  -q,\ q^{3/2},\ -q^{3/2}\\
  0,\ 0
  \end{matrix}
  ;q,\ -t
  \Bigr].
  \label{eq:extra-f1-borel}
\end{equation}
By using Watson's representation, we eventually find the expression in 
\eqref{eq:extra-f1-vertical-integral}. The standard approach used in Subsection 
\ref{prop-as} establishes the formal asymptotic expansion.
  \end{proof}
  
  \begin{proof}[Proof of Lemma~\ref{lem:analytic-m-1/2}] The first part of the lemma follows once more from Proposition \ref{tauhat-gen}. To establish the asymptotic behavior,
 we express the invariant in Mellin--Barnes form.
The $q$-Borel transform is
\be
\ba
 \tilde \tau_{M_{-1/2}}(q,t)
 &=
  2\sum_{m\geq0}\sum_{\ell=0}^m
  (-1)^m q^{\ell(\ell+1)}
  {(q;q)_{2m+1}\over
  (q;q)_m(q;q)_\ell(q;q)_{m-\ell}}
  t^m\\
  &= 2 \sum_{\ell\geq0}
 \gamma_{q^2}(-\ell) t^\ell
  {(q;q)_{2\ell+1}\over(q;q)_\ell^2}
 {_4\phi_3}
  \Bigl[
  \begin{matrix}
  -q^{\ell+1},\ q^{\ell+3/2},\ -q^{\ell+3/2},\ 0\\
  0,\ 0,\ 0
  \end{matrix}
  ;q,\ -t
  \Bigr].
  \ea
  \label{eq:corr-borel-ell}
  \ee
 By using Watson's representation of the $q$-hypergeometric function and following the usual steps,
 we end up with the Mellin--Barnes integral representation~\eqref{mb-M12} for the $q$-Borel--Laplace resummation.
 %
  %
  From this expression it is then easy to obtain the formal asymptotic expansion as in Subsection \ref{prop-as}.
  \end{proof}
\bibliography{biblio}{}

\newcommand{\etalchar}[1]{$^{#1}$}
\def\cprime{$'$} \def\cprime{$'$}
\providecommand{\bysame}{\leavevmode\hbox to3em{\hrulefill}\thinspace}
\providecommand{\href}[2]{#2}
\providecommand{\eprint}{\begingroup \urlstyle{rm}\Url}
\begin{thebibliography}{GGMnW25}

\bibitem[AK14]{MR3227503}
J{\o}rgen~Ellegaard Andersen and Rinat Kashaev, \emph{A {TQFT} from {Q}uantum
  {T}eichm\"uller {T}heory}, Comm. Math. Phys. \textbf{330} (2014), no.~3,
  887--934.

\bibitem[AKP26]{akp}
Sergei Alexandrov, Albrecht Klemm, and Boris Pioline, \emph{{Large Order
  Enumerative Geometry, Black Holes and Black Rings}}, \eprint{2605.19552}.

\bibitem[Ask80]{MR567718}
Richard Askey, \emph{Ramanujan's extensions of the gamma and beta functions},
  Amer. Math. Monthly \textbf{87} (1980), no.~5, 346--359.

\bibitem[Bal94]{MR1317343}
Werner Balser, \emph{From divergent power series to analytic functions},
  Lecture Notes in Mathematics, vol. 1582, Springer-Verlag, Berlin, 1994,
  Theory and application of multisummable power series.

\bibitem[Bal00]{MR1722871}
\bysame, \emph{Formal power series and linear systems of meromorphic ordinary
  differential equations}, Universitext, Springer-Verlag, New York, 2000.

\bibitem[BBL05]{bbl}
Anna Beliakova, Christian Blanchet, and Thang Le, \emph{{Laplace transform and
  universal sl (2) invariants}}, \eprint{math/0509394}.

\bibitem[BNG96]{MR1389962}
Dror Bar-Natan and Stavros Garoufalidis, \emph{On the
  {M}elvin-{M}orton-{R}ozansky conjecture}, Invent. Math. \textbf{125} (1996),
  no.~1, 103--133.

\bibitem[Bor99]{Borel1899}
{\'E}mile Borel, \emph{M{\'e}moire sur les s{\'e}ries divergentes}, Annales
  Scientifiques de l'{\'E}cole Normale Sup{\'e}rieure \textbf{16} (1899),
  9--131.

\bibitem[CCF{\etalchar{+}}19]{3dmod}
Miranda C.~N. Cheng, Sungbong Chun, Francesca Ferrari, Sergei Gukov, and
  Sarah~M. Harrison, \emph{{3d Modularity}}, JHEP \textbf{10} (2019), 010,
  \eprint{1809.10148}.

\bibitem[Cos09]{MR2474083}
Ovidiu Costin, \emph{Asymptotics and {B}orel summability}, Chapman \& Hall/CRC
  Monographs and Surveys in Pure and Applied Mathematics, vol. 141, CRC Press,
  Boca Raton, FL, 2009.

\bibitem[Dre15]{MR3384489}
Thomas Dreyfus, \emph{Building meromorphic solutions of {$q$}-difference
  equations using a {B}orel-{L}aplace summation}, Int. Math. Res. Not. IMRN
  (2015), no.~15, 6562--6587.

\bibitem[DW01]{MR1822139}
David De~Wit, \emph{An infinite suite of {L}inks-{G}ould invariants}, J. Knot
  Theory Ramifications \textbf{10} (2001), no.~1, 37--62.

\bibitem[FW24]{fantini-wheeler}
Veronica Fantini and Campbell Wheeler, \emph{{Summability for State Integrals
  of hyperbolic knots}}, \eprint{2410.20973}.

\bibitem[Gar04]{Garoufalidis2004}
Stavros Garoufalidis, \emph{On the characteristic and deformation varieties of
  a knot}, Geometry \& Topology Monographs \textbf{7} (2004), 291--309,
  \eprint{math/0306230}, Knots in Hellas '98, Vol. 3 (Delphi).

\bibitem[GGM21]{ggm1}
Stavros Garoufalidis, Jie Gu, and Marcos Mari{\~n}o, \emph{The resurgent
  structure of quantum knot invariants}, Communications in Mathematical Physics
  \textbf{386} (2021), no.~1, 469--493.

\bibitem[GGM23]{peacock}
\bysame, \emph{{Peacock patterns and resurgence in complex Chern--Simons
  theory}}, Research in the Mathematical Sciences \textbf{10} (2023), no.~3,
  29.

\bibitem[GGMnW25]{MR4841756}
Stavros Garoufalidis, Jie Gu, Marcos Mari\~no, and Campbell Wheeler,
  \emph{Resurgence of {C}hern-{S}imons theory at the trivial flat connection},
  Comm. Math. Phys. \textbf{406} (2025), no.~1, Paper No. 20, 60.

\bibitem[GK17]{MR3696603}
Stavros Garoufalidis and Rinat Kashaev, \emph{From state integrals to
  {$q$}-series}, Math. Res. Lett. \textbf{24} (2017), no.~3, 781--801.

\bibitem[GK23]{MR4700600}
\bysame, \emph{The descendant colored {J}ones polynomials}, Pure Appl. Math. Q.
  \textbf{19} (2023), no.~5, 2307--2334.

\bibitem[GL11]{MR2860990}
Stavros Garoufalidis and Thang T.~Q. L\^{e}, \emph{Asymptotics of the colored
  {J}ones function of a knot}, Geom. Topol. \textbf{15} (2011), no.~4,
  2135--2180.

\bibitem[GM21]{MR4233201}
Sergei Gukov and Ciprian Manolescu, \emph{A two-variable series for knot
  complements}, Quantum Topol. \textbf{12} (2021), no.~1, 1--109.

\bibitem[GMnP16]{gmp}
Sergei Gukov, Marcos Mari\~no, and Pavel Putrov, \emph{{Resurgence in complex
  Chern-Simons theory}}, \eprint{1605.07615}.

\bibitem[GPPV20]{gppv}
Sergei Gukov, Du~Pei, Pavel Putrov, and Cumrun Vafa, \emph{{BPS spectra and
  3-manifold invariants}}, J. Knot Theor. Ramifications \textbf{29} (2020),
  no.~02, 2040003, \eprint{1701.06567}.

\bibitem[GPV17]{gpv}
Sergei Gukov, Pavel Putrov, and Cumrun Vafa, \emph{{Fivebranes and 3-manifold
  homology}}, JHEP \textbf{07} (2017), 071, \eprint{1602.05302}.

\bibitem[GR04]{MR2128719}
George Gasper and Mizan Rahman, \emph{Basic hypergeometric series}, second ed.,
  Encyclopedia of Mathematics and its Applications, vol.~96, Cambridge
  University Press, Cambridge, 2004, With a foreword by Richard Askey.

\bibitem[Guk05]{gukov}
Sergei Gukov, \emph{{Three-dimensional quantum gravity, Chern-Simons theory,
  and the A polynomial}}, Commun. Math. Phys. \textbf{255} (2005), 577--627,
  \eprint{hep-th/0306165}.

\bibitem[GW22]{gar-wheel}
Stavros Garoufalidis and Campbell Wheeler, \emph{{Modular $q$-holonomic
  modules}}, \eprint{2203.17029}.

\bibitem[Hab02]{MR2002603}
Kazuo Habiro, \emph{On the quantum {$\rm sl_2$} invariants of knots and
  integral homology spheres}, Invariants of knots and 3-manifolds ({K}yoto,
  2001), Geom. Topol. Monogr., vol.~4, Geom. Topol. Publ., Coventry, 2002,
  pp.~55--68.

\bibitem[Hik05]{hikami1}
Kazuhiro Hikami, \emph{{O}n the quantum invariant for the {B}rieskorn homology
  spheres}, International Journal of Mathematics \textbf{16} (2005), no.~06,
  661--685, \eprint{https://doi.org/10.1142/S0129167X05003004}.

\bibitem[Hik07]{MR2330673}
\bysame, \emph{Generalized volume conjecture and the {$A$}-polynomials: the
  {N}eumann--{Z}agier potential function as a classical limit of the partition
  function}, J. Geom. Phys. \textbf{57} (2007), no.~9, 1895--1940.

\bibitem[Hik11]{hikami2}
\bysame, \emph{{Decomposition of Witten-Reshetikhin-Turaev invariant: Linking
  pairing and modular forms}}, AMS/IP Stud. Adv. Math. \textbf{50} (2011),
  131--151.

\bibitem[HMnMO14]{hmmo}
Yasuyuki Hatsuda, Marcos Mari\~{n}o, Sanefumi Moriyama, and Kazumi Okuyama,
  \emph{Non-perturbative effects and the refined topological string}, J. High
  Energy Phys. (2014), no.~9, 168, front matter+41.

\bibitem[Kas97]{MR1434238}
R.~M. Kashaev, \emph{The hyperbolic volume of knots from the quantum
  dilogarithm}, Lett. Math. Phys. \textbf{39} (1997), no.~3, 269--275.

\bibitem[KM25]{MR4992645}
Rinat Kashaev and Vladimir Mangazeev, \emph{On braided {H}opf structures on
  exterior algebras}, J. Pure Appl. Algebra \textbf{229} (2025), no.~12, Paper
  No. 108135, 30.

\bibitem[KR89]{MR1026957}
A.~N. Kirillov and N.~Yu. Reshetikhin, \emph{Representations of the algebra
  {${U}_q({\rm sl}(2)),\;q$}-orthogonal polynomials and invariants of links},
  Infinite-dimensional {L}ie algebras and groups ({L}uminy-{M}arseille, 1988),
  Adv. Ser. Math. Phys., vol.~7, World Sci. Publ., Teaneck, NJ, 1989,
  pp.~285--339.

\bibitem[LZ99]{lz}
Ruth Lawrence and Don Zagier, \emph{Modular forms and quantum invariants of
  3-manifolds}, Asian Journal of Mathematics \textbf{3} (1999), no.~1, 93--108.

\bibitem[Mas03]{MR1997328}
Gregor Masbaum, \emph{Skein-theoretical derivation of some formulas of
  {H}abiro}, Algebr. Geom. Topol. \textbf{3} (2003), 537--556.

\bibitem[Mci99]{mcintosh}
Richard~J. Mcintosh, \emph{Some asymptotic formulae for q-shifted factorials},
  The Ramanujan Journal \textbf{3} (1999), no.~2, 205--214.

\bibitem[MM95]{MR1328734}
P.~M. Melvin and H.~R. Morton, \emph{The coloured {J}ones function}, Comm.
  Math. Phys. \textbf{169} (1995), no.~3, 501--520.

\bibitem[MM01]{MR1828373}
Hitoshi Murakami and Jun Murakami, \emph{The colored {J}ones polynomials and
  the simplicial volume of a knot}, Acta Math. \textbf{186} (2001), no.~1,
  85--104.

\bibitem[Mn15]{Marino2015}
Marcos Mari\~no, \emph{Spectral theory and mirror symmetry}, arXiv:1506.07757,
  2015.

\bibitem[Mur07]{MR2349300}
Hitoshi Murakami, \emph{The colored {J}ones polynomials and the {A}lexander
  polynomial of the figure-eight knot}, JP J. Geom. Topol. \textbf{7} (2007),
  no.~2, 249--269.

\bibitem[Oht95]{ohtsuki}
Tomotada Ohtsuki, \emph{A polynomial invariant of integral homology 3-spheres},
  Math. Proc. Cambridge Philos. Soc. \textbf{117} (1995), no.~1, 83--112.

\bibitem[Par]{park-inverted}
Sunghyuk Park, \emph{{Inverted state sums, inverted Habiro series, and
  indefinite theta functions}}, \eprint{arXiv:2106.03942}.

\bibitem[PK01]{PK}
R.B. Paris and D.~Kaminski, \emph{{Asymptotics and Mellin--Barnes integrals}},
  Encyclopedia of Mathematics and its Applications, vol.~85, Cambridge
  University Press, 2001.

\bibitem[Roz97]{rozansky}
L.~Rozansky, \emph{{Higher order terms in the Melvin-Morton expansion of the
  colored Jones polynomial}}, Communications in Mathematical Physics
  \textbf{183} (1997), no.~2, 291--306.

\bibitem[RT91]{MR1091619}
N.~Reshetikhin and V.~G. Turaev, \emph{Invariants of {$3$}-manifolds via link
  polynomials and quantum groups}, Invent. Math. \textbf{103} (1991), no.~3,
  547--597.

\bibitem[Svo25]{svoboda}
Josef Svoboda, \emph{{Inverted Habiro Series and its Residues}}, arXiv preprint
  arXiv:2509.22610 (2025).

\bibitem[Wat29]{Watson}
G.~N. Watson, \emph{Theorems {S}tated by {R}amanujan ({VII}): {T}heorems on
  {C}ontinued {F}ractions}, J. London Math. Soc. \textbf{4} (1929), no.~1,
  39--48.

\bibitem[Wei76]{Weil1976}
Andr{\'e} Weil, \emph{{Elliptic functions according to Eisenstein and
  Kronecker}}, Ergebnisse der Mathematik und ihrer Grenzgebiete, vol.~88,
  Springer-Verlag, Berlin, Heidelberg, New York, 1976.

\bibitem[Whe25]{wheeler}
Campbell Wheeler, \emph{{Quantum Modularity for a Closed Hyperbolic
  3-Manifold}}, SIGMA \textbf{21} (2025), 004, \eprint{2308.03265}.

\bibitem[Wil22]{MR4404799}
Sonny Willetts, \emph{A unification of the {ADO} and colored {J}ones
  polynomials of a knot}, Quantum Topol. \textbf{13} (2022), no.~1, 137--181.

\bibitem[Wit89]{MR990772}
Edward Witten, \emph{Quantum field theory and the {J}ones polynomial}, Comm.
  Math. Phys. \textbf{121} (1989), no.~3, 351--399.

\bibitem[Wit11]{MR2809462}
\bysame, \emph{Analytic continuation of {C}hern--{S}imons theory},
  Chern-{S}imons gauge theory: 20 years after, AMS/IP Stud. Adv. Math.,
  vol.~50, Amer. Math. Soc., Providence, RI, 2011, pp.~347--446.

\end{thebibliography}
\bibliographystyle{hamsalpha}
%
%
%
%
%
%
%

\end{document}